\documentclass[twoside,11pt]{article}

\usepackage[color,matrix,arrow]{xy} 
\usepackage{color}
\usepackage{amsmath, amsthm, amssymb, bbm, setspace,bigints}
 \usepackage{mathrsfs}
  \usepackage{bm}
 \usepackage[margin=1 in]{geometry}
\usepackage{caption}
\usepackage{subcaption}
\usepackage[toc,page]{appendix}     
\usepackage{pdfpages}
\usepackage{epstopdf}
\usepackage{booktabs}
\usepackage{mathabx}
\usepackage{enumitem}
\usepackage{tabularx, multirow}
\usepackage[square,sort,comma,numbers]{natbib}
\usepackage{empheq}
\usepackage{algorithmic}
\usepackage{algorithm}
\usepackage{accents}
\usepackage{bm}
\usepackage{comment}
\usepackage{subfiles}

\usepackage{kantlipsum}
\allowdisplaybreaks

\theoremstyle{remark}
\newtheorem{thm}{Theorem}

\newtheorem{cor}[thm]{Corollary}
\newtheorem{prop}[thm]{Proposition}

\newtheorem{rem}[thm]{Remark}

\newtheorem{ex}[thm]{Example}
   
\usepackage{jmlr2e}

\newcommand{\dd}{{\rm d}}
\newcommand{\argmin}{\mathop{\rm argmin~}}

\newcommand{\matnorm}[1]{{\left\vert\kern-0.25ex\left\vert\kern-0.25ex\left\vert #1 \right\vert\kern-0.25ex\right\vert\kern-0.25ex\right\vert}_{\rm op}}
\newcommand{\ri}{(\textrm{i})}
\newcommand{\rii}{(\textrm{ii})}
\newcommand{\riii}{(\textrm{iii})}
\newcommand{\wht}{\widehat}
\newcommand{\wt}{\widetilde}
\newcommand{\mx}{\mbox}
\newcommand{\id}{\textrm{Id}}

\def\ms{\mathscr}
\def\mb{\mathbb}
\def\m{\mathcal}

\DeclareMathOperator{\KL}{KL}
\DeclareMathOperator{\W}{W}

\usepackage{lastpage}
\jmlrheading{25}{2024}{1-\pageref{LastPage}}{7/23; Revised
5/24}{8/24}{23-0889}{Rentian Yao, Xiaohui Chen and Yun Yang}

\ShortHeadings{Wasserstein Proximal Coordinate Gradient Algorithms}{Yao, Chen and Yang}
\firstpageno{1}

\begin{document}

\title{Wasserstein Proximal Coordinate Gradient Algorithms}

\author{\name Rentian Yao \email rentian2@illinois.edu \\
       \addr Department of Statistics\\
       University of Illinois at Urbana--Champaign\\
       Champaign, IL 61820, USA
       \AND
       \name Xiaohui Chen \email xiaohuic@usc.edu \\
       \addr Department of Mathematics\\
       University of Southern California\\
       Los Angeles, CA 90089, USA
       \AND
       \name Yun Yang \email yy84@umd.edu \\
       \addr Department of Mathematics\\
       University of Maryland\\
       College Park, MD 20742, USA}

\editor{Matthew Hoffman}

\maketitle

\begin{abstract}
Motivated by approximation Bayesian computation using mean-field variational approximation and the computation of equilibrium in multi-species systems with cross-interaction, this paper investigates the composite geodesically convex optimization problem over multiple distributions. The objective functional under consideration is composed of a convex potential energy on a product of Wasserstein spaces and a sum of convex self-interaction and internal energies associated with each distribution. To efficiently solve this problem, we introduce the Wasserstein Proximal Coordinate Gradient (WPCG) algorithms with parallel, sequential, and random update schemes. Under a quadratic growth (QG) condition that is weaker than the usual strong convexity requirement on the objective functional, we show that WPCG converges exponentially fast to the unique global optimum. In the absence of the QG condition, WPCG is still demonstrated to converge to the global optimal solution, albeit at a slower polynomial rate. Numerical results for both motivating examples are consistent with our theoretical findings.
\end{abstract}


\section{Introduction}\label{sec: intro}
The task of minimizing a functional over the space of probability distributions is common in statistics and machine learning, with a wide range of applications in nonparametric statistics \citep{kiefer1956consistency, yan2023learning}, Bayesian analysis \citep{trillos2020bayesian, yao2022mean, ghosh2022representations, lambert2022variational}, online learning \citep{guo2022online, ballu2022mirror}, single cell analysis \citep{lavenant2021towards}, and spatial economies \citep{blanchet2016existence}. Since a probability distribution is an infinite-dimensional object with rich geometric structures,  analysis of such an optimization problem requires special treatment and usually leverages optimal transport theory where the convergence is characterized through the Wasserstein metric. In this paper, we consider the problem of jointly minimizing an objective functional over $m$ distributions $\{\rho_j\}_{j=1}^m$, where $\rho_j\in \ms P_2(\m X_j)$, the space of all probability distributions over the $j$-th (Euclidean) domain $\m X_j$ with finite second-order moments. Precisely, we consider the following general optimization problem:
\begin{align}\label{eqn: obj_func}
    \min_{\rho_1,\ldots,\rho_m}\m F(\rho_1,\ldots,\rho_m) := \m V(\rho_1,\dots, \rho_m) + \sum_{j=1}^m \m H_j(\rho_j) + \sum_{j=1}^m \m W_j(\rho_j),
\end{align}
\vspace{-2em}
\begin{subequations}\label{eqn:energy_terms}
\begin{align}
\hspace{-5em}  \mbox{where}\qquad  \m V(\rho_1, \dots, \rho_m) &= \int_{\m X} V(x_1, \dots, x_m)\,\dd\rho_1\cdots\dd\rho_m, \label{eqn: potential_energy}\\
\m H_j(\rho_j) &= \int_{\m X_j} h_j(\rho_j(x_j))\,\dd x_j, \,\,\forall\, j\in[m],\\
\m W_j(\rho_j) &= \int\!\!\!\int_{\m X_j\times\m X_j} W_j(x_j, x_j')\,\dd\rho_j(x_j)\dd\rho_j(x_j').
\end{align}
\end{subequations}
Here we use $\m X = \bigotimes_{j=1}^m \m X_j$ to denote the product space of $\{\m X_j\}_{j=1}^m$. Throughout the paper, the notation for a distribution, such as $\rho$, may stand for both the corresponding probability measure or its density function relative to the Lebesgue measure.
In this formulation, $\m V$ can be interpreted as an \emph{interaction} potential energy functional with potential function $V:\,\m X\to\mb R$, describing the interactions among the $m$ input distributions, $\m H_j$ is the individual \emph{internal} energy functional associated with $\rho_j$ for some convex function $h_j:\,[0,\infty)\to \mathbb R$, and $\m W_j$ is the \emph{self-interaction} functional associated with $\rho_j$ for some self-interaction potential function $W_j:\m X_j\times\m X_j\to\mathbb R$. By convention, we define $\m H_j(\rho_j) = \infty$ if $\rho_j$ is not absolutely continuous with respect to the Lebesgue measure and $h_j$ is not a constant function. When $h_j \equiv 0$ and $W_j \equiv 0$, this problem degenerates into the problem of minimizing an $m$-variate function $V$ on the Euclidean space; when $h_j(x) = x\log x$, $W_j = 0$, and $m=1$, this optimization problem amounts to minimize the Kullback-Leibler (KL) divergence $\KL(\cdot\,\,\|\,\rho^\ast_1)$ with $\rho^\ast_1 \,\propto\, e^{-V}$. In general, an  internal energy functional $\m H_j$ with $h_j$ satisfying $h_j(x)\to\infty$ as $x \to\infty$ prevents the optimal solution of $\rho_j$ from degenerating into a point mass measure.  Our problem of minimizing the functional in~\eqref{eqn: obj_func} is chiefly motivated by the two representative examples below.


\begin{ex}[{\bf Mean-field inference in variational Bayes}]
    \label{ex:MFVI}
    In approximate Bayesian computation (ABC) when optimizing a functional over a single high-dimensional  distribution to approximate the posterior distribution, it can often be computationally and theoretically advantageous to employ a mean-field approximation by breaking into the product of many lower-dimensional ones, provided that the bias introduced by this approximation is tolerable or can be suitably controlled. Specifically, consider a Bayesian model with a likelihood function $p(x\,|\,\theta)$ and a prior density function $\pi$ (relative to the Lebesgue measure) over the parameter space $\Theta=\bigotimes_{j=1}^m\Theta_j$, where the parameter $\theta=(\theta_1,\ldots,\theta_m)$ is divided into $m$ pre-specified blocks with $\theta_j\in\Theta_j$. Given observed i.i.d.~data $X^n = (X_1, \dots, X_n)$, the posterior density function of $\theta$ can be expressed via Bayes' rule as
\begin{align}\label{eqn: posterior_density}
    \pi_n(\theta) \coloneqq p(\theta\,|\,X^n) = \frac{\pi(\theta)\prod_{i=1}^n p(X_i\,|\,\theta)}{\int_{\Theta}\pi(\theta)\prod_{i=1}^n p(X_i\,|\,\theta)\,\dd\theta}.
\end{align}
We use $\Pi_n$ to denote the corresponding posterior distribution. 
In practice, $\pi_n$ is often computationally intractable due to the lack of an explicit form of the normalizing constant, i.e., the denominator in~\eqref{eqn: posterior_density}. Mean-field variational inference (MFVI) \citep{bishop2006pattern}, which approaches this task by turning the integration problem into an optimization problem, can be formulated as finding a closest fully factorized distribution $\wht\pi_n=\bigotimes_{j=1}^m
\wht \rho_j$ to approximate the target posterior distribution $\Pi_n$ with respect to the KL divergence; that is, computing
\begin{align*}
    \wht\pi_n = \argmin_{\rho = \bigotimes_{j=1}^m\rho_j}\KL(\rho\,\|\,\Pi_n),\quad\mx{s.t.}\quad \rho_j\in\ms P(\Theta_j)\quad\forall\, j\in[m].
\end{align*}
Up to a $\rho$ independent constant, this is equivalent to
\begin{align}\label{eqn: MFVI}
    \wht\pi_n &= \argmin_{\rho=\bigotimes_{j=1}^m \rho_j} \bigg\{\int_\Theta \big[nU_n(\theta) - \log\pi(\theta)\big]\,\dd\rho_1\cdots\dd\rho_m + \sum_{j=1}^m \int_{\Theta_j} \rho_j\log\rho_j\bigg\},
\end{align}
where $U_n(\theta) = -\frac{1}{n}\sum_{i=1}^n\log p(X_i\,|\,\theta)$ corresponds to the (averaged) negative log-likelihood function of data $X^n$. It is then apparent that above is a special case of the general formulation~\eqref{eqn: obj_func} by taking $V = nU_n - \log\pi$, $h_j(x) = x\log x$, $W_j = 0$, and $\m X_j = \Theta_j$. The optimal solution $(\wht\rho_1, \dots, \wht\rho_m)$ corresponds to the mean-field approximation of the joint posterior distribution of the parameter $\theta = (\theta_1, \dots, \theta_m)$ via the relationship $\wht\pi_n=\bigotimes_{j=1}^m
\wht \rho_j$. 
\citet{ghosh2022representations} shows the connection between MFVI~\eqref{eqn: MFVI} and the objective functional~\eqref{eqn: obj_func} without proposing practical algorithm for solving~\eqref{eqn: MFVI}. \citet{yao2022mean} focuses on solving special cases of problem~\eqref{eqn: MFVI} with $m=2$ blocks corresponding to continuous model parameters and discrete latent variables.

Concurrent and independent to our work,~\cite{arnese2024convergence} and~\cite{lavenant2024convergence} have also investigated the minimization of~\eqref{eqn: MFVI} with coordinate ascent variational inference (CAVI) using optimal transport theory. \cite{arnese2024convergence} explored the convergence of CAVI with parallel and sequential update schemes, while~\cite{lavenant2024convergence} studied the random update scheme. Both papers primarily emphasized the theoretical analysis of the CAVI algorithm, demonstrating similar convergence results to ours for CAVI when the posterior distributions are log-concave.
In contrast, our work introduces a novel algorithm that differs from CAVI and studies its algorithmic convergence rate. Furthermore, our analysis is conducted under the quadratic growth condition (detailed in Assumption D), which is weaker than the log-concavity assumption as later shown in Proposition~\ref{prop: suff_cond_QG}.
\qed
\end{ex}



\begin{ex}[{\bf Equilibrium in multi-species systems with cross-interaction}]
    \label{ex:multi-species}
    Multi-species systems \citep{carrillo2018zoology, daus2022random} arise in applications in cell biology \citep{pinar2021reaction} and population dynamics \citep{zamponi2017analysis}. In this example, we consider the following non-local multi-species cross-interaction model with diffusion,
\begin{align}\label{eqn:aggregation_model}
    \partial_t \rho_j = \nabla\cdot\Big(\rho_j\, \Big[\,\nabla V_j - \sum_{i=1}^m (\nabla K_{ij})\ast\rho_i + \nabla h_j'(\rho_j)\,\Big]\Big),\ \ j\in[m],
\end{align}
where $\rho_j(x,t)$ is the unknown mass density of species $j$ at location $x$ and time $t$, $V_j$ is the external potential field affecting species $j$, $K_{jj}$ is the self-interaction potential of species $j$, and $\{K_{ij}:\,1\leq i\neq j\leq m\}$ are the cross-interaction potentials.
When $\m H_j$ is the negative self-entropy functional $h_j(\rho_j) = \rho_j\log\rho_j$ for the $j$-th species, PDE~\eqref{eqn:aggregation_model} corresponds to the Fokker--Planck equation associated with the mean-field multi-species stochastic interacting particle system \citep{daus2022random},
\begin{align}\label{eqn: multi-species SDE}
\dd X_j(t) = -\nabla V_j\big(X_j(t)\big)\,\dd t + \sum_{i=1}^m\big(\nabla K_{ij}\ast\rho_i(\cdot, t)\big)\big(X_j(t)\big)\,\dd t + \sqrt{2}\,\dd B_j(t),\quad \forall\, j\in[m].
\end{align}
In particular, $\rho_j(\cdot, t)$ corresponds to the density function associated with the distribution of $X_j(t)$ with initialization $X_j(0)\sim\rho_j(\cdot, 0)$ for all $j\in[m]$. Another common class of $\m H_j$ has $h_j(\rho_j)=\rho_j^{m_j}$ with $m_j>1$, corresponding to diffusion in porous media~\citep{aronson2006porous,vazquez2007porous}.
For such an entropy functional $\m H_j$, there is no analogous stochastic differential equation (SDE) corresponding to equation~\eqref{eqn:aggregation_model}.  For symmetric multi-species models where $K_{ij} = K_{ji}$, finding the equilibrium of equation~\eqref{eqn:aggregation_model} is equivalent to solving the optimization problem~\eqref{eqn: obj_func} with $V(x) = \sum_{i=1}^m V_i(x_i) - \sum_{1\leq i<j\leq m} K_{ij}(x_i-x_j)$ and $W_j(x_j, x_j') = -K_{jj}(x_j-x_j') / 2$ (cf. Proposition \ref{lem: multi_species_system}). The symmetric interaction kernel assumption is natural in physics applications due to Newton's third law of motion.

The demand of computing the stationary distribution (or equilibrium) of a multi-species system naturally arises in physical science, such as chemical engineering~\citep{carrayrou2002new, paz2013computing}. However, most existing literature considers computational methods for calculating the equilibrium of interacting particle systems with only one species. For example, \cite{gutleb2022computing} consider the case where $m=1$, $V = h = 0$, and the interaction potential $W$ has attractive-repulsive power-law form, by approximating the equilibrium measure with a series of orthogonal polynomials. Under their approximating schemes, the original problem of finding the equilibrium turns into an optimization problem of solving the coefficient of each polynomial. 
In the multi-species setting, \cite{owolabi2019computational} studied the equilibrium of multi-species fractional reaction-diffusion systems by simply approximating the spatial derivatives with second-order Taylor expansion without analyzing the approximation error of their methods.
\qed
\end{ex}


\subsection{Our Contributions}
In this paper, we propose the \emph{Wasserstein Proximal Coordinate Gradient} (WPCG) algorithm as a tool to minimize a composite geodesically convex functional of form~\eqref{eqn: obj_func} over \emph{multiple} distributions, which extends the coordinate descent algorithm from the Euclidean space to the space of probability distributions. We provide detailed convergence analysis for WPCG with three updating schemes: \emph{parallel}, \emph{sequential}, and \emph{random}. Specifically, we show that: (i) 
under the condition that the potential function $V$ is smooth and convex, and $\{h_j\}_{j=1}^m$ and $\{W_j\}_{j=1}^m$ are convex, WPCG converges to a global optimal solution of~\eqref{eqn: obj_func} at rate $O(1/k)$, where $k$ denotes the iteration count; (ii) under the additional quadratic growth (QG) condition in~\eqref{eqn: QG},
WPCG converges to the unique global optimum exponentially fast in the Wasserstein-2 metric. Note that QG condition is a weaker requirement than the strong convexity on $V$ and/or $\{h_j\}_{j=1}^m$ and $\{W_j\}_{j=1}^m$; see Figure~\ref{fig:ass_diagram} for the implications of various assumptions for proving convergence in optimization and Section~\ref{sec: WPCG} for further details.
In addition, implementation of WPCG with the parallel updating scheme inherently supports parallelization, which enhances its computational scalability to high-dimensional problems.

Previous studies, including \citep{ambrosio2005gradient, salim2020wasserstein, wibisono2018sampling}, predominantly concentrate on using Wasserstein gradient flow for the minimization of a functional over a single distribution. To the best of our knowledge, the current work is among the first study to: 
\begin{enumerate}[topsep=0pt,itemsep=-1ex,partopsep=1ex,parsep=1ex]
    \item investigate the application of Wasserstein gradient flow for minimizing a functional over multiple distributions;
    \item extend the design and analysis of coordinate descent type algorithms for optimization in Euclidean spaces to those for optimization in Wasserstein spaces. 
\end{enumerate}
\vspace{0.5em}

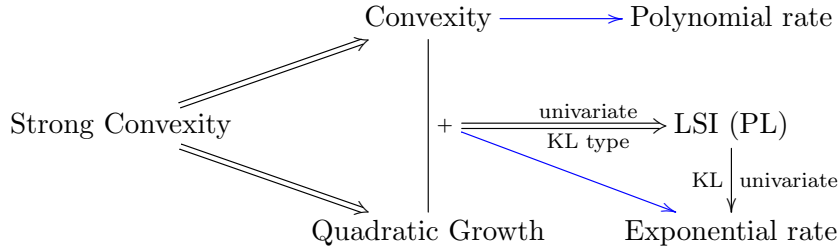
\begin{figure}[ht]
\begin{align*}
\xymatrix{
    & \text{Convexity} \ar@[blue][r] & \text{Polynomial rate} \\
    \text{Strong Convexity} \ar@{=>}[ru] \ar@{=>}[rd]  &   & \text{LSI (PL)}\ar@{>}[d]^{\text{univariate}}_{\text{KL}}\\
    & \text{Quadratic Growth} & \text{Exponential rate}
    \ar @{-} "1,2" ; "3,2"^+="a"
    \ar @{=>} "a" ;"2,3"^{\text{univariate}}_{\text{KL type}}
    \ar @[blue] "a" ;"3,3"
\restore}
\end{align*}
\caption{Implication diagram of commonly used assumptions for proving convergence results in optimization. Double arrows show the connections between assumptions on the Wasserstein space: strong convexity implies convexity and the quadratic growth condition; convexity and the quadratic growth condition together imply LSI for KL-type functionals. Single arrows show the convergence rates implied by different assumptions, among which the blue ones are studied in this work for minimizing multivariate objective functionals.}
\label{fig:ass_diagram}
\end{figure}

We emphasize two key differences from optimization on the Euclidean space.

First, most common internal energy functionals $\{\m H_j\}_{j=1}^m$ in~\eqref{eqn:energy_terms} (e.g., negative self-entropy) are {\bf non-smooth} in the usual sense---the difference incurred by the linearization of $\m H_j$ around any $\rho$ cannot be upper bounded by a multiple of the Wasserstein distance to $\rho$. For coordinate descent optimization on the Euclidean space, such a smoothness property serves as two critical purposes: (i) it ensures that the discretization error of gradient descent is compensated by the contraction of the corresponding continuous-time gradient flow; (ii) it prevents overshooting of the coordinate descent update by controlling the misalignment between the coordinate-wise steepest descent direction and the entire steepest descent direction, so that coordinate descent algorithm keeps making progress in minimizing the objective functional. In the Wasserstein space, (explicit) gradient/coordinate ascent algorithms with non-smooth $\m H_j$ require the linearization of $\m H_j$ in the algorithm and a uniform control of smoothness of the Wasserstein gradient of $\m H_j$. Our analysis reveals that considering an implicit scheme (or a proximal type update) in designing coordinate descent type algorithms in Wasserstein spaces can avoid imposing a smoothness assumption on $\{\m H_j\}_{j=1}^m$ when analyzing their convergence. In particular, the linear convergence of WPCG only requires the potential gradient $\nabla V$ to be Lipschitz (cf. Assumption B in Section~\ref{sec: WPCG}). Thus, convergence rates of WPCG are derived under more realistic conditions on the objective function $\m F$ and hold for general functionals $\{\m H_j\}_{j=1}^m$ than the negative self-entropy. This allows us to compute the stationary measure of non-local multi-species degenerate cross diffusion beyond heat diffusion (cf. more details for our numeric example on porous medium equation in Section~\ref{subsec:multi-species_systems}).

A second challenge in demonstrating the convergence of WPCG lies in the difficulty of proving a Polyak--{\L}ojasiewicz (PL) type inequality, which requires some positive power of certain norm of the ``Wasserstein gradient" of target functional $\m F$ at any $\rho_{1:m} := (\rho_1,\ldots,\rho_m)$ to be lower bounded by a multiple of $\m F(\rho_{1:m}) - \m F(\rho^*_{1:m})$, where $\rho^*_{1:m}$ denotes a global minimizer of $\m F$. It is known~\citep{karimi2016linear} that proving a PL inequality is crucial and common for analyzing gradient-based optimization algorithms in the Euclidean case.
A PL inequality is typically implied by strong convexity or a log-Sobolev inequality (LSI). It is also known that a quadratic growth (QG) condition with convexity is strictly weaker than the strong convexity (see discussions after Assumption D). While the QG condition together with convexity implies LSI for the KL-type functional over a single distribution, there is no corresponding result for a general functional over multiple distributions $\rho_{1:m}$. 
The main technical difficulty comes from the lack of tensorization for a ``multivariate" Wasserstein gradient.

\subsection{Related Work} 
{\bf Optimization over one distribution.} The seminal paper by Jordan, Kinderlehrer, and Otto~\citep{jordan1998variational} introduces an iterative scheme (now commonly known as the \emph{JKO scheme}) serving as the time-discretization of a continuous dynamic in the space of probability distributions following the direction of the steepest descent of a target functional with respect to the Wasserstein-$2$ metric. 
In later developments, the JKO scheme has been recognized as a promising numerical method for optimizing a functional over a probability distribution and can be viewed as a special proximal gradient update~\citep{rockafellar1997convex} relative to the Wasserstein-$2$ metric in the space of probability distributions (see Section~\ref{sec: convex_functional} for more details). Therefore, in the remainder of the paper, we will primarily use the term Wasserstein proximal gradient scheme instead of the JKO scheme, as this terminology more clearly suggests its close link with general convex optimization.


Recently, several other discretization methods are proposed for minimizing a functional over a single distribution (i.e.,~$m=1$ in our setting). Different from the Wasserstein proximal gradient scheme that can be viewed as an implicit (backward) scheme for discretizing the Wasserstein gradient flow (WGF), some explicit (forward) schemes, analogous to the usual gradient descent on the Euclidean space, are considered and analyzed for solving specific problems.
For example,~\cite{chewi2020gradient} analyze a gradient descent algorithm for computing the barycenter on the Bures--Wasserstein manifold of centered Gaussian probability measures. One key ingredient in their convergence analysis is to prove a quadratic growth condition for the barycenter functional that leads to a PL inequality and implies the exponential convergence of the algorithm.
For the KL divergence functional $\KL(\,\cdot\,\|\,\rho^\ast)$ with $\rho^\ast\,\propto \,e^{-V}$ for some strongly convex and smooth potential function $V$,~\cite{wibisono2018sampling} shows that a symmetrized Langevin algorithm, which is an implementable discretization of the Langevin dynamics,
reduces the bias in the classical unadjusted Langevin algorithm~\citep{durmus2017nonasymptotic} and attains exponential convergence up to a step size dependent discretization error. 
\cite{salim2020wasserstein} consider a more general setting where the objective functional consists of the same potential functional and a general non-smooth term that is convex along generalized geodesics on the Wasserstein space; they propose a forward-backward algorithm, which runs a forward step (gradient descent) for the potential functional and a backward step (proximal gradient) for the non-smooth term, and prove its exponential convergence under the same conditions on potential $V$.


\noindent {\bf Optimization over special cases of multiple distributions.} Several recent works focus on solving an important special case of problem~\eqref{eqn: obj_func} in the context of MFVI as we described in Example~\ref{ex:MFVI}, where a multi-dimensional Bayesian posterior distribution is approximated by the product of several lower-dimensional distributions relative to the KL divergence. \citet{yao2022mean} consider a block MFVI for Bayesian latent variable models with two blocks, one for the discrete latent variables and one for the continuous model parameters, which corresponds to problem~\eqref{eqn: obj_func} with $m=2$. They propose and analyze a majorization-minimization algorithm for solving MFVI, which can be viewed as a distributional extension of the classical expectation–maximization (EM) algorithm. Due to the special property that minimizing the latent variable block in the problem admits a closed-form expression, their algorithm is effectively a time-discretized WGF for optimizing a single distribution, with an effective potential function $V$ changing over the iterations. When the population-level log-likelihood function is locally strictly concave, they show that their algorithm converges exponentially fast to the solution of MFVI. \citet{ghosh2022representations} study a WGF-based algorithm for solving MFVI without latent variables. Their study shows that the discretized flow (Wasserstein proximal gradient scheme) converges to a mathematically well-defined continuous flow when the step size is small, assuming certain conditions on the potential function $V$. However, their analysis is asymptotic and they do not directly examine the convergence of the time-discretized algorithm, leaving it unclear if the discrete dynamic system is stable so that the discretization error does not accumulate over time and whether an explicit convergence rate can be obtained. 

The rest of this paper is organized as follows. In Section \ref{sec: preliminary}, we present an overview of essential concepts in optimal transport that are necessary for developing our methods and theory. Section \ref{sec: WPCG} introduces the WPCG algorithm with three different update schemes: parallel, sequential, and random update schemes. We also propose two different numerical methods to solve the Wasserstein proximal gradient scheme, a key step for implementing the WPCG algorithms. Theoretical analyses of the WPCG algorithms are provided in Section~\ref{sec:WPCG-conv} in the presence and absence of a $\lambda$-quadratic growth condition.  In Section~\ref{sec:apps}, we demonstrate the applications of our algorithm and theory to mean-field variational inference and multi-species systems, along with numerical experiments. In Section~\ref{sec: discussion}, we conclude the paper and discuss some open problems for potential future research.

\subsection{Notations}
For any $\m X\subset\mb R^d$, let $\ms P(\m X)$ be the collection of all probability measures on $\m X$, $\ms P_2(\m X) = \{\mu: \mb E_{X\sim\mu}[\|X\|^2] < \infty\}$ be the set of probability measures with finite second-order moments, and $\ms P_2^r(\m X) \subset \ms P_2(\m X)$ be the subset that includes all absolutely continuous probability measures with respect to the Lebesgue measure on $\m X$. For a measurable map $T:\m X\to\m X$, let $T_\#: \ms P(\m X)\to\ms P(\m X)$ be its corresponding pushforward operator, i.e., $\nu = T_\#\mu$ if and only if $\nu(A) = \mu(T^{-1}(A))$ for any measurable set $A\subset\m X$. Throughout the paper, we assume $\m X$ to be convex and compact unless otherwise stated.

Let $\pi^0(x, y) = x$ and $\pi^1(x, y) = y$ be two projection maps (projecting into the first and second components). Let $\Pi(\mu, \nu) = \{\gamma: (\pi^0)_\#\gamma = \mu, (\pi^1)_\#\gamma = \nu\}$ be the set of couplings\footnote{A coupling between two distributions $\mu$ and $\nu$ is a joint distribution whose two marginals are $\mu$ and $\nu$.}, or all transport plans, between $\mu$ and $\nu$. Let $\Pi_o(\mu, \nu)$ be the set of all optimal transport plans, whose precise definition can be found in Section \ref{sec: preliminary} below. For a random variable $X$, we use $\m L(X)$ to denote its distribution.

Let $C^1(\m X)$ be the set of functions on $\m X$ that have continuous derivatives. Let $L^1(\m X)$ and $L^\infty(\m X)$ represent the class of integrable functions and uniformly bounded functions on $\m X$, respectively. For a vector $x = (x_1, \dots, x_m)$, we assume its sub-vector without the $j$-th entry is denoted by the shorthand $x_{-j} = (x_1, \dots, x_{j-1}, x_{j+1}, \dots, x_m)$. Let $\|\cdot\|$ denote the vector $\ell_2$ norm. For a matrix $A$, let $\matnorm{A} = \sup_{\|x\|=1}\|Ax\|$ denote its matrix $\ell_2 \to \ell_2$ operator norm.

\section{Preliminaries on Optimal Transport}\label{sec: preliminary}
In this section, we briefly review some concepts and results in optimal transport that are necessary for explaining and developing our methods and analysis.

\subsection{Wasserstein Space and Geodesics}
For any $\mu, \nu\in\ms P_2(\m X)$, the Wasserstein-2 distance between $\mu$ and $\nu$ is defined as 
\begin{align}\label{eqn: W2_dist_def}
    \W_2^2(\mu, \nu) = \inf_{\gamma\in\Pi(\mu, \nu)} \bigg\{\int_{\m X\times\m X} \|x - y\|^2\,\dd\pi(x, y)\bigg\}\, ,
\end{align}
It is known that $(\ms P_2(\m X), \W_2)$ is a metric space called Wasserstein space. Moreover, if $\m X$ is complete and separable, $(\ms P_2(\m X), \W_2)$ is complete as well \citep{bolley2008separability}. 
The infimum in (\ref{eqn: W2_dist_def}) always admits a solution $\gamma$ called an \emph{optimal transport plan} \citep{santambrogio2015optimal}.
When $\mu\in\ms P_2^r(\m X)$, the optimal transport plan is unique and takes form $\gamma = (\id, T_{\mu}^\nu)_\#\mu$, implying $\nu = (T_\mu^\nu)_\#\mu$, where $\id$ is the identity map and $T_\mu^\nu$ is called the \emph{optimal transport map} from $\mu$ to $\nu$~\citep{santambrogio2015optimal}. 
Moreover, Brenier's Theorem \citep{brenier1987decomposition} states that this unique optimal transport map $T_\mu^\nu = \nabla\psi$ is the gradient of a convex function $\psi$.

Note that $(\ms P_2(\m X), \W_2)$ is not a flat metric space, but is positively curved in the Alexandrov sense \citep{ambrosio2005gradient}. Geodesics on a curved space are the shortest paths connecting two points.
$(\ms P_2(\m X), \W_2)$ is indeed a geodesic space. For any $\mu_0, \mu_1\in\ms P_2^r(\m X)$ and $\gamma\in\Pi_o(\mu_0, \mu_1)$, the constant-speed geodesic connecting $\mu_0$ and $\mu_1$ is $\mu_t = (\pi_t)_\#\gamma$, where $\pi_t = (1-t)\pi^0 + t\pi^1$.

\subsection{Convex Functionals on Wasserstein Space}\label{sec: convex_functional}
Convexity is an important concept in optimization. On the Euclidean space, a function $f:\m X\to\mb R$ is \emph{$\lambda$-strongly convex} if
\begin{align*}
    f(x_t) \leq (1-t)f(x_0) + tf(x_1) - \frac{t(1-t)\lambda}{2}\|x_1-x_0\|^2,
\end{align*}
and we say $f$ is \emph{$L$-smooth} if
\begin{align*}
    f(x_t) \geq (1-t)f(x_0) + tf(x_1) - \frac{t(1-t)L}{2}\|x_1-x_0\|^2,
\end{align*}
for all $t\in[0, 1]$ where $x_t = (1-t)x_0 + tx_1$. In convex optimization, $L$-smoothness and $\lambda$-strong convexity are common sufficient conditions to guarantee the exponential convergence rate of an (explicit) gradient descent algorithm, $x^{k+1}_{\text{grad}} = x^k_{\text{grad}} - \tau\nabla f(x^k_{\text{grad}})$, $k=0,1,\ldots$, towards the unique global minimum of $f$. In comparison,
the following proximal gradient algorithm
\begin{align}\label{eqn: proximal_grad}
x^{k+1}_{\text{prox}} = \argmin_{x\in\m X} \Big\{f(x) + \frac{1}{2\tau}\|x-x^k_{\text{prox}}\|^2\Big\}\, , \quad k=0,1,\ldots,
\end{align}
only requires the $\lambda$-strong convexity to guarantee its exponential convergence rate~\citep{beck2017first}. Note that a proximal gradient algorithm is also called an \emph{implicit} gradient descent algorithm, since the first-order optimality condition (FOC) for~\eqref{eqn: proximal_grad} reads $x^{k+1}_{\text{prox}} = x^k_{\text{prox}} - \tau\nabla f(x^{k+1}_{\text{prox}})$. The theoretical advantage of the proximal scheme comes at the price of the additional computational cost of numerically solving this FOC equation.

Similar to these notions in the Euclidean case, a functional $\m F:\ms P_2(\m X)\to (-\infty, +\infty]$ to be \emph{$\lambda$-strongly convex along geodesics}, if for all $\mu_0,\mu_1\in\ms P_2(\m X)$, we have
\begin{align}
\begin{aligned}
    \m F(\mu_t) &\leq (1-t)\m F(\mu_0) + t\m F(\mu_1) - \frac{\lambda}{2}t(1-t)\W_2^2(\mu_0, \mu_1),
\end{aligned}
\end{align}
where $\{\mu_t: 0\leq t\leq 1\}$ is the constant speed geodesic curve connecting $\mu_0$ and $\mu_1$.
When $\lambda=0$, we simply say that $\m F$ is \emph{geodesically convex}. It can be shown that both functionals $\m W_j(\rho_j)$ and $\m H_j(\rho_j)$ in our problem formulation~\eqref{eqn: MFVI} are geodesically convex when their defining functions $W_j$ and $h_i$ are convex \citep[Chapter 9.3]{ambrosio2005gradient}. However, different from the Euclidean case where all convex functions are continuous, geodesically convex functionals may not be continuous with respect to the $\W_2$ distance. 

\begin{ex}[Geodesically convex functional may not be continuous]
    Consider the negative self-entropy functional $\rho\mapsto \int_{\m X}\rho\log\rho$ which is geodesically convex. Let $\rho_\infty$ be the uniform distribution on $[0, 1]$ and $\rho_n$ be the uniform distribution on $A_n$ where $A_n = \bigcup_{i=0}^{n-1}[\frac{2i}{2n}, \frac{2i+1}{2n}]$. Then $\W_2(\rho_n, \rho_\infty) \to 0$ as $n \to \infty$, while $\int\rho_n\log\rho_n = \log 2$ for all $n$ and $\int\rho_\infty\log\rho_\infty = 0$.
    \qed
\end{ex}


As we mentioned in the introduction, the Wasserstein proximal gradient scheme, an iterative algorithm defined as below, is commonly used for minimizing a functional $\m F: \ms P_2(\m X)\to (-\infty, +\infty]$ defined on the Wasserstein space, 
\begin{align}\label{eqn: JKO}
    \rho^{k+1} = \argmin_{\rho\in\ms P_2(\m X)} \bigg\{\m F(\rho) + \frac{1}{2\tau}\W_2^2(\rho, \rho^{k})\bigg\}\, ,\quad k=0, 1, \ldots,
\end{align}
with the initial distribution $\rho^0$ and the step size $\tau > 0$. This is analogous to the proximal gradient algorithm \eqref{eqn: proximal_grad}, i.e., the backward time-discretization of a gradient flow on the Euclidean space. The convergence of Wasserstein proximal gradient scheme for the KL divergence functional $\m F_{\KL}(\rho) = \int V\,\dd\rho + \int\rho\log\rho$ is well studied in literature. As on the Euclidean space, piecewise interpolation $\rho_t = \rho^{\lfloor \frac{t}{\tau}\rfloor}$ converges to a continuous dynamics on Wasserstein space, i.e., the solution of Fokker--Planck equation $\partial_t\rho_t = \nabla\cdot(\rho_t\nabla(V+\log\rho_t))$ \citep{santambrogio2015optimal}. Since the Fokker--Planck equation is also the density evolution equation of Langevin dynamics \citep{jordan1998variational}
\begin{align}\label{eqn: Langevin}
    \dd X_t = -\nabla V(X_t)\,\dd t + \sqrt{2}\,\dd W_t,
\end{align}
one can also use the Euler–Maruyama method to numerically approximate the Wasserstein proximal gradient scheme for the KL using particle approximation.
It is shown in the literature that when $V$ is strongly convex, the solution $\{\rho_t:\,t\geq 0\}$ of the Fokker--Planck equation converges to the unique global minimizer $\rho^\ast  \,\propto\, e^{-V}$ of $\m F_{\rm{KL}}$ \citep{villani2021topics}. 
Corollary 2.8 in \citep{yao2022mean} further shows that the time-discretization of the continuous time dynamic via Wasserstein proximal gradient scheme $\{\rho^k: k\in\mb Z_+\}$ is unbiased, and also converges to $\rho^\ast\propto e^{-V}$ exponentially fast when $V$ is strongly convex.

In this paper, we are primarily interested in functional $\m F: \ms P_2(\m X_1)\times \cdots \times \ms P_2(\m X_m)\to (-\infty,\infty]$ that is defined over $m$ distributions. Analogously, we say that $\m F$ is \emph{(blockwise) $\lambda$-strongly convex} if 
\begin{align*}
\m F(\mu_1^t, \dots, \mu_m^t) \leq \m (1-t)\m F(\mu_1^0, \dots, \mu_m^0) + t\m F(\mu_1^1, \dots, \mu_m^1) - \frac{\lambda}{2}t(1-t)\sum_{j=1}^m\W_2^2(\mu_j^0, \mu_j^1)
\end{align*}
for all $\mu_j^0, \mu_j^1\in\ms P_2(\m X_j)$ and the corresponding constant speed geodesics $\{\mu_j^t: 0\leq t\leq 1\}$, $\forall j\in[m]$. In the next section, we will extend the aforementioned Wasserstein proximal gradient scheme from minimizing a simple functional of one distribution into minimizing a functional of multiple distributions by describing several coordinate ascent versions of JKO and analyze their convergence.

\subsection{First Variation and Wasserstein Gradient}
Let $\m F:\ms P_2(\m X) \to (-\infty, +\infty]$ be a lower semi-continuous functional. For a measure $\mu\in\ms P_2^r(\m X)$, we call $\mu$ to be \emph{regular} for $\m F$ if $\m F((1-\varepsilon)\mu + \varepsilon\nu) < \infty$ for every $\varepsilon\in[0, 1]$ and every absolutely continuous probability measure $\nu\in\ms P(\m X)\bigcap L^\infty(\m X)$. If $\mu$ is regular for $\m F$, the \emph{first variation} (or Gateaux derivative) of $\m F$ at $\mu$ is a map $\frac{\delta\m F}{\delta\rho}(\mu): \m X\to \mathbb{R}$ satisfying
\begin{align*}
    \frac{\dd}{\dd\varepsilon}\m F(\mu+\varepsilon\chi)\bigg|_{\varepsilon=0} = \int_{\m X}\frac{\delta\m F}{\delta\rho}(\mu)\,\dd\chi
\end{align*}
for any perturbation $\chi=\tilde\mu - \mu$ such that $\tilde\mu\in\ms P_2^r(\m X)$ and $\int \, \dd\chi = 0$. Clearly, $\frac{\delta\m F}{\delta\rho}(\mu)$ is uniquely defined up to additive constant. This is analogous to the gradient of $\m F$ in the $L^2(\m X)$ sense. If $\m F$ is a geodesically convex functional defined on the Wasserstein space, and $\m F(\mu)<\infty$ at some $\mu$, then the vector field $\xi = \nabla \frac{\delta\m F}{\delta\rho}(\mu)\in L^2(\mu; \m X)$ will satisfy~\citep{ambrosio2005gradient}
\begin{align}\label{eqn: convexity}
    \m F(\nu) \geq \m F(\mu) + \int_{\m X}\big\langle \xi, T_\mu^\nu - \id\big\rangle\,\dd\mu, \,\,\forall\,\nu\in\ms P_2^r(\m X).
\end{align}
We say that $\xi=\nabla\frac{\delta\m F}{\delta\rho}(\mu)$ is a \emph{strong subdifferential} (or Fr\'{e}chet derivative) of $\m F$ at $\mu$. Strong subdifferential captures the ``gradient'' with respect to the $\W_2$ metric and is more convenient to use than the first variation when analyzing the convergence of a gradient flow on the Wasserstein space. Henceforth, we will also refer to a strong subdifferential of $\m F$ as its Wasserstein gradient.

\section{Wasserstein Proximal Coordinate Gradient (WPCG) Algorithms}\label{sec: WPCG}

In this section, we introduce the WPCG algorithm with three different yet common update schemes: parallel update (WPCG-P), sequential update (WPCG-S), and random update (WPCG-R).

\subsection{Parallel Update}\label{subsec:WPCG-P}

In the WPCG-P algorithm, each coordinate is updated synchronously, i.e.,~we solve the following Wasserstein proximal gradient scheme (\ref{eqn: subproblem-P}) for all $j\in[m]$ at the same time to get $\rho_j^k$,
\begin{align}\label{eqn: subproblem-P}
    \rho_j^{k+1} = \argmin_{\rho_j\in\ms P_2^r(\m X_j)} \Big\{\, \m V(\rho_j, \rho_{-j}^{k}) + \m H_j(\rho_j) + \m W_j(\rho_j) + \frac{1}{2\tau}\W_2^2(\rho_j, \rho_j^{k}) \Big\}\, .
\end{align}
Recall that we have used the shorthand $\rho_{-j}^{k} = (\rho_1^k, \dots, \rho_{j-1}^k, \rho_{j+1}^k, \dots, \rho_m^k)$ to denote the collection of all distributions at iteration $k$ except for the $j$-th one, for each $k\in\mb N$ and $j\in[m]$. When $\m H_j = \m W_j = 0$, the solution $\rho_j^{k+1}$ of the subproblem~\eqref{eqn: subproblem-P} satisfies $(\id + \nabla V_j^k)_\#\rho_j^{k+1} = \rho_j^k$, where $V_j^k = \int V(x_j, x_{-j})\,\dd\rho_{-j}^k$ is a function on $\m X_j$. Furthermore, when the initialization $\rho_j^0$ is a point mass for all $j\in[m]$, the subproblem degenerates to the coordinate proximal descent method for updating the $j$-th block. The parallel update scheme can be parallelly computed and therefore is preferred when the number of coordinates $m$ is large. However, as we shall see in Section~\ref{sec:WPCG-conv}, this scheme is more sensitive to the step size compared with the other two (i.e., sequential and random) update schemes since it may diverge when the step size $\tau$ is large. Pseudo-code of WPCG-P is shown in Algorithm~\ref{algo: WPCG-P}.

\begin{algorithm}[ht]
   \caption{WPCG-P}
   \label{algo: WPCG-P}
\begin{algorithmic}
   \STATE Initialize distribution $\rho^0 = (\rho_1^0, \dots, \rho_m^0)$ arbitrarily, number of iterations $T$, and step size $\tau$.
   \FOR{$k=0$ {\bfseries to} $T-1$}
   \FOR{$j=1$ {\bfseries to} $m$}
     \STATE $\rho_{j}^{k+1} = \argmin_{\rho_{j}\in\ms P_2^r(\m X_{j})} \m V(\rho_j, \rho_{-j}^k) + \m H_{j}(\rho_{j}) + \m W_j(\rho_j) + \frac{1}{2\tau}\W_2^2(\rho_{j}, \rho_{j}^{k})$;
     \ENDFOR
   \ENDFOR
\end{algorithmic}
\end{algorithm}

\subsection{Sequential Update}\label{subsec:WPCG-S}

In the WPCG-S algorithm, the updates are made sequentially through all coordinates, one by one, at each iteration. Although WPCG-S cannot be made parallel, the functional value is always convergent regardless of the step size $\tau$ magnitude since WPCG-S is always a descent algorithm. To ease the presentation of the result, we use the notation $\rho_{i:j}^k = (\rho_i^k, \dots, \rho_j^k)$ to denote the distributions at iteration $k$ from index $i$ to $j$, with the convention that $\rho_{i:j}^k =\emptyset$ if $i > j$. Pseudo-code of WPCG-S is shown in Algorithm~\ref{algo: WPCG-S}.

\begin{algorithm}[ht]
   \caption{WPCG-S}
   \label{algo: WPCG-S}
\begin{algorithmic}
   \STATE Initialize distribution $\rho^0 = (\rho_1^0, \dots, \rho_m^0)$ arbitrarily, number of iterations $T$, and step size $\tau$.
   \FOR{$k=0$ {\bfseries to} $T-1$}
     \FOR{$j=1$ {\bfseries to} $m$}
     \STATE $\rho_{j}^{k+1} = \argmin_{\rho_{j}\in\ms P_2^r(\m X_{j})} \m V(\rho_{1:(j-1)}^{k+1}, \rho_{j}, \rho_{(j+1): m}^{k}) + \m H_{j}(\rho_{j}) + \m W_j(\rho_j) + \frac{1}{2\tau}\W_2^2(\rho_{j}, \rho_{j}^{k})$;
     \ENDFOR
   \ENDFOR
\end{algorithmic}
\end{algorithm}

\subsection{Random Update}\label{subsec:WPCG-R}

In the WPCG-R algorithm, we sample the index $j_l$ of the coordinate to be updated randomly and uniformly from $[m]$, independently of the previous selections of indices. To make the convergence rate comparable to the other two schemes, we consider one iteration of WPCG-R as the process of updating $M$ randomly selected coordinates, where the batch size $M$ is of the order $O(m\log m)$. This ensures that with high probability, each coordinate has been updated at least once per iteration. Similar to WPCG-S, WPCG-R is also a descent algorithm regardless of the choice of the step size $\tau$. Pseudo-code of WPCG-R is shown in Algorithm~\ref{algo: WPCG-R}.

\begin{algorithm}[ht]
   \caption{WPCG-R}
   \label{algo: WPCG-R}
\begin{algorithmic}
   \STATE Initialize distribution $\rho^0 = (\rho_1^0, \dots, \rho_m^0)$ arbitrarily, number of iterations $T$, step size $\tau$, and the number $M$ of coordinates (or batch size) updated in each iteration.
   \FOR{$k=0$ {\bfseries to} $T-1$}
   \STATE $\rho^{k, 0} = \rho^k$;
   \FOR{$l=0$ {\bfseries to} $M-1$}
   \STATE Choose index $j_l \sim \text{unif}([m])$;
   \STATE $\rho_{j_l}^{k, l+1} = \argmin_{\rho_{j_l}\in\ms P_2^r(\m X_{j_l})} \m V(\rho_{j_l}, \rho_{-j_l}^{k, l}) + \m H_{j_l}(\rho_{j_l}) + \m W_{j_l}(\rho_{j_l}) + \frac{1}{2\tau}\W_2^2(\rho_{j_l}, \rho_{j_l}^{k, l})$;
   \STATE $\rho_{-j_l}^{k, l+1} = \rho_{-j_l}^{k, l}$;
   \ENDFOR
   \STATE $\rho^{k+1} = \rho^{k, M}$;
   \ENDFOR
\end{algorithmic}
\end{algorithm}

\subsection{Implementation}
\label{subsec: numerical}

Note that a common key step in the WPCG algorithm is to solve the Wasserstein proximal gradient scheme~\eqref{eqn: JKO} with $\m F(\rho) = \m V(\rho) + \m H(\rho) + \m W(\rho)$, which does not have an explicit solution. Here we introduce two numerical methods: function approximation (FA) approach and particle approximation via SDE.
 

\subsubsection{Particle Approximation via SDE}
The first method is particle approximation via SDE, which is only applicable when $\m H(\rho)$ is the negative self-entropy functional. Recall that the Wasserstein proximal gradient scheme is an implicit scheme for discretizing the WGF. When $\m V(\rho) = \int V\,\dd\rho$, $\m H(\rho) = \int\rho\log\rho$, and $\m W(\rho) = \int\!\int W(x, x')\,\dd\rho(x)\dd\rho(x')$, the WGF of $\m F$ starting from $\rho_0$ is the evolution of the distribution of the following SDE,
\begin{align}\label{eqn: MFSDE}
\dd X_t &= -\nabla V(X_t)\,\dd t - \Big(\int\nabla_1 W(X_t, x) + \nabla_2 W(x, X_t)\,\dd\rho_t(x)\Big)\,\dd t + \sqrt{2}\,\dd W_t, \quad X_t\sim\rho_t,
\end{align}
where $\nabla_1$ and $\nabla_2$ are the gradients with respect to the first and the second variates of $W$. This connection between SDE and WGF motivates us to discretize the WGF by discretizing its corresponding SDE. When the step size (for discretization) is small, these two different discretization schemes are expected to be close. Then, we can approximate the Wasserstein proximal gradient scheme by the evolution of the discretized SDE through the empirical measure of particles which satisfy the following updating formula
\begin{align*}
X_b^{k+1} - X_b^k = -\Big(\nabla V(X^k_b) + \frac{1}{B}\sum_{b'=1}^B \big[\nabla_1W(X_b^k, X_{b'}^k) + \nabla_2 W(X_{b'}^k, X_b^k)\big]\Big)\tau + \sqrt{2\tau}\eta_b^k,
\end{align*}
where $\tau$ is the step size, $B$ is the number of particles, and $\eta_b^k\stackrel{\rm{i.i.d.}}{\sim}\m N(0, 1)$ for all $k\in\mb Z_+$ and $b\in[B]$. The last recursive equation is the discretized representation of the dynamics \eqref{eqn: MFSDE} for approximating the Wasserstein proximal gradient scheme \eqref{eqn: JKO}, and $\rho^{k+1}$ can be approximated by the empirical distribution of $\{X_b^{k+1}: b\in[B]\}$.


\subsubsection{Function Approximation Method}

The SDE approach is no longer applicable for a more general internal energy functional $\m H$ beyond the negative self-entropy. Instead, we leverage the function approximation methods converting Wasserstein proximal gradient scheme~\eqref{eqn: JKO} into an optimization problem over the function space. 
Note that finding $\rho^{k+1}$ that minimizes~\eqref{eqn: JKO} is equivalent to finding a transport map $T$ such that $T_\#\rho^k$ minimizes~\eqref{eqn: JKO}. Precisely, we have the following statement.
\begin{prop}\label{prop: FA}
If $\rho^k\in\ms P_2^r$ and 
\begin{align}\label{eqn: FA_objfunc}
T^{k+1} = \argmin_{T} \m V(T_\#\rho^k) + \m H(T_\#\rho^k) + \m W(T_\#\rho^k) + \frac{1}{2\tau}\int \|T(x) - x\|^2\,\dd\rho^k,
\end{align}
then $\rho^{k+1} = T^{k+1}_\#\rho^k$ minimizes~\eqref{eqn: JKO}.
\end{prop}
We highlight the optimization problem~\eqref{eqn: FA_objfunc} is \emph{unconstrained}. By Brenier's Theorem, the optimal transport map from $\rho^k$ to $\rho^{k+1}$ is the gradient of a convex function. \citet{mokrov2021large} consider the constrained optimization problem by restricting $T = \nabla\psi$ to the gradient of a convex function $\psi$ and finding the optimal convex function $\psi^{k+1}$ by using input-convex neural networks (ICNN)~\citep{pmlr-v70-amos17b}. On the contrary, Proposition~\ref{prop: FA} shows that an unconstrained optimization problem is enough; directly minimizing the objective functional~\eqref{eqn: FA_objfunc} yields the optimal transport map $T^{k+1}$ from $\rho^k$ to $\rho^{k+1}$. The intuition is that, when $T^{k+1}$ is not the optimal transport map, changing $T^{k+1}$ to the optimal transport map $\widetilde T^{k+1}$ from $\rho^k$ to $\rho^{k+1}$ does not change the first three functional values in~\eqref{eqn: FA_objfunc} but decreases the last term, which contradicts to the optimality of $T^{k+1}$.

In practice, the optimization problem in Proposition~\ref{prop: FA} over the function space can be solved by function approximation methods such as kernel methods and neural networks. Moreover, since we do not have access to the time-evolving density $\rho^k$, we still need to approximate the objective functional in Proposition~\ref{prop: FA}. Suppose $\{X_b^k: b\in[B]\}$ are samples drawn from $\rho^k$. We provide two concrete examples, which will be used later in Section~\ref{sec:apps}, to show how to approximate the objective functional in Proposition~\ref{prop: FA}. The details will be postponed to Appendix~\ref{app: FA_detail}.

\begin{ex}[Approximating negative self-entropy]\label{ex: negative_entropy}
For $\m H(\rho^k) = \int\rho^k\log\rho^k$, we can consider the optimization problem
\begin{align*}
T^{k+1} &= \argmin_{T} \frac{1}{B}\sum_{b=1}^B V(T(X_b^k)) - \frac{1}{B}\sum_{b=1}^B \log\,\lvert\det\nabla T(X_b^k)\rvert\\
&\qquad\qquad\qquad + \frac{1}{B^2}\sum_{b,b'=1}^B W\big(T(X_b^k), T(X_{b'}^k)\big) + \frac{1}{2B\tau}\sum_{b=1}^B \|T(X_b^k) - X_b^k\|^2.
\end{align*}
\qed
\end{ex}

\begin{ex}[Approximating porous medium type diffusion]\label{ex: porous_medium}
For $\m H(\rho^k) = \int(\rho^k)^n$, we can consider the following optimization problem
\begin{align}\label{eqn: FA_generalh}
\begin{aligned}
T^{k+1} &= \argmin_{T} \frac{1}{B}\sum_{b=1}^B V(T(X_b^k)) + \frac{1}{B}\sum_{b=1}^B \big[\wht\rho_{\rm kde}^k(X_b^k)\cdot\lvert\det\nabla T(X_b^k)\rvert^{-1}\big]^{n-1}\\
&\qquad\qquad\qquad + \frac{1}{B^2}\sum_{b,b'=1}^B W\big(T(X_b^k), T(X_{b'}^k)\big) + \frac{1}{2B\tau}\sum_{b=1}^B \|T(X_b^k) - X_b^k\|^2,
\end{aligned}
\end{align}
where $\wht\rho_{\rm kde}^k$ is the kernel density estimation of $\rho^k$ by $\{X_b^k: b\in[B]\}$.
\qed
\end{ex}

\begin{rem}[Comparison between FA and SDE approaches]
    Compared with the SDE approach, the FA approach can be applied when the diffusion term $\m H_j$ is beyond the negative self-entropy. According to our numerical results in Section~\ref{sec:apps}, the performance of the FA approach will not be affected when the system contains super-quadratic drift terms. 
    However, smaller step size has to be chosen to apply the SDE approach, since $\tau$ needs to be smaller than $O(L_{\rm c}^{-1})$, where $L_{\rm c} = \max_{j\in[m], x\in\m X}|\!|\!|\nabla_j^2 V(x)|\!|\!|_{\rm op}$) is defined in the discussion after Assumption B (Section~\ref{sec:WPCG-conv}), so that the explicit discretization of the SDE converges. It then takes more iterations for the SDE approach to converge due to this requirement on a smaller step size. An additional attractive aspect of the FA approach is its unbiasedness for any step size $\tau$. This is due to the fact that any fixed-point solution (from a distributional perspective) to its iterative formula must correspond to a global minimum of $\mathcal F$. In comparison, the SDE approach often exhibits bias---its fixed point usually deviates by $O(\sqrt{\tau})$ from a global minimum of $\mathcal F$.
    Although choosing a decreasing step size sequence when discretizing the SDE~\eqref{eqn: MFSDE} can alleviate the constant bias, it does so at the expense of convergence speed. This is evidenced by the slow convergence observed in Figure~\ref{fig: optim_err_superquad}, which persists even for the largest permissible step size due to the existence of a super-quadratic error term.
    \qed
\end{rem}

\section{Convergence Analysis of WPCG Algorithms}\label{sec:WPCG-conv}

In this section, we derive the convergence rates of the WPCG (-P, -S, -R) algorithms. First, we shall make the following assumptions and discuss their implications. These assumptions are assumed to hold throughout the rest of paper unless otherwise specified.

\paragraph{Assumption A (Internal energy).} All $h_1, h_2, \dots, h_m \in C^1(\mb R_+)$ are convex and satisfy:
\begin{enumerate}[topsep=1pt,itemsep=-1ex,partopsep=1ex,parsep=2ex]
\item For some $\alpha > \frac{d_j}{d_j+2}$, it holds
\begin{align*}
    h_j(0) = 0, \quad \liminf_{x_j\downarrow 0_+}\frac{h_j(x_j)}{x_j^\alpha} > -\infty;
\end{align*}
\item the map $x_j\mapsto x_j^{d_j}h_j(x_j^{-d_j})$ is convex and non-increasing in $\mb R_+$;
\item $h_j(\rho_j)\in L^1(\m X_j)$ implies $\rho_j h_j'(\rho_j) \in L^1(\m X_j)$.
\end{enumerate}
\vspace{0.5em}

The first condition on $h_j$ implies that the negative part $h_j(\rho_j)$ is integrable. Details are referred to Remark 3.9 in \citep{ambrosio2007gradient}. The second condition implies the convexity of $h_j$ along geodesics \citep{ambrosio2007gradient, santambrogio2015optimal}. The third condition guarantees that the first variation of $\m H_j$ at $\rho_j$ admits the explicit form as $h_j'(\rho_j)$ (Lemma \ref{lem: ot_map} in Appendix). For example, $h_j(x_j) = x_j\log x_j$ (corresponding to negative self-entropy) and $h_j(x_j) = x_j^{m_j}$ with $m_j > 1$ (corresponding to porous medium type diffusion) satisfy all these three conditions.

\paragraph{Assumption B (Potential energy).} There exists $L > 0$ such that 
\begin{align*}
\big\|\nabla_j V(x_j, x_{-j}) - \nabla_j V(x_j, x_{-j}')\big\| \leq L\|x_{-j} - x_{-j}'\|,\quad\forall\, j\in[m].
\end{align*}\vspace{0.5em}

This Lipschitz constant is different from the coordinate Lipschitz constant $L_{\rm c}$ used by \cite{wright2015coordinate} in $\|\nabla_j{V}(y_j, x_{-j}) - \nabla_j{V}(z_j, x_{-j})\| \leq L_{\rm c}\|y_j - z_j\|$ for all $j\in[m]$. A finite $L_{\rm c}$ guarantees the smoothness of $V$ on each coordinate, which helps control the change of the function value of $V$ after each iteration through the gradient $\nabla V$.
\cite{wright2015coordinate} also defined the restricted Lipschitz constant $L_{\rm r}$ as $\|\nabla{V}(y_j, x_{-j}) - \nabla{V}(z_j, x_{-j})\| \leq L_{\rm r}\|y_j - z_j\|$ to study the convergence property of asynchronous coordinate GD. We also let $L_{\rm g}$ be the global Lipschitz constant defined through $\|\nabla {V}(x) - \nabla{V}(y)\| \leq L_{\rm g}\|x-y\|$, which is commonly used in literature of GD. In fact, we have $0 < L/L_{\rm c} \leq \sqrt{m-1}$, $0\leq L/L_{\rm r}\leq \sqrt{1 - 1/m}$, and $0\leq L/L_{\rm g} \leq 1$. The proof of these connections is provided in Appendix~\ref{app: connection}.

From the above connections, we can see that our Lipschitz assumption is the weakest, which only requires a control on the off-diagonal term of $\nabla^2V$ (when $V$ is twice differentiable). 
This difference is due to our choice of the proximal-type optimization algorithm, where only the magnitude of interaction (off-diagonal) components in $V$ matters in the sense that changing the target function from $V(x)+\sum_{j=1}^m f_j(x_j)$ into $V$ for any univariate functions $\{f_j\}_{j=1}^m$ does not affect the convergence of a proximal coordinate descent algorithm. Technically, this irrelevance of marginal (diagonal) components is due to the reason that convergence of a proximal gradient algorithm for minimizing a univariate function does not require the smoothness (gradient) condition. However, the interaction component, which incurs the overshooting when alternating the coordinate-wise steepest descent direction, cannot be eliminated by using the proximal scheme.

\paragraph{Assumption C (Self-interaction energy).} There are functions $f_j$ and $g_j$ such that both
\begin{align*}
\widetilde{V}(x) \coloneqq V(x) - \sum_{j=1}^m\big[f_j(x_j) + g_j(x_j)\big] \quad\!\mx{and}\quad\! \widetilde{W}_j(x_j, x_j') \coloneqq f_j(x_j) + W_j(x_j, x_j') + g_j(x_j')
\end{align*}
are convex for all $j\in[m]$.

\vspace{0.5em}
  When $W_j$ is symmetric, we can simply take $f_j = g_j$ for all $j\in[m]$. Without loss of generality, in the remaining of this paper, we will only consider the case where $f_j = g_j = 0$ for all $j\in[m]$. For the general case, note that the functional value $\m F(\rho_{1:m})$ in~\eqref{eqn: obj_func} remains the same while replacing $V$ and $W_j$ with $\widetilde V$ and $\widetilde W_j$, respectively. Thus our proof with $f_j = g_j=0$ can be easily extended to the general case. This freedom of allocating the marginal components between the potential energy and the self-interaction energy in our analysis is again due to the fact that the convergence of the proximal coordinate algorithm is only affected by the interaction components within $V$ (cf. discussions after Assumption B).


\paragraph{Assumption D (Overall growth).}\label{assump: QG} The target functional $\m F$ satisfies the following $\lambda$-quadratic growth ($\lambda$-QG) condition: for $\lambda \geq 0$,
\begin{align}\label{eqn: QG}
    \m F(\rho_{1:m}) - \m F(\rho_{1:m}^\ast) \geq \frac{\lambda}{2}\sum_{j=1}^m\W_2^2(\rho_j, \rho_j^\ast), \quad\forall\,\rho_j\in\ms P_2^r(\m X_j).
\end{align}\vspace{0.5em}

 When $\lambda > 0$, the $\lambda$-QG condition guarantees the uniqueness of the solution to the optimization problem \eqref{eqn: obj_func}. When $\lambda = 0$, the minimizer of \eqref{eqn: obj_func} may not be unique, but the functional value convergences to $\m F(\rho^\ast_{1:m})$ for any global minimizer $\rho^\ast_{1:m}$ of~$\eqref{eqn: obj_func}$ (cf.~Theorem \ref{thm: convex_case}).
On the Euclidean space, it is known that $\lambda$-strong convexity is stronger than $\lambda$-QG condition plus convexity, which together implies a PL inequality \citep{karimi2016linear}. A similar result holds in the Wasserstein setting. In fact, the following proposition shows that Assumption D holds when $V$ is $\lambda$-strongly convex, which is a sufficient condition for $\m F$ being strongly convex. Its proof is provided in Appendix~\ref{appendix: lemmas}.
\begin{prop}\label{prop: suff_cond_QG}
Assumption D holds when $V$ is $\lambda$-strongly convex. 
\end{prop}

On the other hand, we also want to mention that the $\lambda$-QG in Assumption D is strictly weaker than $\lambda$-strong convexity. A simple illustrative example is $W_j \equiv 0$, $h_j(\rho_j) = \rho_j\log\rho_j$, and $V(x) = V_1(x_1) + \cdots + V_m(x_m)$ with 
$V_j(x_j) = \frac{\lambda}{2}\,|x_j|$ for $x_j\in[-1,1]$ and $V_j(x_j) = \frac{\lambda}{2}\,x_j^2$ elsewhere. In this case, $\m F(\rho_{1:m}) = \m F_1(\rho_1) + \cdots + \m F_m(\rho_m)$ with
\begin{align*}
    \m F_j(\rho_j) = \int_{\m X_j} V_j\,\dd\rho_j + \int \rho_j\log\rho_j, \quad\m X_j\subset\mb R.
\end{align*}
It is straightforward to verify that $\m F_j$ satisfies the $\lambda$-QG condition due to the growth rate of $V_j$ (see p.280 in~\citep{villani2021topics} for more details) and is convex along geodesics, but is not strongly convex along geodesics since $V$ is not strongly convex in $[-1, 1]^m$ (e.g. consider two absolutely continuous measures supported on $[-1, 1]^m$).

The PL inequality has an analogy on the Wasserstein space which is called log-Sobolev inequality (LSI) \citep{villani2021topics}. However, it is not easy to directly prove an LSI, mainly due to the following two reasons. Firstly, LSI is often exclusively used to study the KL divergence type objective functionals associated to the relative entropy functional, while we are considering more general objective functionals. Secondly, when we consider a multivariate functional on the Wasserstein space, the multivariate Wasserstein gradient does not tensorize, i.e.,~the sum of the norm of its blockwise Wasserstein gradient does not equal to the norm of the Wasserstein gradient of the functional treated as a univariate functional (see ahead~\eqref{eqn: tensorization_issue} in Remark~\ref{rem:WPCF-R_compare} below).


\subsection{Convergence Rates of WPCG under $\lambda$-QG Assumption}

Now, we are ready to present our main theoretical results on the convergence rates of the WPCG algorithms. For notation simplicity, we denote $\rho^* := \rho^*_{1:m}$ as a solution of the optimization problem~\eqref{eqn: obj_func} and $\rho^k := \rho^k_{1:m}$ as the solution of the WPCG algorithms (-P, -S, -R) at the $k$-th iteration when the context has no ambiguity.

\begin{thm}[{\bf Exponential convergence rate of WPCG-P under $\lambda$-QG}]
\label{thm: main_thm}
Assume Assumptions A, B, C, and D hold for some $\lambda > 0$. Let $\rho^k$ be the solution of WPCG-P in~\eqref{eqn: subproblem-P} at the $k$-th iteration. If the step size satisfies $0 < \tau < \frac{m-1/2}{L(m-1)^{3/2}}$, then
\begin{align*}
    \W_2^2(\rho^k, \rho^\ast ) \leq \frac{2}{\lambda}\big[1+C_1(m, L, \tau)\lambda\big]^{-k} \big[\m F(\rho^0) - \m F(\rho^\ast )\big], \quad\forall\, k\in\mb Z_+,
\end{align*}
where
\begin{align*}
    C_1(m, L, \tau) = \frac{\frac{1}{2\tau} + (m-1)\big(\frac{1}{\tau} - L\sqrt{m-1}\big)}{2m[2L^2(m-1) + \frac{2}{\tau^2}]} > 0.
\end{align*}
\end{thm}

\begin{rem}[{\bf Comments on the step size and iteration complexity of WPCG-P}]
\label{rem: parallel_strongly_convex}
(i) By taking $\tau^{-1} = KL\sqrt{m-1}$ for $K>1$, we have $C_1(m, L, \tau) \geq \frac{(K-1)\sqrt{m-1}}{4L(K^2+1)m}$. Therefore, the best iteration complexity based on our theory for achieving $\varepsilon$-accuracy is 
\begin{align*}
    \frac{\log\frac{2[\m F(\rho^0) - \m F(\rho^\ast)]}{\lambda\varepsilon}}{\log(1 + \lambda C_1(m, L, \tau))} \leq \frac{\log\frac{2[\m F(\rho^0) - \m F(\rho^\ast)]}{\lambda\varepsilon}}{\log(1 + \frac{\lambda(K-1)\sqrt{m-1}}{4L(K^2+1)m})} \leq \frac{4L(K^2+1)m}{\lambda(K-1)\sqrt{m-1}}\cdot\frac{\log\frac{2[\m F(\rho^0) - \m F(\rho^\ast)]}{\lambda\varepsilon}}{1 - \frac{\lambda(K-1)\sqrt{m-1}}{8L(K^2+1)m}},
\end{align*}
i.e., WPCG-P requires at most $O(\frac{\sqrt{m}L}{\lambda} \log({1 \over \lambda \varepsilon}))$ iterations to achieve $\varepsilon$-accuracy when the step size $\tau$ is on the scale $O(\frac{1}{L\sqrt{m}})$. [In the sequel, we will drop the logarithmic factor in the iteration complexity when exponential convergence is achieved.]
(ii) Both the upper bound of the step size $\tau$ and the smoothness condition of $V$ are necessary for the convergence of the parallel update scheme. The following Example~\ref{eg: parallel_CD} shows that the upper bound $O(\frac{1}{L\sqrt{m}})$ of the step size $\tau$ cannot be relaxed, i.e.,~there exists $V$, $\m H_j$, and $\m W_j$ such that WPCG-P diverges when $\tau$ exceeds the scale $O(\frac{1}{L\sqrt{m}})$. Moreover, when we choose $\tau$ on this scale, the iteration complexity is exactly $O(\frac{\sqrt{m}L}{\lambda})$ which coincides with the result in the previous discussion.
\qed
\end{rem}

\begin{ex}[{\bf Optimality of step size and convergence rate of WPCG-P}]\label{eg: parallel_CD}
Consider the case where $V(x) = \frac{1-\alpha}{2}\|x\|^2 + \frac{\alpha}{2}(x_1 + \cdots + x_m)^2$ and $W_j = h_j = 0$ so that $\rho_j^\ast =\delta_0$, the point mass measure at $0$ for each $j\in[m]$. In this case, when we initialize $\rho_j^0$ to a point mass measure for all $j\in[m]$, the optimization problem~\eqref{eqn: obj_func} on the Wasserstein space degenerates into an optimization problem on $\mb R^m$ with objective function $V$. The (gradient) Lipschitz constant of $V$ in Assumption B is $L = \alpha\sqrt{m-1}$. In each iteration, the update scheme in WPCG-P algorithm is equivalent to solve
\begin{align*}
    x_j^{k+1} &= \argmin_{x_j\in\mb R} \bigg\{V(x_j, x_{-j}^k) + \frac{(x_j - x_j^k)^2}{2\tau}\bigg\}\\
    &= \argmin_{x_j\in\mb R} \bigg\{\frac{1-\alpha}{2}x_j^2 + \frac{\alpha}{2}(x_j + s_{-j}^k)^2 + \frac{(x_j-x_j^k)^2}{2\tau}\bigg\}\,,
\end{align*}
where we use the shorthand $s^k = x_1^k + \cdots + x_m^k$ to denote the sum of all coordinates at the $k$-th iterate, and $s_{-j}^k = s^k - x_j^k$. A direct calculation for this example reveals that
\begin{align*}
    s^k = \Big(\frac{\tau^{-1} - (m-1)\alpha}{1+\tau^{-1}}\Big)^k s^0.
\end{align*}
Therefore, $s^k$ only converges as $k\to\infty$ when $|\tau^{-1} - (m-1)\alpha| \leq |1+\tau^{-1}|$, which is equivalent to 
\begin{align*}
    \tau \leq \frac{2}{(m-1)\alpha - 1} = O\Big(\frac{1}{L\sqrt{m}}\Big).
\end{align*}
This calculation shows that our upper bound requirement on step size $\tau$ is necessary and tight.
To study the tightness of the convergence rate implied by the theorem, we note that a direct calculation leads to
\begin{align*}
\begin{pmatrix}
x_1^{k+1}\\ \vdots\\ x_m^{k+1}
\end{pmatrix}
= \frac{1}{1+\tau}
\begin{pmatrix}
1 & -\alpha\tau & \cdots & -\alpha\tau\\
-\alpha\tau & 1 & \cdots & -\alpha\tau\\
\vdots & \vdots & \ddots & \vdots\\
-\alpha\tau & -\alpha\tau & \cdots & 1
\end{pmatrix}
\begin{pmatrix}
x_1^k\\ \vdots\\ x_m^k
\end{pmatrix}
\eqqcolon A_{\rm p}x^k.
\end{align*}
It is easy to verify that the $m$ eigenvalues of the symmetric matrix $A_{\rm p}$ are
\begin{align*}
    \lambda_1 = \cdots = \lambda_{m-1} = \frac{1+\alpha\tau}{1+\tau},\quad \lambda_m = \frac{1+\alpha\tau - \alpha\tau m}{1+\tau}.
\end{align*}
Under the optimal step size (modulo constants) $\tau = 1/(L\sqrt{m-1}) = 1/(\alpha (m-1))$, we have
\begin{align*}
    \lambda_1 = \cdots = \lambda_{m-1} = \frac{m}{m-1+\frac{1}{\alpha}},\quad \lambda_m = 0.
\end{align*}
Therefore, we can always find some initialization $x^0\in\mb R^m$ (e.g.,~any eigenvector associated with $\lambda_1$) such that
\begin{align*}
    \|x^k\| = \lambda_1^k \|x^0\| = \Big(1 - \frac{1-\alpha}{m\alpha +1 - \alpha}\Big)^k\|x^0\|.
\end{align*}
Under this initialization, the algorithm takes $O(m\alpha/(1-\alpha)) = O(L\sqrt{m}/\lambda)$ iterations to converge. This example shows that our convergence rate bound is unimprovable when $\tau$ is $O(\frac{1}{L\sqrt{m}})$.
\qed
\end{ex}

Next, we establish the exponential convergence rate of WPCG-S.
 
\begin{thm}[{\bf Exponential convergence rate of WPCG-S under $\lambda$-QG}]
\label{thm: sequential}
Assume Assumptions A, B, C, and D hold for some $\lambda > 0$. Let $\rho^k$ be the solution of WPCG-S at the $k$-th iteration. Then for any step size $\tau > 0$ we have
\begin{align*}
    \W_2^2(\rho^k, \rho^\ast) \leq \frac{2}{\lambda}\big[1 + \lambda C_2(m, L, \tau)\big]^{-k}\Big(\m F(\rho^0) - \m F(\rho^\ast)\Big),
\end{align*}
where
\begin{align*}
    C_2(m, L, \tau) = \Big(8\tau L^2(m-1) + 8\tau^{-1}\Big)^{-1} > 0.
\end{align*}
\end{thm}

\begin{rem}[{\bf Comparison with existing results in the Euclidean case}]
Similar to the parallel update scheme, by taking $\tau^{-1} = L\sqrt{m-1}$ the iteration complexity also grows at a rate of $O(\frac{L\sqrt{m}}{\lambda})$. When degenerated to optimization on the Euclidean space (by taking $\m H_j\equiv 0$ and $\m W_j \equiv 0$), our result is comparable to the rate $O(\frac{\sqrt{m}L_{\rm g}}{\lambda})$ in Theorem 6.3 of~\citep{wright2022optimization} for sequential coordinate gradient descent method with strongly convex and smooth functions. 
When we take $\tau = L^{-1}$, the iteration complexity turns out to be $O(\frac{Lm}{\lambda})$. This iteration complexity is no larger than $O(\frac{\max_j L_j}{\min_j\{L_j + \mu_j\}}\cdot\frac{m L_{\rm g}}{\lambda})$ derived by \citet{li2017faster} when applying coordinate proximal gradient descent with the sequential update scheme, where $\lambda_j$ and $L_j$ are the parameters of strong convexity and smoothness of the $j$-th block, and the step size for updating the $j$-th block is $L_j^{-1}$. Their result has an additional factor $\frac{\max_jL_j}{\min_j\{L_j+\mu_j\}}$ because they approximated the smooth part of the objective function by its first-order Taylor expansion before applying the proximal descent step. As we mentioned, we believe that the difference between the Lipschitz constant in our result and the ones chosen in~\citep{wright2022optimization, li2017faster} for optimization problems on the Euclidean space is due to our choice of a proximal-type optimization method. In fact, for the sequential update scheme, the Lipschitz constant can be further relaxed to the lower triangular Lipschitz constant~\citep{hua2015iteration}.
\qed
\end{rem}

\begin{rem}[{\bf Effect of step size in WPCG-S}]
The iteration complexity $O(\frac{L\sqrt{m}}{\lambda})$ when $\tau^{-1} = {L\sqrt{m-1}}$ in WPCG-S is smaller than the iteration complexity $O(\frac{mL_{\rm g}}{\lambda})$ derived by~\cite{sun2021worst} for alternating minimization algorithm (corresponding to $\tau = \infty$) with the sequential update scheme on the Euclidean space, which shows that using an overly large $\tau$ can still drag down the convergence speed. 
It is an interesting open problem that if the iteration complexity $O(\frac{L\sqrt{m}}{\lambda})$ can be improved. 
Even when it degenerates to an optimization problem on $\mb R^m$ ($\m H_j = \m W_j = 0$ and $\rho^0$ is a point mass measure), the optimality of this rate remains as an open problem. \cite{sun2021worst} showed that there exists an optimization problem (related to the Example~\ref{eg: parallel_CD} by choosing $\alpha$ carefully) on $\mb R^m$ such that alternating minimization algorithm with the sequential update scheme has iteration complexity $O(\frac{mL_{\rm g}}{\lambda})$.
\qed
\end{rem}

For WPCG-R, we can establish the following rate of convergence.

\begin{thm}[{\bf Exponential convergence rate of WPCG-R under $\lambda$-QG}]
\label{thm: randomized}
Let $M = \big\lceil 2m\log (mL)\big\rceil$. Assume Assumptions A, B, C, and D hold for some $\lambda > 0$. Let $\rho^k$ be the solution of WPCG-R at the $k$-th iteration with batch size $M$. Then for any step size $\tau > 0$, we have
\begin{align*}
    \mb E\m F(\rho^{k}) - \m F(\rho^\ast) \leq \min\{C_3(m, L, \tau), 1\big\}^k\big[\m F(\rho^0) - \m F(\rho^\ast)\big]
\end{align*}
where
\begin{align*}
    C_3(m, L, \tau) = \bigg(1 + \frac{\lambda\tau}{8[1+2e\tau^2L^2m\log(mL)]}\bigg)^{-1} + \frac{2}{mL^2}.
\end{align*}
The constant $C_3$ is a strictly positive constant less than one if
\begin{align*}
    \frac{4}{mL^2} < \frac{\lambda}{8[\tau^{-1} + 2e\tau L^2m\log(mL)]} < 1.
\end{align*}
\end{thm}

\begin{rem}[{\bf Comments on the batch size in WPCG-R}]
In the proof of Theorem~\ref{thm: randomized}, we may control the norm of the blockwise Wasserstein gradient by the distance between two consecutive iterates only when all coordinates have been updated at least once in each iteration. For this reason, we need to choose a larger batch size, $M = \lceil 2m\log(mL)\rceil$, compared to the other two update schemes (parallel and sequential), where precisely $M = m$ coordinates are updated in each iteration.
\qed
\end{rem}

\begin{rem}[\bf{Comparison with existing results in the Euclidean case}] \label{rem:WPCF-R_compare}
In optimization problems on the Euclidean space, the random update scheme usually has smaller iteration complexity compared with the other two schemes in worst case scenario, such as in coordinate GD~\citep{wright2015coordinate} and alternating minimization algorithm~\citep{sun2021worst}. However, on the Wasserstein space, we are only able to derive the same iteration complexity up to a log factor as in the other two update schemes, given $\lambda > 0$ and a fixed step size $\tau > 0$. In particular, when $\tau^{-1} = L\sqrt{2em\log(mL)}$, our result implies an $O(\frac{L\sqrt{m\log(mL)}}{\lambda})$ iteration complexity of WPCG-R, where in each iteration we need to solve $\lceil 2m\log(mL)\rceil$ sub-problems. Recall that there are only $m$ sub-problems in each iteration in WPCG-P and WPCG-S. Therefore, the iteration complexity in WPCG-R is $[\log(mL)]^{3/2}$ times larger than WPCG-P and WPCG-S. This logarithmic factor appears since we need to solve more sub-problems in each iteration compared with the other two update schemes, so that every coordinate has been updated with high probability in each iteration. When $\tau = L^{-1}$, we derive the iteration complexity $O(\frac{mL\log(mL)}{\lambda})$, which is slower than $O(\frac{L_{\rm c}}{\lambda})$ in the coordinate GD with random update scheme derived by~\citet{wright2022optimization} if $mL\log(mL) > L_{\rm c}$. We conjecture that such a sub-optimal rate when specialized to the Euclidean case~\citep{sun2021worst, wright2022optimization}
is due to the extra complexity of dealing with the Wasserstein space, where there is no tensorization structure of the Wasserstein gradient. More precisely, the following key identity
\begin{align*}
    \mb E_{j_k}\big\|\nabla_{j_k}f(x^k)\big\|^2 = \frac{1}{m}\sum_{j=1}^m\big\|\nabla_j f(x^k)\big\|^2 = \frac{1}{m}\big\|\nabla f(x^k)\big\|^2,\quad\text{where}\quad j_k\sim\mx{unif}\,([m])
\end{align*}
holds in the Euclidean case, indicating on average that randomly selecting a coordinate to update for $m$ times is similar to updating all coordinates at once as in the parallel scheme. 
So, a majority of the analysis for the parallel update scheme can be reused for analyzing the random update scheme that is expected to have smaller iteration complexity than the other two update schemes for Euclidean optimizations. Nevertheless, the similar identity does not hold for the Wasserstein gradient of potential energy $\m V$. Specifically, if we denote $\nabla_{\W_2}\m V(\rho) = \nabla V$ and $\nabla_{\W_2,j}\m V(\rho) = \int \nabla_{j}V(x_{j}, x_{-j})\,\dd\rho_{-j}$ as the Wasserstein gradient and the $j$-th marginal Wasserstein gradient of $\m V$ (with respect to $\rho_j$) at $\rho$, then $\nabla_{\W_2}\m V(\rho)$ does not tensorize, that is, 
\begin{align}\label{eqn: tensorization_issue}
\begin{aligned}
\sum_{j=1}^m\|\nabla_{\W_2,j}\m V(\rho)\|_{L^2(\rho)}^2
&= \sum_{j=1}^m \int_{\m X_j}\bigg\|\int_{\m X_{-j}}\nabla_j V(x_j, x_{-j})\,\dd\rho_{-j}\bigg\|^2\,\dd\rho_j\\
&\neq \sum_{j=1}^m \int_{\m X_j\times \m X_{-j}}\big\|\nabla_j V(x_j, x_{-j})\big\|^2\,\dd\rho_j\dd\rho_{-j}
= \|\nabla_{\W_2}\m V(\rho)\|_{L^2(\rho)}^2.
\end{aligned}
\end{align}
Thus in our proof, different from the Euclidean case, we cannot directly calculate the norm of the blockwise Wasserstein gradient of the updated coordinate and take expectation with respect to the index of the updated coordinate due to the same non-tensorization issue described in~\eqref{eqn: tensorization_issue}. We instead prove a high probability bound of the norm of blockwise Wasserstein gradient to control the functional value of $\m F$ after each iteration. When this bound does not hold, we directly use the non-increasing property of the functional value to bound $\m F$. By tuning the number $M$ of updates in each iteration, we can derive the convergence rate of WPCG-R in Theorem~\ref{thm: randomized}. It is still an open question whether the iteration complexity in the theorem can be improved by applying other strategies.
\qed
\end{rem}

\subsection{Convergence Rates of WPCG without $\lambda$-QG Assumption}

When the $\lambda$-QG condition is not met for a strictly positive $\lambda$, we obtain the following convergence result for all three update schemes. This result parallels that for the first-order optimization algorithm in Euclidean spaces under convexity but not strict convexity.

\begin{thm}[{\bf Polynomial convergence rate of WPCG without $\lambda$-QG}]
\label{thm: convex_case}
Let $D_{\m X} < \infty$ be the diameter of $\m X$. Assume the step size $0 < \tau < \frac{m-1/2}{L(m-1)^{3/2}}$ if the coordinate is updated with the parallel scheme, and Assumptions A, B, and C hold. Then, for each update scheme, there exists a constant $C > 0$ depending on $\tau, m, L$, and the update scheme, such that
\begin{align*}
    \mb E\m F(\rho^k) - \m F(\rho^\ast) \leq \frac{CD_{\m X}^2 + \m F(\rho^0) - \m F(\rho^\ast )}{k}.
\end{align*}
In the parallel and the sequential update schemes, $\mb E\m F(\rho^k) = \m F(\rho^k)$ since there is no randomness.
\end{thm}

\begin{rem}[{\bf Comparison with existing results in the Euclidean case}]
When the QG condition is not met for some $\lambda > 0$, the functional value in a typical first-order optimization algorithm decreases with a rate no faster than $O(k^{-1})$ in the worst case. This has been demonstrated through various studies on the Euclidean space especially for the sequential update scheme. Theorem 3 in \citep{wright2015coordinate} shows that the function value decreases with a rate $O(k^{-1})$ when applying the coordinate GD with the sequential update scheme. Theorem 11.18 in \citep{beck2017first} proves the same $O(k^{-1})$ convergence rate for the coordinate proximal GD with the sequential update scheme. Here, we derive the same convergence rate for WPCG with all three update schemes.
\qed
\end{rem}

The assumption of compactness of the parameter space is frequently utilized when applying the Wasserstein proximal gradient scheme to optimization problems, as seen in works by \cite{yao2022mean, santambrogio2015optimal}. This assumption plays a crucial role in our proof, allowing us to derive a uniform bound for $\W_2^2(\rho^k, \rho^\ast)$ and the first-order optimality condition of Wasserstein proximal gradient scheme (Lemma \ref{lem: ot_map}). The compactness is specifically used to determine the first variation of $\W_2^2(\mu, \nu)$ with respect to $\mu$. More details on this can be found in Proposition 7.17 and Theorem 1.52 in \citep{santambrogio2015optimal}. It would be interesting to explore potential methods for relaxing this technical condition.

\subsection{Convergence Rates of Inexact WPCG}

In this subsection, we investigate the inexact WPCG algorithm, where non-zero numerical errors are allowed in each iteration. Our focus will be on the parallel update scheme (WPCG-P), though analogous convergence results can be derived similarly for the other two update schemes.

To precisely characterize the inexact WPCG-P algorithm, let $\wt\rho_j^{k+1}$ denote the inexact solution of the subproblem~\eqref{eqn: subproblem-P} for updating the $j$-th coordinate in the $(k+1)$-th iteration. We define a vector-valued function
\begin{align*}
\eta_j^{k+1} \!=\! T_{\wt\rho_j^{k+1}}^{\wt\rho_j^{k}} \!- \id - \!\tau\bigg[\nabla V_j^{k} + \nabla h_j'(\wt \rho_j^{k+1}) + \nabla\!\int_{\m X}W_j(\cdot, y)\,\dd\wt\rho_j^{k+1}(y) + \nabla\!\int_{\m X}W_j(y, \cdot)\,\dd\wt\rho_j^{k+1}(y)\bigg]
\end{align*}
as the first variation of the objective functional in the subproblem~\eqref{eqn: subproblem-P} evaluated at $\wt\rho_j^{k+1}$.
If $\wt\rho_j^{k+1}$ is the exact solution of the subproblem~\eqref{eqn: subproblem-P}, the first-order optimality condition (Lemma~\ref{lem: ot_map}) implies $\eta_j^{k+1} = 0$. Therefore, we can employ $\|\eta_j^{k+1}\|_{L^2(\m X_j; \wt\rho_j^{k+1})}$ to characterize the numerical error incurred while solving the subproblem. Let $\{\varepsilon_j^k: j\in[m], k\in\mb Z_+\}$ be a sequence of error tolerance levels such that
\begin{align*}
\int_{\m X_j}\|\eta_j^{k+1}\|^2\,\dd\wt\rho_j^{k+1} \leq (\varepsilon_j^{k+1})^2
\end{align*}
for every $j\in[m]$ and $k\in\mb N$. This rigorous formulation allows us to quantify the impact of numerical errors stemming from the inexact subproblem solutions on the overall convergence behavior of the WPCG-P algorithm.

\begin{thm}[{\bf Convergence rate of inexact WPCG-P under $\lambda$-QG}]\label{thm: inexactWPCG_QG}
Assume Assumptions A, B, C, and D hold for some $\lambda > 0$, and the step size satisfies $0 < \tau < \frac{1}{2L\sqrt{m-1}}$. Let $\wt\rho^k$ be the inexact solution of WPCG-P at the $k$-th iteration with error tolerance levels $\varepsilon^k = (\varepsilon_1^k, \dots, \varepsilon_m^k)$.

\vspace{0.5em}
\noindent (1) If $\|\varepsilon^k\| \leq \varepsilon \kappa^k$ for some $\kappa\in(0, 1)$ and $\varepsilon > 0$, then there are positive constants $C_5 = C_5(\lambda, \tau, m, L) < 1$ and $C_6 = C_6(\lambda, \tau, m, L, \kappa)$ such that
\begin{align*}
\W_2^2(\wt\rho^k, \rho^\ast) \leq \frac{2}{\lambda}\big[\m F(\wt\rho^k) - \m F(\rho^\ast)\big]
\leq C_5^{k}\big[\m F(\rho^0) - \m F(\rho^\ast)\big] + C_6\varepsilon^2\max\{C_5, \kappa^2\}^{k+1}.
\end{align*}

\vspace{0.5em}
\noindent (2) If $\|\varepsilon^k\| \leq \varepsilon k^{-\alpha}$ for some $\varepsilon, \alpha \geq 0$, then there exists a constant $C_7 = C_7(\lambda, \tau, m, L, \alpha) > 0$ such that
\begin{align*}
\W_2^2(\wt\rho^k, \rho^\ast) \leq \frac{2}{\lambda}\big[\m F(\wt\rho^k) - \m F(\rho^\ast)\big]
\leq C_5^k\big[\m F(\rho^0) - \m F(\rho^\ast)\big] + \frac{C_7\varepsilon^2}{k^{2\alpha}}.
\end{align*}
\end{thm}

\begin{rem}[{\bf Impact of numerical error}]
This result elucidates the impact of numerical errors, highlighting how the rate of numerical error reduction influences the overall convergence behavior. Specifically, the convergence rate of WPCG-P will be hindered if the numerical error decreases at a polynomial rate across iterations. Conversely, the algorithm maintains exponential convergence if the numerical error decays exponentially. Furthermore, if $\kappa^2 \leq C_5$, the same dependence of the iteration count on the condition number $\frac{L}{\lambda}$ and the number of blocks $m$ can be derived by carefully tracking the constant $C_5$.
\end{rem}

The following result demonstrates that the convergence rate remains polynomial in the presence of a numerical error that decreases polynomially, when the $\lambda$-QG condition is not met for a strictly positive $\lambda$. Moreover, the iteration count depends on the rate at which the numerical error diminishes, and as anticipated, will be no less than the iteration count of the exact WPCG-P algorithm.

\begin{thm}[{\bf Polynomial convergence rate of inexact WPCG-P without $\lambda$-QG}]\label{thm: inexactWPCG_conv}
Let $D_{\m X} < \infty$ be the diameter of $\m X$. Assume the step size satisfies $0 < \tau < \frac{1}{2L\sqrt{m-1}}$ and Assumptions A, B, and C hold.
If $\|\varepsilon^k\| \leq \varepsilon k^{-\alpha}$ for some $\varepsilon, \alpha \geq 0$, there exists a constant $C_8 = C_8(\alpha, \tau, m, L, D_{\m X}, \varepsilon) > 0$ such that
\begin{align*}
\m F(\wt\rho^k) - \m F(\rho^\ast) \leq \frac{C_8}{k^{\min\{1,\alpha\}}}
\end{align*}
holds for all $k\in\mb Z_+$.
\end{thm}

\section{Numerical Experiments}\label{sec:apps}

In this section, we demonstrate the application of our WPCG algorithms to approximate the posterior distribution via the mean-field variational inference in Example~\ref{ex:MFVI} and to compute the stationary distribution of the multi-species systems with cross-interaction in Example~\ref{ex:multi-species}.

\subsection{Mean-field Variational Inference}\label{sec: MFVI}


\begin{figure}[t]
    \centering
    \includegraphics[width=0.9\textwidth]{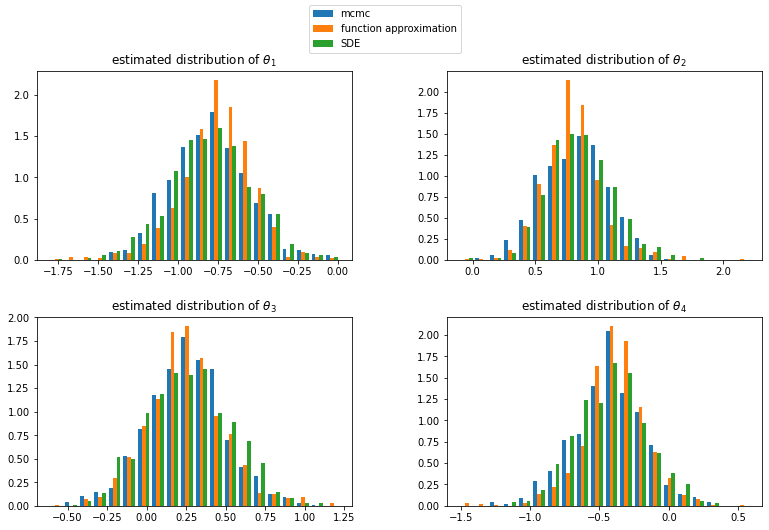}
    \caption{Histograms of particles used to approximate posterior distributions by MCMC and MFVI through WPCG-P with the FA approach and the SDE approach. Similar marginal distributions are derived by both WPCG-P (with the SDE approach or the FA approach) and MCMC.}
    \label{fig: MFVI_histogram}
\end{figure}

\begin{figure}[t]
    \centering
    \includegraphics[scale=0.8]{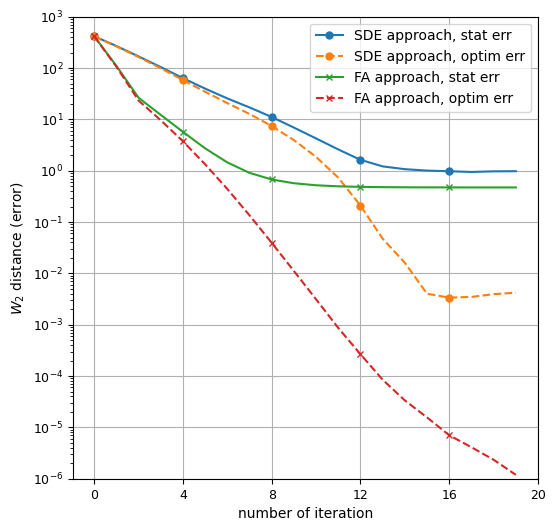}
    \caption{Numerical errors of the SDE approach and the FA approach. The numerical error $\W_2^2(\rho^k, \delta_{\theta^\ast})$ is first dominated by the optimization error $\W_2^2(\rho^k, \rho^\ast)$ which decays exponentially fast, and later dominated by the statistical error $\W_2^2(\rho^\ast, \delta_{\theta^\ast})$. The statistical error derived by the SDE approach is a bit larger than the one derived by the FA approach since we discretize the Langevin dynamics to approximate the Wasserstein proximal gradient scheme. The optimization error derived by the FA approach decreases exponentially fast as predicted by our theoretical results. In contrast, in the SDE approach, the optimization error will be dominated by the approximation error after several iterations.}
    \label{fig: MFVI_loss}
\end{figure}

Recall that $\Theta = \bigotimes_{j=1}^m\Theta_j$ is the parameter space, $p(x\,|\,\theta)$ is the likelihood function, $\pi$ is the prior density function, and $X_1, \dots, X_n$ are $n$ observations. In MFVI, we approximate the posterior distribution by the solution of the following optimization problem,
\begin{align*}
    \wht\pi_n = \argmin_{\rho = \bigotimes_{j=1}^m\rho_j}\KL(\rho\,\|\,\Pi_n),\quad\mx{s.t.}\quad \rho_j\in\ms P(\Theta_j)\quad\forall\, j\in[m],
\end{align*}
where $\Pi_n$ is the posterior distribution with the density function given by~\eqref{eqn: posterior_density}. As a corollary of Theorem~\ref{thm: main_thm}, we have the following convergence result about computing $\wht\pi_n$ via WPCG-P.

\begin{cor}\label{cor:MFVI_rate}
Let $\rho^k$ be the $k$-th iterate by applying WPCG-P (Algorithm \ref{algo: WPCG-P}) to MFVI with $V = V_n = -\sum_{i=1}^n\log p(X_i\,|\,\theta) - \log\pi(\theta)$, $h_j(x) = x\log x$, and $W_j = 0$. If $V_n$ is $\lambda_n$-strongly convex and $L_n$-smooth, and $0 < \tau < \frac{2m-1}{2L_n(m-1)^{3/2}}$, we have
\begin{align*}
    \W_2^2(\rho^k, \wht\pi_n) \leq (1+C_n\lambda_n)^{-k}\big[\KL(\rho^0\,\|\,\Pi_n) - \KL(\wht\pi_n\,\|\,\Pi_n)\big],
\end{align*}
where the constant
\begin{align*}
C_n = \frac{\frac{1}{2\tau} + (m-1)\big(\frac{1}{\tau} - L_n\sqrt{m-1}\big)}{2m[2L_n^2(m-1) + \frac{2}{\tau^2}]} > 0.
\end{align*}
\end{cor}

Note that in Corollary~\ref{cor:MFVI_rate}, both $\lambda_n$ and $L_n$ depend on the samples $X_1, \dots, X_n$. When the log-likelihood function $\log p(x\,|\,\theta)$ satisfies certain mild conditions \citep[e.g., Assumptions 2 and 3 in][]{mei2016landscape}, the sample convexity parameter $\lambda_n$ and the smoothness parameter $L_n$ will concentration in the small neighborhoods of the population convexity parameter $\lambda = \mb E_{\theta^\ast}\lambda_n$ and the smoothness parameter $L = \mb E_{\theta^\ast}L_n$ respectively with high probability.

\begin{rem}
When applying WPCG-P to MFVI, our method coincides with representations in \citep{ghosh2022representations}. However, \citet{ghosh2022representations} only studied the convergence of the piecewise constant interpolation $\{\rho_t = \rho^{\lfloor t/\tau\rfloor}: t\geq 0\}$ towards the solution of its corresponding WGF \citep[Equation (15) in][]{ghosh2022representations} as the step size $\tau\to 0$ when $V_n$ is strongly convex.
\citet{yao2022mean} proposed the MF-WGF algorithm for solving MFVI with latent variables numerically. They showed exponential convergence of iterates towards the mean-field approximation $\wht\pi_n$ of the true posterior in $\W_2$ sense. Their proof heavily depended on the strong convexity of the negative log-likelihood function, while our analysis only uses the strong convexity to show QG condition.
\citet{lambert2022variational} studied Gaussian variational inference (i.e., $m=1$ and the variational family is restricted to all Gaussian distributions) and showed that the corresponding gradient flow (so-called Bures--Wasserstein gradient flow) converges to the Gaussian variational approximation exponentially fast in $\W_2$ sense. Compared with Gaussian variational inference, our method is more flexible with the number of blocks and the variational family.
\qed
\end{rem}

To support our theoretical findings, we simulate a Bayesian Logistic Regression model with prior $\pi = \m N(0, 4I_4)$ and data generating process
\begin{align*}
X_i\sim\m N(0, I_4), \quad y_i\,|\,X_i, \theta &\sim \text{Bernoulli}(p_i) \quad\mbox{where} \quad \log\frac{p_i}{1-p_i} = X^T_i\theta,
\end{align*}
with the ground truth being $\theta^\ast = (-1, 1, 0.3, -0.3)$. $n=100$ samples generated from the true model are then used to estimate the posterior by implementing MFVI through the WPCG-P algorithm. 
Two different numerical methods, the FA approach and the SDE approach, are implemented. In both numerical methods, $B = 1000$ particles are used to approximate each marginal distribution. In the FA approach, a neural network with three fully connected hidden layers is used to approximate the optimal transport map. Each hidden layer consists of $1000$ neurons, and the activation function is ReLu defined as $\text{ReLu}(x) = \max\{x, 0\}$. 

Figure \ref{fig: MFVI_histogram} compares the WPCG-P algorithm with MCMC. The results show that similar approximations of each marginal distribution of the posterior distribution are derived by both methods. Figure \ref{fig: MFVI_loss} presents the numerical errors by applying the FA approach and the SDE approach in WPCG-P. The numerical error $\W_2^2(\rho^k, \delta_{\theta^\ast})$ is first dominated by the optimization error $\W_2^2(\rho^k, \rho^\ast)$ (corresponding to the decreasing parts of the solid lines) and then dominated by the statistical error $\W_2^2(\rho^\ast, \delta_{\theta^\ast})$ (corresponding to the horizontal parts of the solid lines) after several iterations. In the SDE approach, the optimization error will finally be dominated by the approximation error, which appears since we use discretized SDE to approximate the Wasserstein proximal gradient scheme, and stop decreasing as shown by the orange dashed line. As comparison, the optimization error in the FA approach will decrease exponentially fast as predicted by our theoretical findings, which is shown by the red dashed line.

\subsection{Equilibrium in Multi-species Systems}\label{subsec:multi-species_systems}

\begin{figure}[t]
    \centering
    \includegraphics[scale=0.75]{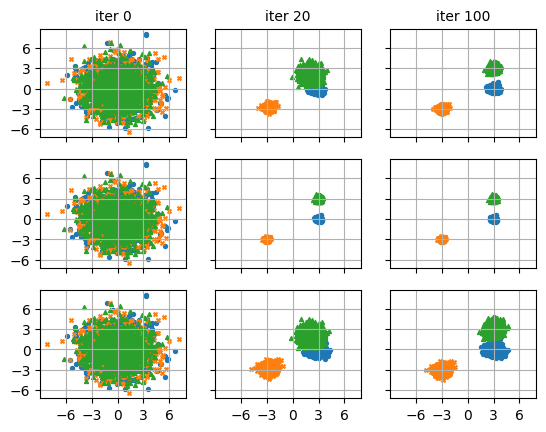}
\caption{Scatter plots of all three species for different $\alpha$ and $\beta$. We use $B=1000$ particles to approximate each species. The first row corresponds to $\alpha = 1$ and $\beta=1$; the second row corresponds to $\alpha=5$ and $\beta=1$; the third row corresponds to $\alpha=1$ and $\beta = 10$. Species concentrate 
more around the centers $\theta_1$, $\theta_2$, and $\theta_3$  with larger $\alpha$ and smaller $\beta$.}
    \label{fig: scatter_plot}
\end{figure}

\begin{figure*}[t!]
    \centering
    \begin{subfigure}[t]{0.49\textwidth}
        \centering
        \includegraphics[height=3in]{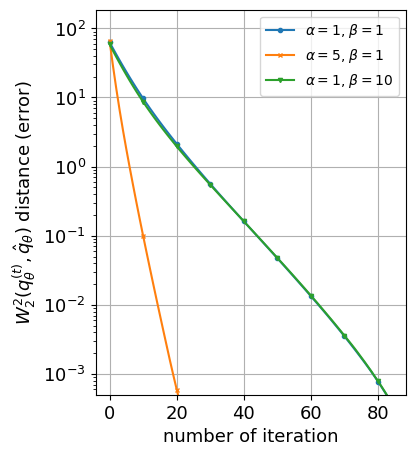}
        \caption{Optimization error $\W_2^2(\rho_\theta^{k}, \rho_\theta^\ast)$.}
        \label{fig: optim_err}
    \end{subfigure}%
    ~ 
    \begin{subfigure}[t]{0.49\textwidth}
        \centering
        \includegraphics[height=3in]{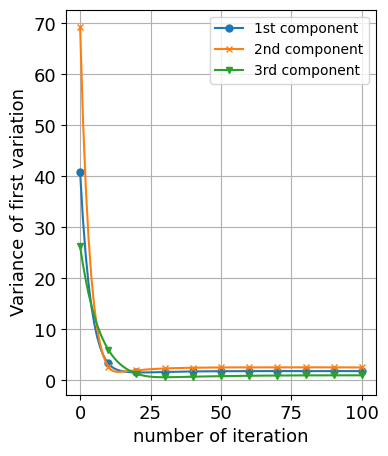}
        \caption{$\text{Var}_{\rho_j^{k}}\big(\frac{\delta\m F}{\delta\rho_j}(\rho^{k})(X_j^k)\big)$ for $1\leq j\leq 3$.}
        \label{fig: var_firstvaiation}
    \end{subfigure}
    \caption{(a) Comparison of optimization error $\W_2^2(\rho_\theta^{k}, \rho_\theta^\ast)$ for different $\alpha$ and $\beta$. All these lines are straight, indicating that the optimization error decays exponentially fast. Contraction rate does not change by only changing $\beta$ (blue line v.s.~green line), while it gets smaller (i.e., faster convergence) when $\alpha$ gets larger (orange line v.s.~blue line). (b) Variance of first variation $\text{Var}_{\rho_j^{k}}\big(\frac{\delta\m F}{\delta\rho_j}(\rho^{k})(X_j^k)\big)$ versus number of iterations. That the variance converges to $0$ indicates that the first variation $\frac{\delta\m F}{\delta\rho_j}(\rho^k)$ converges to a constant, which implies that $(\rho_1^{k}, \rho_2^{k}, \rho_3^{k})$ converges to the minimum of $\m F$.}
    \label{fig: cross_diffusion_loss}
\end{figure*}

\begin{figure}[t]
\centering
    \includegraphics[scale=0.8]{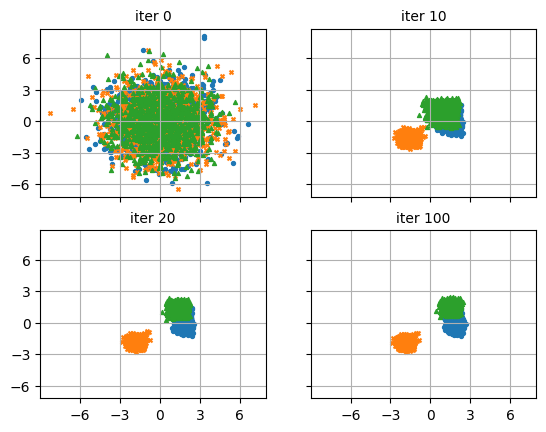}
    \caption{Scatter plots of all three species when the potential $V$ contains a super quadratic term $\|x\|^4/4$. As expected, all species converge to their equilibrium; they get closer to the origin compared with the case without the super quadratic term.}
    \label{fig: scatter_plot_superquad}
\end{figure}

\begin{figure*}[t!]
    \centering
    \begin{subfigure}[t]{0.49\textwidth}
        \centering
        \includegraphics[height=3in]{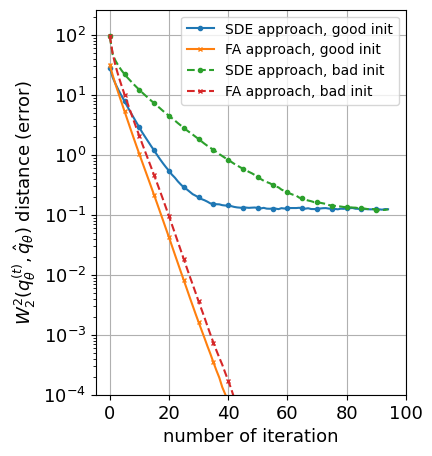}
        \caption{Optimization error $\W_2^2(\rho_\theta^{k}, \rho_\theta^\ast)$.}
        \label{fig: optim_err_superquad}
    \end{subfigure}%
    ~ 
    \begin{subfigure}[t]{0.49\textwidth}
        \centering
        \includegraphics[height=3in]{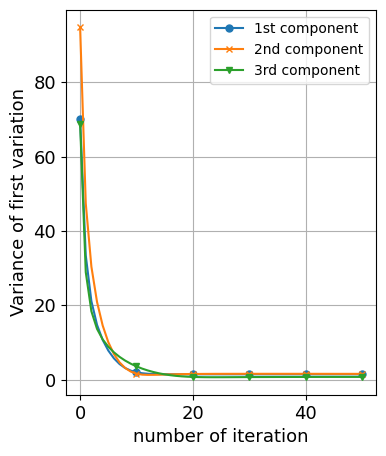}
        \caption{$\text{Var}_{\rho_j^{k}}\big(\frac{\delta\m F}{\delta\rho_j}(\rho^{k})(X_j^k)\big)$ for $1\leq j\leq 3$.}
        \label{fig: first_variation_superquad}
    \end{subfigure}
    \caption{(a) Comparison of optimization error $\W_2^2(\rho_\theta^{k}, \rho_\theta^\ast)$ derived by the SDE approach and the FA approach when $h_j(x) = x\log x$. The optimization error derived by applying the FA approach (the orange line and the red line) decays exponentially fast despite of the existence of the super-quadratic term. In contrast, the optimization error derived by applying the SDE approach (the green line and the blue line) is dominated due to the approximation error. Moreover, due to the super-quadratic term, it takes more iterations for the SDE approach to converge when the initialization gets farther from the origin.  (b) Variance of the first variation $\text{Var}_{\rho_j^{k}}\big(\frac{\delta\m F}{\delta\rho_j}(\rho^{k})(X_j^k)\big)$ versus number of iterations. That the variance converges to $0$ indicates that the first variation $\frac{\delta\m F}{\delta\rho_j}(\rho^k)$ converges to a constant, which implies that $(\rho_1^{k}, \rho_2^{k}, \rho_3^{k})$ converges to the minimum of $\m F$.}
    \label{fig: cross_diffusion_loss_superquad}
\end{figure*}

Here we apply WPCG to find the stationary distribution of the following non-local multi-species cross-interaction model with diffusion, whose aggregation equation is given in~\eqref{eqn:aggregation_model}. 
Our numerical result below shows that the algorithm converges exponentially fast as predicted by our theory when the corresponding objective functional is convex and satisfies the QG condition.
We also conduct a numerical study to compare the FA approach and the SDE approach when $\m H_j(\rho_j) = \int\rho_j\log\rho_j$ for all $j\in[m]$. The numerical study shows that the FA approach is not affected by the existence of super-quadratic terms in the potential energy function. 

Let us start from the following result which connects the PDE~\eqref{eqn:aggregation_model} and the optimization problem~\eqref{eqn: obj_func}.
\begin{prop}\label{lem: multi_species_system}
Consider the optimization problem~\eqref{eqn: obj_func} with
\begin{align*}
    V(x_1, \dots, x_m) &= \sum_{i=1}^m V_i(x_i) - \sum_{1\leq i<j\leq m}K_{ij}(x_i - x_j)\\
    W_j(x_j, x_j') &= K_{jj}(x_j - x_j')
\end{align*}
If Assumptions A, B, C, and D hold, and $K_{ij} = K_{ji}$ are even functions, then the solution $\rho^\ast  = (\rho_1^\ast, \dots, \rho_m^\ast)$ of problem~\eqref{eqn: obj_func} is the stationary distribution of the system~\eqref{eqn:aggregation_model}.
\end{prop}
\begin{rem}\label{rmk: symmetirc_assump}
(i) In this lemma, we assume $K_{ij} = K_{ji}$ for all $i, j\in[m]$ so that the evolution of $\{\rho_j(\cdot, t): j\in[m], t\geq 0\}$ follows the WGF of some functional $\m F$. When $m=2$, this condition can be relaxed to $K_{12} = \alpha K_{21}$ for some constant $\alpha > 0$ by rescaling the WGF of $\rho_1$ and $\rho_2$ \citep{di2013measure}.
(ii) The choice of $h_j(x) = x\log x$ corresponds to the standard diffusion without medium of species $j$, and \eqref{eqn:aggregation_model} turns to be the PDE of multi-species stochastic interacting particle systems \citep{daus2022random}. As a comparison, $h_j(x) = x^{m_j}$\footnote{Here the exponent $m_j>1$ depends on physical properties of the diffusing material.} corresponds to a porous medium type diffusion \citep{otto2001geometry}.
When $m=1$, \eqref{eqn:aggregation_model} is the aggregation equation of interacting particle systems \citep{cabrales2020numerical, carrillo2019aggregation, li2010wellposedness}; when $m=2$ and $V_j = h_j = 0$, \eqref{eqn:aggregation_model} turns into the PDE of non-local interaction systems with two species \citep{di2013measure, evers2017equilibria}. To our knowledge, Equation~\eqref{eqn:aggregation_model} has not been studied for general $m\in\mb N^\ast$ and functions $h_j$, $j\in[m]$. We believe that our theory is also useful to study the existence and the uniqueness of this kind of PDEs from the perspective of WGF.
\qed
\end{rem}

In the numerical study, we consider the following test example \citep{daus2022random, jin2020random} with $m=3$ species in $\mb R^2$. We specify the functions
\begin{align*}
    K_{ij}(x) = \frac{Q_iQ_j}{2}\arctan\|x\|^2,\quad h_j(\rho_j)=\beta\rho_j^2, \quad\mbox{and}\quad V_j(x_j) = \frac{\alpha r_j}{2}\|x_j - \theta_j\|^2, \quad 1\leq i,j\leq 3,
\end{align*}
with model parameters $(Q_1, Q_2, Q_3) = (1, -1, 0.5)$, $(r_1, r_2, r_3) = (6, 7, 3)$, and $\theta_1 = (3, 0)$, $\theta_2 = (-3, -3)$, and $\theta_3 = (3, 3)$. It is easy to check that (see Appendix \ref{appendix: multi_species} for more details)
\begin{align*}
    V(x_1, x_2, x_3) = \sum_{j=1}^n \frac{\alpha r_j}{2}\|x_j - \theta_j\|^2 - \sum_{1\leq i<j\leq 3}\frac{Q_iQ_j}{2}\arctan\|x_i-x_j\|^2
\end{align*}
and 
\begin{align*}
    W_j(x_j, x_j') = -\frac{Q_j^2}{4}\arctan\|x_j - x_j'\|^2,\quad j=1, 2, 3
\end{align*}
satisfy Assumptions A, B, C, and D in Section \ref{sec: WPCG} when $\alpha\geq 1$ and $\beta \geq 0$. We will apply WPCG with the FA approach by solving the optimization problem~\eqref{eqn: FA_generalh}, since the corresponding $\m H_j$ is not the negative self-entropy functional.

Figures \ref{fig: scatter_plot}, \ref{fig: cross_diffusion_loss}, \ref{fig: scatter_plot_superquad}, and \ref{fig: cross_diffusion_loss_superquad} present the numerical results of this example. $B=1000$ particles are generated to approximate each species, and a neural network with 2 hidden layers is used to solve~\eqref{eqn: FA_generalh} in the FA approach. Each hidden layer consists of 800 neurons, and the activation function is ReLu. In Figures \ref{fig: scatter_plot_superquad} and \ref{fig: cross_diffusion_loss_superquad}, we add a super-quadratic term $\|x\|^4/4$ to the potential function $V$ and change the entropy functional to the negative self-entropy to compare the SDE approach and the FA approach.

Figure~\ref{fig: scatter_plot} shows the scatter plots of all species with different $\alpha$ and $\beta$. When $\alpha$ gets larger, all species concentrate more around the corresponding centers ($\theta_1$, $\theta_2$, and $\theta_3$) due to a larger external force. In contrast, larger $\beta$ makes all species less concentrate since the entropy term penalizes the concentration of species. 
Figure~\ref{fig: cross_diffusion_loss} shows the optimization error $\W_2^2(\rho_\theta^{k}, \rho_\theta^\ast)$ for different $\alpha$ and $\beta$ and the variance of the first variation when $\alpha=\beta=1$. In Figure~\ref{fig: optim_err}, the optimization error decays faster with larger $\alpha$ (compare the blue line with the orange line) since the convexity gets stronger, while changing $\beta$ will not change the convergence speed (compare the blue line with the green line) since the entropy term does not provide any convexity to the whole functional. Figure~\ref{fig: var_firstvaiation} plots the variance of the first variation $\frac{\delta\m F}{\delta\rho_j}(\rho^{k})(X_j^{k})$ where $X_j^{k}\sim\rho_j^{k}$. This variance converges to $0$ if and only if the first variation $\frac{\delta\m F}{\delta\rho_j}(\rho^{k})$ converges to a constant [$\rho_j$]-a.e., which indicates that $(\rho_1^{k}, \rho_2^{k}, \rho_3^{k})$ converges to the minimum of $\m F$. 

Figure~\ref{fig: scatter_plot_superquad} shows the scatter plot of all species when there exists an extra super-quadratic term $\|x\|^4/4$ in the potential function $V$. Due to this super-quadratic term, all species get closer to the origin when the system is at its equilibrium. 
Figure~\ref{fig: cross_diffusion_loss_superquad} shows the optimization error $\W_2^2(\rho_\theta^k, \rho_\theta^\ast)$ and the variance of the first variation. Figure~\ref{fig: optim_err_superquad} compares the optimization error derived by WPCG with the FA approach and the SDE approach. Compare with the FA approach, the optimization error derived by applying the SDE approach (the blue line and the green line) is dominated by the approximation error after several iterations. Due to the existence of the super-quadratic term, a smaller step size has to be chosen to make the SDE approach converge when the initialization gets farther from the origin, which increases the number of iterations for the SDE approach to converge. In contrast, the optimization error derived by the FA approach (the red line and the orange line) is not affected by the super-quadratic term as predicted by our theoretical result. Fig~\ref{fig: first_variation_superquad} plots the variance of the first variation $\frac{\delta\m F}{\delta\rho_j}(\rho^k)$. The variance converges to zero as a sign of convergence of $\rho^{k}$ to the minimum of $\m F$.

\subsection{Real Data Example}
We apply the WPCG-P algorithm to analyze the Pima Indian diabetes data set~\citep{smith1988using}. This data set comprises 8 medical predictors (\texttt{pregnancies}, \texttt{glucose}, \texttt{blood pressure}, \texttt{skin thickness}, \texttt{insulin}, \texttt{body mass index}, \texttt{diabetes pedigree function}, \texttt{age}) and a binary outcome variable ($0$ for no diabetes, $1$ have diabetes) from 768 individuals.

We construct a Bayesian logistic regression model (with an intercept term) to predict the outcome based on all 8 medical predictors. A stratified sample of 68 data points is set aside as the test set, while the remaining 700 points constitute the training set. We employ the WPCG-P algorithm with the FA approach to compute the mean-field variational approximation of the posterior distribution over the intercept and predictor coefficients. Figure~\ref{fig: optim_loss_realdata} depicts the optimization error $\W_2^2(\rho^k, \rho^\ast)$ across iterations, which decays exponentially fast as predicted by our theoretical findings.

\begin{figure}
    \centering
    \includegraphics[scale=0.6]{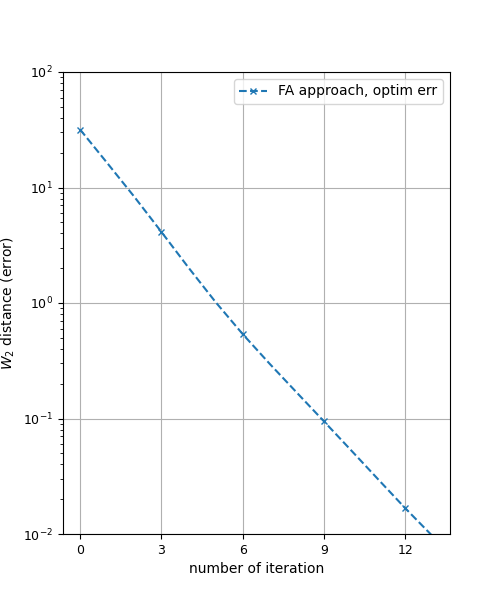}
    \caption{Optimization error $\W_2^2(\rho^k, \rho^\ast)$.}
    \label{fig: optim_loss_realdata}
\end{figure}

Furthermore, We compare our algorithm with MCMC for sampling from the posterior distribution. Both WPCG-P and MCMC achieve a misclassification rate $19.1\%$ and a cross entropy loss $0.496$ on the test set. Figure~\ref{fig: box_plot} presents the box plot of the mean-field variational approximations of the posterior distributions over the intercept and predictor coefficients. Consistent with domain knowledge, the analysis indicates that larger values of \texttt{pregnancies},\texttt{Glucose}, \texttt{BMI}, and \texttt{diabetes pedigree function} are associated with an increasing possibility of having diabetes.

\begin{figure}
    \centering
    \includegraphics[scale=0.65]{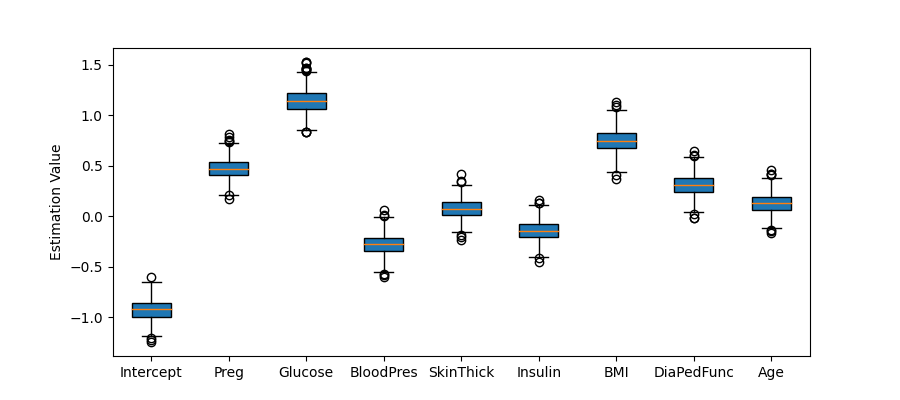}
    \caption{Box plot of the mean-field variational approximation of the posterior distribution of the intercept and all predictors.}
    \label{fig: box_plot}
\end{figure}


\section{Conclusions and Discussions}\label{sec: discussion}
In this paper, we introduced the WPCG algorithms for solving multivariate composite convex minimization problem on the Wasserstein space. We established its exponential convergence results with different update schemes under the quadratic growth condition. When this condition is not met, we also showed a slower polynomial convergence result for the WPCG algorithm. Additionally, we analyze the convergence behavior of the inexact WPCG algorithm with the parallel update scheme. We believe that similar results hold for the other two update schemes. We also conducted numerical studies about mean-field variation inference and multi-species systems to verify the predictions from our theory.

\noindent To conclude this paper, we list several open problems as future directions.

\begin{enumerate}
    \item An interesting topic is to develop the PL-type condition for multivariate objective functionals and investigate its relationship with the QG condition. This will enable us to relax the assumptions of convexity and the QG condition, and analyze the convergence rate of the WPCG algorithm for non-convex objective functionals. 

    \item It is interesting to see whether it would be beneficial to augment the Euclidean space $\mathbb R^m$ to the space of all product measures $\bigotimes_{j=1}^m \rho_j$, which includes $\mathbb R^m$ as a special case by restricting each $\rho_j$ to be a point mass measure, when minimizing a function $V$ over $\mathbb R^m$. This augmentation may help avoid the algorithm quickly getting trapped into a local minimum in case of non-convexity, similar to SGD which avoids getting trapped by injecting random noise; the augmented algorithm can be viewed as the infinite particle limit of an evolutionary algorithm for minimizing $V$.

    \item It is still unknown if the convergence rates of WPCG are optimal with respect to the number of blocks $m$. The dependence on $m$ will be important when $m$ is very large. Also, The current convergence result of WPCG-R seems worse than the one in the coordinate GD with the random update scheme on the Euclidean space. It is interesting to see if this convergence rate is due to the analyzing strategies or the difference between the Wasserstein space and the Euclidean space.

    \item Another potential topic is to consider the coordinate Wasserstein forward-backward gradient descent algorithm, which updates each coordinate by approximating the potential energy $\m V$ in first order (forward step) followed by a proximal descent step (backward step). This algorithm corresponds to the  proximal coordinate gradient descent algorithm in the Euclidean case. The case $m=1$ on the Wasserstein space has been studied by~\cite{salim2020wasserstein}. We believe our analysis of WPCG can be extended to the coordinate Wasserstein forward-backward gradient descent algorithm.

    \item In many problems in statistics (such as MFVI in Section~\ref{sec: MFVI}), the potential function takes the form $V = \sum_{i=1}^n V_i$, where $n$ can be treated as the sample size. When $n$ is large, there will be huge computational cost to calculate $\nabla V$ in the proximal gradient descent step. One possible way is to approximate $\nabla V\approx \frac{n}{n_s}\sum_{l=1}^{n_s}\nabla V_{i_l}$ by subsampling $i_l$ from $[n] := \{1, \dots, n\}$ independently for $l\in[n_s]$ in each update, where $n_s$ is the batch size. It is interesting to see how this stochastic WPCG algorithm converges and the impact of the batch size $n_s$.
\end{enumerate}


\acks{Xiaohui Chen was partially supported by NSF CAREER grant DMS-2347760, NSF grant DMS-2413404, and a gift from the Simons Foundation. Yun Yang was partially supported by NSF grant DMS-2210717.}



\newpage
\begin{center}
    {\LARGE \bf Appendices}
\end{center}
\appendix
\section{Proof of Main Results}\label{app: main proof}
Since the whole proof is quite involved, we will first provide with the sketch, and detail the proof in next several subsections. For every update scheme when $\lambda > 0$, we proceed with the following two steps

\vspace{0.5em}
\noindent\underline{Step 1.} (Sufficient decrease.) In this step, we will prove
\begin{align*}
\m F(\rho^{k-1}) - \m F(\rho^k) \geq C\W_2^2(\rho^{k-1}, \rho^k)
\end{align*}
holds with some constant $C > 0$ depending only on $\tau$, $\lambda$, $L$, and the update scheme.


\vspace{0.5em}
\noindent\underline{Step 2.} (Bound of functional value.) In this step, we will prove
\begin{align*}
\m F(\rho^k) - \m F(\rho^\ast) \leq C'\W_2^2(\rho^k, \rho^{k-1})
\end{align*}
holds with some constant $C' > 0$ depending only on $\tau$, $\lambda$, $L$, and the update scheme.

With the above two steps, we have
\begin{align*}
\m F(\rho^{k-1}) - \m F(\rho^k) \geq C\W_2^2(\rho^{k-1}, \rho^k) \geq \frac{C}{C'}\big(\m F(\rho^k) - \m F(\rho^\ast)\big)
\end{align*}
This implies
\begin{align*}
    \m F(\rho^k) - \m F(\rho^\ast) \leq \Big(1+\frac{C}{C'}\Big)^{-1}\big(\m F(\rho^{k-1}) - \m F(\rho^\ast)\big) \leq\cdots\leq \Big(1+\frac{C}{C'}\Big)^{-k}\big(\m F(\rho^0) - \m F(\rho^\ast)\big).
\end{align*}
Since $\m F$ satisfies ($\lambda$-QG) condition, we have
\begin{align*}
\frac{\lambda}{2}\W_2^2(\rho^k, \rho^\ast) \leq \Big(1 + \frac{C}{C'}\Big)^{-k}\big(\m F(\rho^0) - \m F(\rho^\ast)\big).
\end{align*}
This implies the desired result. Next, we will detail the proof for different update schemes.

\subsection{Parallel Update Scheme: Proof of Theorem \ref{thm: main_thm}}
As introduced in the sketch, we will use the following two results to bound the decrease of functional values and the difference between the functional value and the minimum of $\m F$, whose proof will be postponed to Appendix~\ref{app: additional_proof}.
\begin{lemma}[sufficient decrease]\label{lem: sufficient_decrease_cond}
Under Assumptions A, B, and C, when the step size satisfies $0 < \tau < \frac{2m-1}{2L(m-1)^{3/2}}$, we have
\begin{align*}
    \m F(\rho^{k-1}) - \m F(\rho^{k}) \geq C_4(m, L, \tau)\W_2^2(\rho^{k-1}, \rho^{k}) > 0,
\end{align*}
where the constant
\begin{align*}
    C_4(m, L,\tau) = \frac{1}{m}\bigg[\frac{1}{2\tau} + (m-1)\Big(\frac{1}{\tau} - L\sqrt{m-1}\Big)\bigg].
\end{align*}
\end{lemma}

\begin{lemma}[bound of functional value]\label{lem: bound_func_value}
Under Assumptions A, B, C, and D, in both the parallel update scheme and the sequential update scheme, we have
\begin{align*}
    \m F(\rho^k) - \m F(\rho^\ast ) \leq \frac{2}{\lambda}\Big(2L^2(m-1) + \frac{2}{\tau^2}\Big)\W_2^2(\rho^k, \rho^{k-1}).
\end{align*}
\end{lemma}

\vspace{0.5em}
With the above two lemmas,
\begin{align*}
\m F(\rho^{k-1}) - \m F(\rho^k) 
&\stackrel{\text{Lem~\ref{lem: sufficient_decrease_cond}}}{\geq} \frac{1}{m}\bigg[\frac{1}{2\tau} + (m-1)\Big(\frac{1}{\tau} - L\sqrt{m-1}\Big)\bigg]\W_2^2(\rho^{k-1}, \rho^k)\\
&\stackrel{\text{Lem~\ref{lem: bound_func_value}}}{\geq}\frac{\frac{1}{m}\big[\frac{1}{2\tau} + (m-1)(\frac{1}{\tau} - L\sqrt{m-1})\big]}{\frac{2}{\lambda}\big(2L^2(m-1) + \frac{2}{\tau^2}\big)}\big(\m F(\rho^k) - \m F(\rho^\ast)\big)\\
&= \lambda C_1(m, L, \tau)\big(\m F(\rho^k) - \m F(\rho^\ast)\big),
\end{align*}
where 
\begin{align*}
C_1(m, L, \tau) = \frac{\frac{1}{2\tau} + (m-1)\big(\frac{1}{\tau} - L\sqrt{m-1}\big)}{2m[2L^2(m-1) + \frac{2}{\tau^2}]}
\end{align*}
is defined in Theorem \ref{thm: main_thm}. Therefore, we have
\begin{align*}
\frac{\lambda}{2}\W_2^2(\rho^k, \rho^\ast) \leq \m F(\rho^k) - \m F(\rho^\ast ) \leq \big[1+\lambda C_1(m, L, \tau)\big]^{-t}\big(\m F(\rho^0) - \m F(\rho^\ast )\big),
\end{align*}
where the first inequality is due to Assumption D ($\lambda$-QG condition). This ends the proof.

\subsection{Sequential Update Scheme: Proof of Theorem \ref{thm: sequential}}
In the sequential update scheme, we have the following lemma to bound the decrease of functional values, whose proof will be postponed to Appendix~\ref{app: additional_proof}.

\begin{lemma}[sufficient decrease]\label{lem: sufficient_decrease_sequential}
For two consecutive updates $\rho^{k-1}$ and $\rho^k$, we have
\begin{align*}
    \m F(\rho^{k-1}) - \m F(\rho^{k}) \geq \frac{1}{2\tau}\W_2^2(\rho^{k-1}, \rho^k).
\end{align*}
\end{lemma}

\noindent With the above lemma, we have
\begin{align*}
\m F(\rho^{k-1}) - \m F(\rho^k) &\stackrel{\text{Lem~\ref{lem: sufficient_decrease_sequential}}}{\geq} \frac{1}{2\tau}\W_2^2(\rho^{k-1}, \rho^k)\\
&\stackrel{\text{Lem~\ref{lem: bound_func_value}}}{\geq} \frac{\lambda}{4\tau}\Big(2L^2(m-1) + \frac{2}{\tau^2}\Big)^{-1}\big(\m F(\rho^k) - \m F(\rho^\ast)\big)\\
&=\lambda C_2(m, L, \tau)\big(\m F(\rho^k) - \m F(\rho^\ast)\big).
\end{align*}
Therefore, we have
\begin{align*}
\frac{\lambda}{2}\W_2^2(\rho^k, \rho^\ast) \leq \m F(\rho^k, \rho^\ast) \leq \big[1 + \lambda C_2(m, L, \tau)\big]^{-k}\big(\m F(\rho^0) - \m F(\rho^\ast)\big),
\end{align*}
where the first inequality is due to Assumption D ($\lambda$-QG condition). This ends the proof.

\subsection{Random Update Scheme: Proof of Theorem \ref{thm: randomized}}
For any fixed $k\in\mb N$, let $\m A_k^{T}(M)$ be the event such that all coordinates are updated at least once but at most $T$ times in the $k$-th iteration. Take
\begin{align*}
    M_\delta = \Big\lceil m\log\frac{m}{\delta}\Big\rceil \quad\mx{and}\quad T_{\delta} = \Big\lceil e\log\frac{m}{\delta}\Big\rceil.
\end{align*}
The following lemma controls the probability of $\m A_k^{T_\delta}(M_\delta)^c$, whose proof will be postponed to Appendix~\ref{app: additional_proof}.
\begin{lemma}\label{lem: covering}
For any $\delta > 0$ such that $0 <\delta < 1$, define
\begin{align*}
    M = \Big\lceil m\log\frac{m}{\delta}\Big\rceil \quad\mx{and}\quad T = \Big\lceil e\log\frac{m}{\delta}\Big\rceil.
\end{align*}
$j_0, \dots, j_{M-1}$ are i.i.d. random variables sampled from $\text{Uniform}\{1, \dots, m\}$. Then, there is probability at least $1-2\delta$, such that every integer in $\{1, 2, \dots, m\}$ appears at least once but at most $T$ times in $\{j_l\}_{l=0}^{M-1}$.
\end{lemma}
\noindent From the above lemma, we know $\mb P(\m A_t^{T_{\delta}}(M_\delta)^c) \leq 2\delta$. Note that
\begin{align*}
\mb E\big[\m F(\rho^{k}) - \m F(\rho^\ast)\,\big|\, \rho^{k-1}\big]
&= \mb E\big[\m F(\rho^{k}) - \m F(\rho^\ast)\,\big|\, \m A_k^{T_{\delta}}(M_\delta), \rho^{k-1}\big]\cdot\mb P(\m A_k^{T_{\delta}}(M_\delta))\\
&\qquad\qquad + \mb E\big[\m F(\rho^{k}) - \m F(\rho^\ast)\,\big|\, \m A_k^{T_{\delta}}(M_\delta)^c, \rho^{k-1}\big]\cdot\mb P(\m A_k^{T_{\delta}}(M_\delta)^c).
\end{align*}
The second term can be bounded by $2\delta\big(\m F(\rho^{k-1}) - \m F(\rho^\ast)\big)$ since $\m F(\rho^k)$ is non-increasing with respect to $k$. Next, we will bound the first term with the following two lemmas, whose proof will be postponed to Appendix~\ref{app: additional_proof}.

\begin{lemma}[sufficient decrease]\label{lem: sufficient_decrease_randomized}
For any $k, M\in\mb Z_+$, we have
\begin{align*}
    \m F(\rho^{k-1}) - \m F(\rho^{k}) \geq \frac{1}{2\tau}\sum_{l=0}^{M-1}\W_2^2(\rho^{k-1, l}, \rho^{k-1, l+1}).
\end{align*}
\end{lemma}

\begin{lemma}[bound of functional value]\label{lem: bound_func_value_randomized}
For any $k\in \mb Z_+$, let $M\in\mb Z_+$ be an integer such that all coordinates are updated at least once but at most $T$ times through $M$ updates in the $k$-th iteration. Under Assumptions A, B, C, and D, we have
\begin{align*}
\m F(\rho^k) - \m F(\rho^\ast) \leq \frac{4}{\lambda}\Big(L^2mT + \frac{1}{\tau^2}\Big)\sum_{l=0}^{M-1}\W_2^2(\rho^{k-1, l}, \rho^{k-1, l+1}).
\end{align*}
\end{lemma}

Under $\m A_k^{T_{\delta}}(M_\delta)$, we have
\begin{align*}
\m F(\rho^{k-1}) - \m F(\rho^k) &\stackrel{\text{Lem~\ref{lem: sufficient_decrease_randomized}}}{\geq} \frac{1}{2\tau}\sum_{l=0}^{M-1}\W_2^2(\rho^{k-1, l}, \rho^{k-1, l+1})\\
&\stackrel{\text{Lem~\ref{lem: bound_func_value_randomized}}}{\geq} \frac{\lambda}{8\tau}\Big(L^2mT_\delta + \frac{1}{\tau^2}\Big)^{-1}\big(\m F(\rho^k) - \m F(\rho^\ast)\big).
\end{align*}
This implies
\begin{align*}
    \mb E\big[\m F(\rho^{k}) - \m F(\rho^\ast)\,\big|\, \m A_k^{T_{\delta}}(M_\delta), \rho^{k-1}\big] \leq \Big(1 + \frac{\lambda\tau}{8(1 + \tau^2L^2mT_{\delta})}\Big)^{-1}\Big(\m F(\rho^{k-1}) - \m F(\rho^\ast)\Big).
\end{align*}
Combining all the pieces above yields
\begin{align*}
    \mb E\big[\m F(\rho^{k}) - \m F(\rho^\ast)\,\big|\, \rho^{k-1}\big] \leq \bigg[\Big(1 + \frac{\lambda\tau}{8(1 + \tau^2L^2mT_{\delta})}\Big)^{-1} + 2\delta\bigg]\big(\m F(\rho^{k-1}) - \m F(\rho^\ast)\big)
\end{align*}
Taking the expectation with respect to $\rho^{k-1}$, we have
\begin{align*}
    \mb E\big[\m F(\rho^{k}) - \m F(\rho^\ast)\big] \leq \bigg[\Big(1 + \frac{\lambda\tau}{8(1 + \tau^2L^2mT_{\delta})}\Big)^{-1} + 2\delta\bigg]\mb E\big[\m F(\rho^{k-1}) - \m F(\rho^\ast)\big].
\end{align*}

Now, let us choose a proper $\delta$ to derive the convergence rate. Taking $\delta = 1/(mL^2)$, we have $T_\delta = 2e\log(mL)$ and $M_\delta = \big\lceil2m\log(mL)\big\rceil$. Therefore, we have
\begin{align*}
    \mb E\big[\m F(\rho^{k}) - \m F(\rho^\ast)\big] \leq C_3(m, L, \tau)\mb E\big[\m F(\rho^{k-1}) - \m F(\rho^\ast)\big]
\end{align*}
where
\begin{align*}
    C_3(m, L, \tau) = \bigg(1 + \frac{\lambda\tau}{8[1+2e\tau^2L^2m\log(mL)]}\bigg)^{-1} + \frac{2}{mL^2},
\end{align*} 
Note that this is a decreasing algorithm, meaning that $\m F(\rho^{k}) \leq \m F(\rho^{k-1})$. Therefore,
\begin{align*}
\mb E\big[\m F(\rho^{k}) - \m F(\rho^\ast)\big] \leq \min\big\{C_3(m, L, \tau), 1\big\}\mb E\big[\m F(\rho^{k-1}) - \m F(\rho^\ast)\big].
\end{align*}
Taking $\tau^{-1} = \sqrt{2emL^2\log(mL)}$ yields
\begin{align*}
    C_3 = \bigg(1 + \frac{\lambda}{2\sqrt{2eL^2m\log(mL)}}\bigg)^{-1} + \frac{2}{mL^2}.
\end{align*}

\subsection{Case $\lambda = 0$: Proof of Theorem~\ref{thm: convex_case}}
\underline{Parallel and sequential update schemes.} First, let us show there is a constant $C$ depending on $m, L, \tau$ and the update scheme, such that
\begin{align*}
\m F(\rho^k) - \m F(\rho^\ast) \leq C\sqrt{\m F(\rho^{k-1}) - \m F(\rho^k)}\W_2(\rho^k, \rho^\ast).
\end{align*}
In fact, we have
\begin{align*}
&\quad\,\m F(\rho^k) - \m F(\rho^\ast ) \stackrel{\ri}{\leq} C\W_2(\rho^k, \rho^{k-1})\W_2(\rho^k, \rho^\ast ) \stackrel{\rii}{\leq} C\sqrt{\m F(\rho^{k-1}) - \m F(\rho^k)}\cdot \W_2(\rho^k, \rho^\ast ),
\end{align*}
where $C$ is a constant varying from line to line.
Here, (i) is by the inequality~\eqref{eqn: upper_bound_functional_value2} in both the parallel and the sequential update schemes; (ii) is by Lemma~\ref{lem: sufficient_decrease_cond} in the parallel update scheme and Lemma~\ref{lem: sufficient_decrease_sequential} in the sequential update scheme.

The above inequality implies
\begin{align*}
    \m F(\rho^{k-1}) - \m F(\rho^\ast ) &\geq \Big(\m F(\rho^k) - \m F(\rho^\ast )\Big) + \frac{1}{C\W_2^2(\rho^k, \rho^\ast )}\cdot\Big(\m F(\rho^k) - \m F(\rho^\ast )\Big)^2\\
    &\geq \Big(\m F(\rho^k) - \m F(\rho^\ast)\Big) + \frac{1}{CD_{\m X}^2}\cdot\Big(\m F(\rho^k) - \m F(\rho^\ast )\Big)^2.
\end{align*}
In the last inequality, we use the fact that
\begin{align*}
    \W_2^2(\rho^t, \rho^\ast ) = \inf_{X\sim\rho^k, Y\sim\rho^\ast} \mb E\|X-Y\|^2 \leq D_{\m X}^2.
\end{align*}
Therefore, we have
\begin{align*}
\frac{1}{\m F(\rho^k) - \m F(\rho^\ast  )} - \frac{1}{\m F(\rho^{k-1}) - \m F(\rho^\ast )} &\geq \frac{1}{CD_{\m X}^2 + \m F(\rho^k) - \m F(\rho^\ast )}
\geq \frac{1}{CD_{\m X}^2 + \m F(\rho^0) - \m F(\rho^\ast  )},
\end{align*}
where the last inequality is due the non-increasing property of $\m F$ derived by Lemma~\ref{lem: sufficient_decrease_cond} and Lemma~\ref{lem: sufficient_decrease_sequential}. This implies
\begin{align*}
    \m F(\rho^k) - \m F(\rho^\ast ) \leq \frac{CD_{\m X}^2 + \m F(\rho^0) - \m F(\rho^\ast )}{k}.
\end{align*}

\noindent\underline{Random update scheme.} In the random update scheme, let $\m A_k^{T}(M)$ be the event such that all coordinates are updated at least once but at most $T$ times in the $k$-th iteration. We have $\mb P(\m A_k^{T_\delta}(M_\delta)^c) \leq 2\delta$ by taking
\begin{align*}
    M_\delta = \Big\lceil m\log\frac{m}{\delta}\Big\rceil \quad\mx{and}\quad T_{\delta} = \Big\lceil e\log\frac{m}{\delta}\Big\rceil.
\end{align*}
Similar to the proof of Theorem~\ref{thm: randomized}, we consider the decomposition
\begin{align*}
\mb E\big[\m F(\rho^{k}) - \m F(\rho^\ast)\,\big|\, \rho^{k-1}\big]
&= \mb E\big[\m F(\rho^{k}) - \m F(\rho^\ast)\,\big|\, \m A_k^{T_{\delta}}(M_\delta), \rho^{k-1}\big]\cdot\mb P(\m A_k^{T_{\delta}}(M_\delta))\\
&\qquad\qquad + \mb E\big[\m F(\rho^{k}) - \m F(\rho^\ast)\,\big|\, \m A_k^{T_{\delta}}(M_\delta)^c, \rho^{k-1}\big]\cdot\mb P(\m A_k^{T_{\delta}}(M_\delta)^c).
\end{align*}
By~\eqref{eqn: upper_bound_functional_value2_random} and Lemma~\ref{lem: sufficient_decrease_randomized}, the first term is bounded by
\begin{align*}
\sqrt{2\tau\Big(2L^2mT_\delta + \frac{2}{\tau^2}\Big)}\cdot\m D_{\m X}\sqrt{\m F(\rho^{k-1}) - \m F(\rho^k)}.
\end{align*}
Since $\m F(\rho^k)$ is non-increasing, the second term can be bounded by $2\delta\big(\m F(\rho^{k-1}) - \m F(\rho^\ast)\big)$. Thus, we have
\begin{align}\label{eqn: recurrsive_random}
\begin{aligned}
&\quad\,\mb E\big[\m F(\rho^k) - \m F(\rho^\ast)\big]\\
&\leq \sqrt{2\tau\Big(2L^2mT_\delta + \frac{2}{\tau^2}\Big)}\cdot\mb E\m D_{\m X}\sqrt{\m F(\rho^{k-1}) - \m F(\rho^k)} + 2\delta\mb E\big[\m F(\rho^{k-1}) - \m F(\rho^\ast)\big]\\
&\leq \sqrt{2\tau\Big(2L^2mT_\delta + \frac{2}{\tau^2}\Big)}\cdot\m D_{\m X}\sqrt{\mb E\big[\m F(\rho^{k-1}) - \m F(\rho^k)\big]} + 2\delta\mb E\big[\m F(\rho^{k-1}) - \m F(\rho^\ast)\big],
\end{aligned}
\end{align}
where the last inequality is due to Cauchy--Schwarz inequality. 

Now, let us prove $\mb E\big[\m F(\rho^k) - \m F(\rho^\ast)\big] \lesssim \frac{1}{k}$. For simplicity, assume $4\delta \leq 1$ and let
\begin{align*}
x_k = \mb E\big[\m F(\rho^k) - \m F(\rho^\ast)\big], \quad A_\delta = D_{\m X}\sqrt{2\tau\Big(2L^2mT_\delta + \frac{2}{\tau^2}\Big)}, \quad\mx{and}\quad \tilde A_\delta = \frac{A_\delta}{1-4\delta}.
\end{align*}
We will prove $x_k \leq (\tilde A_\delta^2 + x_0)/k$ for all $k\geq 1$ by induction. $k=1$ is obvious since $x_1 \leq x_0$. For $k\geq 2$, note that~\eqref{eqn: recurrsive_random} implies
\begin{align}\label{eqn: recurrsive}
    x_k \leq A_\delta\sqrt{x_{k-1}-x_k} + 2\delta x_{k-1}.
\end{align}
\emph{Case 1: $x_{k-1} \leq 2x_k$.} \eqref{eqn: recurrsive} implies $x_k \leq \tilde A_\delta\sqrt{x_{k-1}-x_k}$. Thus, we have
\begin{align*}
    \frac{1}{x_k} \geq \frac{1}{x_k+\tilde A_\delta^2} + \frac{1}{x_{k-1}} \stackrel{\ri}{\geq} \frac{1}{x_0 + \tilde A_\delta^2} + \frac{k-1}{\tilde A_\delta^2 + x_0} = \frac{k}{x_0 + \tilde A_\delta^2}.
\end{align*}
Here, (i) is by $x_k \leq x_0$ and the induction hypothesis.

\noindent\emph{Case 2: $x_{k-1} > 2x_k$.} In this case, we have
\begin{align*}
x_k < \frac{x_{k-1}}{2} \stackrel{\ri}{\leq} \frac{\tilde A_\delta^2 + x_0}{2(k-1)} \stackrel{\rii}{\leq} \frac{\tilde A_\delta^2 + x_0}{k}.
\end{align*}
Here, (i) is due to the induction hypothesis, and (ii) is by $k\geq 2$. By induction, we have $x_k \leq (\tilde A_\delta^2 + x_0) / k$. We finish the proof.

\subsection{Inexact WPCG-P with $\lambda > 0$: Proof of Theorem~\ref{thm: inexactWPCG_QG}}
The proof is similar to the exact WPCG, but we need to take the cumulative error into consideration. Some analogous lemmas as in the proof of Theorem~\ref{thm: main_thm} are required, the proofs of which are postponed to Appendix~\ref{app: additional_proof}.

\begin{lemma}[sufficient decrease for inexact WPCG-P]\label{lem: inexact_suff_decrease}
Under Assumptions A, B, and C, when the step size satisfies $0 < \tau < \frac{1}{2L\sqrt{m-1}}$, we have
\begin{align*}
\m F(\wt\rho^k) - \m F(\wt\rho^{k+1}) \geq \Big(\frac{1}{2\tau} - L\sqrt{m-1}\Big)\W_2^2(\wt\rho^{k+1}, \wt\rho^k) - \frac{\|\varepsilon^{k+1}\|^2}{2\tau}.
\end{align*}
\end{lemma}

\begin{lemma}[bound of functional value for inexact WPCG-P]\label{lem: inexact_func_upperbd}
Under Assumptions A, B, C, and D, we have
\begin{align*}
\m F(\wt\rho^k) - \m F(\rho^\ast) \leq \frac{6}{\lambda}\Big[\Big(L^2(m-1) + \frac{1}{\tau^2}\Big)\W_2^2(\wt\rho^k, \wt\rho^{k-1}) + \frac{\|\varepsilon^k\|^2}{\tau^2}\Big].
\end{align*}
\end{lemma}

With the above two lemmas, we have
\begin{align*}
\frac{\lambda}{6}\big[\m F(\wt\rho^k) - \m F(\rho^\ast)\big] - \frac{\|\varepsilon^k\|^2}{\tau^2} 
&\stackrel{\ri}{\leq} \Big(L^2(m-1) + \frac{1}{\tau^2}\Big)\W_2^2(\wt\rho^k, \wt\rho^{k-1})\\
&\stackrel{\rii}{\leq} \frac{L^2(m-1) + \frac{1}{\tau^2}}{\frac{1}{2\tau} - L\sqrt{m-1}}\Big[\m F(\wt\rho^{k-1}) - \m F(\wt\rho^{k}) + \frac{\|\varepsilon^k\|^2}{2\tau}\Big].
\end{align*}
Here (i) is by Lemma~\ref{lem: inexact_func_upperbd}, and (ii) is by Lemma~\ref{lem: inexact_suff_decrease}.
Reorganizing the inequality yields
\begin{align*}
\bigg[\frac{\lambda}{6} + \frac{L^2(m-1) + \frac{1}{\tau^2}}{\frac{1}{2\tau} - L\sqrt{m-1}}\bigg]\big[\m F(\wt\rho^k) - \m F(\rho^\ast)\big] 
&\leq  \frac{L^2(m-1) + \frac{1}{\tau^2}}{\frac{1}{2\tau} - L\sqrt{m-1}}\big[\m F(\wt\rho^{k-1}) - \m F(\rho^\ast)\big]\\
&\qquad\qquad + \bigg[\frac{L^2(m-1) + \frac{1}{\tau^2}}{1 - 2\tau L\sqrt{m-1}} + \frac{1}{\tau^2}\bigg]\|\varepsilon^k\|^2.
\end{align*}
Applying Lemma~\ref{lem: seq_QG_inexact} with
\begin{align*}
A = \Big(1 + \frac{\lambda}{6}\cdot \frac{\frac{1}{2\tau} - L\sqrt{m-1}}{L^2(m-1) + \frac{1}{\tau^2}}\Big)^{-1},\quad
B = \frac{\frac{L^2(m-1) + \frac{1}{\tau^2}}{1 - 2\tau L\sqrt{m-1}} + \frac{1}{\tau^2}}{\frac{L^2(m-1) + \frac{1}{\tau^2}}{\frac{1}{2\tau} - L\sqrt{m-1}} + \frac{\lambda}{6}},\quad\mbox{and}\quad
\xi_k = \|\varepsilon^k\|^2
\end{align*}
yields the convergence result.

\subsection{Inexact WPCG-P with $\lambda > 0$: Proof of Theorem~\ref{thm: inexactWPCG_conv}}
Recall that we have Equation~\eqref{eqn: func_value_uppbd}
\begin{align*}
\m F(\wt\rho^k) - \m F(\rho^\ast) 
&\leq \sqrt{\Big(3L^2(m-1) + \frac{3}{\tau^2}\Big)\W_2^2(\wt\rho^k, \wt\rho^{k-1}) + \frac{3\|\varepsilon^k\|^2}{\tau^2}} \cdot \W_2(\wt\rho^k, \rho^\ast)\\
&\leq \sqrt{3}\W_2(\wt\rho^k, \rho^\ast)\bigg[\frac{\|\varepsilon^k\|}{\tau} + \sqrt{L^2(m-1) + \frac{1}{\tau^2}}\W_2(\wt\rho^k, \wt\rho^{k-1})\bigg]\\
&\leq \sqrt{3}D_{\m X}\bigg[\frac{\|\varepsilon^k\|}{\tau} + \sqrt{\frac{L^2(m-1) + \frac{1}{\tau^2}}{\frac{1}{2\tau} - L\sqrt{m-1}}\cdot\Big(\m F(\wt\rho^{k-1}) - \m F(\wt\rho^{k}) + \frac{\|\varepsilon^k\|^2}{2\tau}\Big)}\bigg].
\end{align*}
Here, we apply Lemma~\ref{lem: inexact_suff_decrease} and the fact that $\W_2(\wt\rho^k,\rho^\ast)\leq D_{\m X}$. If $\m F(\wt\rho^k) - \m F(\rho^\ast) \geq \sqrt{3}D_{\m X}\frac{\|\varepsilon^k\|}{\tau}$, we have
\begin{align*}
\Big(\frac{\m F(\wt\rho^k) - \m F(\rho^\ast)}{\sqrt{3}D_{\m X}} - \frac{\|\varepsilon^k\|}{\tau}\Big)^2 
\leq \frac{L^2(m-1) + \frac{1}{\tau^2}}{\frac{1}{2\tau} - L\sqrt{m-1}}\Big(\m F(\wt\rho^{k-1}) - \m F(\wt\rho^{k}) + \frac{\|\varepsilon^k\|^2}{2\tau}\Big),
\end{align*}
otherwise
\begin{align*}
\Big(\frac{\m F(\wt\rho^k) - \m F(\rho^\ast)}{\sqrt{3}D_{\m X}} - \frac{\|\varepsilon^k\|}{\tau}\Big)^2 
\leq \frac{\|\varepsilon^k\|^2}{\tau^2}. 
\end{align*}
Therefore, we have
\begin{align*}
\Big(\frac{\m F(\wt\rho^k) - \m F(\rho^\ast)}{\sqrt{3}D_{\m X}} - \frac{\|\varepsilon^k\|}{\tau}\Big)^2 
\leq \frac{L^2(m-1) + \frac{1}{\tau^2}}{\frac{1}{2\tau} - L\sqrt{m-1}}\Big(\m F(\wt\rho^{k-1}) - \m F(\wt\rho^{k}) + \frac{\|\varepsilon^k\|^2}{2\tau}\Big) + \frac{\|\varepsilon^k\|^2}{\tau^2},
\end{align*}
i.e.
\begin{align*}
&\frac{\frac{1}{2\tau} - L\sqrt{m-1}}{L^2(m-1) + \frac{1}{\tau^2}}\cdot\frac{[\m F(\wt\rho^k) - \m F(\rho^\ast)]^2}{3D_{\m X}^2} + \Big[1 - \frac{\frac{1}{2\tau} - L\sqrt{m-1}}{L^2(m-1) + \frac{1}{\tau^2}}\cdot \frac{2\|\varepsilon^k\|}{\sqrt{3}\tau D_{\m X}}\Big]\cdot\big[\m F(\wt\rho^k) - \m F(\rho^\ast)\big]\\
&\qquad\qquad\qquad\qquad\qquad\leq \big[\m F(\wt\rho^{k-1}) - \m F(\rho^\ast)\big] 
+ \frac{\|\varepsilon^k\|^2}{2\tau}.
\end{align*}
Applying Lemma~\ref{lem: seq_convex_inexact} with
\begin{align*}
A = \frac{\frac{1}{2\tau} - L\sqrt{m-1}}{L^2(m-1) + \frac{1}{\tau^2}}\cdot\frac{1}{3D_{\m X}^2},\quad\!
B = \frac{\frac{1}{2\tau} - L\sqrt{m-1}}{L^2(m-1) + \frac{1}{\tau^2}} \cdot\frac{2}{\sqrt{3}\tau D_{\m X}},\quad\!
C = \frac{1}{2\tau},
\!\quad\mbox{and}\quad\!
\xi_k = \|\varepsilon^k\|
\end{align*}
yields the convergence result.

\section{Proof of Lemmas in the Main Theorems}\label{app: additional_proof}
To prove the two steps mentioned in the sketch at the beginning of Appendix~\ref{app: main proof}, the most important result is the following first-order optimality condition (FOC). The proof, which will be postponed to Appendix~\ref{app: proof_FOC}, is similar to Proposition 8.7 in~\citep{santambrogio2015optimal} for $h(x) = x\log x$ and $W = 0$.

\begin{lemma}[FOC]\label{lem: ot_map}
Consider the Wasserstein proximal gradient scheme
\begin{align*}
\rho_\tau = \argmin_{\mu\in\ms P_2^r(\m X)}\int_{\m X} U(x)\,\dd\mu + \int_{\m X} h(\mu) +  \int\!\!\!\int_{\m X\times\m X} W(x, y)\,\dd\mu(x)\dd\mu(y) + \frac{1}{2\tau}\W_2^2(\rho, \mu).
\end{align*}
Under Assumption A, $\rho_\tau$ satisfies
\begin{align*}
    T_{\rho_\tau}^\rho - \id = \tau\bigg(\nabla U + \nabla h'(\rho_\tau) + \nabla\!\int_{\m X}W(\cdot, y)\,\dd\rho_\tau(y) + \nabla\!\int_{\m X}W(y, \cdot)\,\dd\rho_\tau(y)\bigg)
\end{align*}
\end{lemma}

For simplicity, in the following proof, define
\begin{align*}
\m H(\rho) = \sum_{j=1}^m\m H_j(\rho_j) \quad\mbox{and}\quad \m W(\rho) = \sum_{j=1}^m\m W_j(\rho_j)
\end{align*}
for $\rho = (\rho_1, \dots, \rho_m)$. Also, let
\begin{align*}
    \nabla_{\W_2}\m W_j(\rho_j) = \nabla\!\int_{\m X_j}W_j(\cdot, y)\,\dd\rho_j(y) + \nabla\!\int_{\m X_j}W_j(y, \cdot)\,\dd\rho_j(y)
\end{align*}
be the Wasserstein gradient of $\m W_j$ at $\rho_j$ with respect to $\W_2$ metric. Note that this Wasserstein gradient is a function on $\m X_j$.

\subsection{Proof of Lemma~\ref{lem: sufficient_decrease_cond}}

For simplicity, define
\begin{align*}
V_j^k = \int_{\m X_{-j}}V(x_j, x_{-j})\,\dd \rho_{-j}^k(x_{-j}),
\quad\mbox{and}\quad
V_{-j} = \int_{\m X_j}V(x_j, x_{-j})\,\dd \rho_j^k(x_j).
\end{align*}
\underline{Step 1.} First, let us show
\begin{align}\label{eqn: upper_bound_sumV}
    \sum_{j=1}^m\m V(\rho_{j}^{k+1}\otimes \rho_{-j}^k) \leq m\m V(\rho^k) + \m H(\rho^k) + \m W(\rho^k) - \m H(\rho^{k+1}) - \m W(\rho^{k+1}) - \frac{1}{2\tau}\W_2^2(\rho^{k+1}, \rho^k).
\end{align}
In fact, by definition of $\rho_j^{k+1}$, we have
\begin{align*}
    \m V(\rho_j^{k+1}\otimes \rho_{-j}^{k}) + \m H_j(\rho_j^{k+1}) + \m W_j(\rho_j^{k+1}) + \frac{1}{2\tau}\W_2^2(\rho_j^{k+1}, \rho_j^{k}) \leq \m V(\rho^{k}) + \m H_j(\rho_j^{k}) + \m W_j(\rho_j^k).
\end{align*}
Summing from $j=1$ to $m$ yields equation (\ref{eqn: upper_bound_sumV}).

\vspace{0.5em}
\noindent\underline{Step 2.} Next, we will show that
\begin{align}\label{eqn: lower_bound_sumV}
\sum_{j=1}^m\m V(\rho_{j}^{k+1}\otimes \rho_{-j}^k) \geq m\m V(\rho^{k+1}) + (m-1)\sum_{j=1}^m\int_{\m X_j}\big\langle\nabla V_j^{k+1}, T_{\rho_j^{k+1}}^{\rho_j^k}-\id\big\rangle\,\dd\rho_j^{k+1}.
\end{align}
To prove this, by convexity of $\m V$ we have
\begin{align*}
\m V(\rho_j^{k+1}\otimes \rho_{-j}^k) 
&\stackrel{\ri}{\geq} \m V(\rho_j^{k+1}\otimes\rho_{-j}^{k+1}) + \int_{\m X_{-j}}\bigg\langle \nabla \frac{\delta\m V(\rho_j^{k+1}\otimes\cdot)}{\delta\rho_{-j}}(\rho_{-j}^{k+1}), T_{\rho_{-j}^{k+1}}^{\rho_{-j}^k}-\id\bigg\rangle\,\dd\rho_{-j}^{k+1}\\
&\stackrel{\rii}{=} \m V(\rho^{k+1}) + \sum_{i\neq j}\int_{\m X_{-j}}\big\langle \nabla_i V_{-j}^{k+1}, T_{\rho_i^{k+1}}^{\rho_i^{k}}-\id\big\rangle\,\dd\rho_{-j}^{k+1}\\
&= \m V(\rho^{k+1}) + \sum_{i\neq j}\int_{\m X_i}\big\langle \nabla V_i^{k+1}, T_{\rho_i^{k+1}}^{\rho_i^k}-\id\big\rangle\,\dd\rho_i^{k+1}
\end{align*}
Here, (i) is by Equation~\eqref{eqn: convexity}; (ii) is by
\begin{align*}
\nabla\frac{\delta\m V(\rho_j^{k+1}\otimes\cdot)}{\delta\rho_{-j}}(\rho_{-j}^{k+1}) 
&=  \nabla V_{-j}^{k+1} 
= \big(\nabla_1 V_{-j}^{k+1}, \dots, \nabla_{j-1} V_{-j}^{k+1}, \nabla_{j+1} V_{-j}^{k+1},\dots, \nabla_m V_{-j}^{k+1}\big).
\end{align*}
Summing the inequality above from $j=1$ to $m$ yields equation (\ref{eqn: lower_bound_sumV}).

\vspace{0.5em}
\noindent\underline{Step 3.} Let us show
\begin{align}\label{eqn: lower_bound_inner_prod_gradV}
\begin{aligned}
&\quad\,\sum_{j=1}^m\int_{\m X_j}\big\langle\nabla V_j^{k+1}, T_{\rho_j^{k+1}}^{\rho_j^k}-\id\big\rangle\,\dd\rho_j^{k+1}\\ 
&\geq \big(\tau^{-1}-L\sqrt{m-1}\big)\W_2^2(\rho^{k+1}, \rho^k) - \Big(\m H(\rho^k) - \m H(\rho^{k+1})\Big) - \Big(\m W_j(\rho_j^k) - \m W_j(\rho_j^{k+1})\Big).
\end{aligned}
\end{align}
By Lemma \ref{lem: ot_map}, we have $T_{\rho_j^{k+1}}^{\rho_j^k}-\id = \tau\nabla V_j^k + \tau\nabla h_j'(\rho_j^{k+1}) + \tau \nabla_{\W_2}\m W_j(\rho_j^{k+1})$. 
This implies
\begin{align*}
&\quad\, \sum_{j=1}^m\int_{\m X_j}\big\langle\nabla V_j^{k+1}, T_{\rho_j^{k+1}}^{\rho_j^k}-\id\big\rangle\,\dd\rho_j^{k+1}\\
&= \sum_{j=1}^m\int_{\m X_j} \bigg\langle\frac{T_{\rho_j^{k+1}}^{\rho_j^k}-\id}{\tau} - \big(\nabla V_j^k + \nabla h_j'(\rho_j^{k+1}) + \nabla_{\W_2}\m W_j(\rho_j^{k+1})\big), T_{\rho_j^{k+1}}^{\rho_j^k}-\id\bigg\rangle\\
&\qquad\qquad\qquad\qquad\qquad\qquad\qquad\qquad\qquad\qquad\qquad\qquad\qquad  + \big\langle\nabla V_j^{k+1}, T_{\rho_j^{k+1}}^{\rho_j^k}-\id\big\rangle \,\dd\rho_j^{k+1}\\
&=\sum_{j=1}^m\bigg[ \int_{\m X_j}\Big\langle\nabla V_j^{k+1}-\nabla V_j^k, T_{\rho_j^{k+1}}^{\rho_j^k}-\id\Big\rangle\,\dd\rho_j^{k+1} - \int_{\m X_j}\Big\langle\nabla h_j'(\rho_j^{k+1}), T_{\rho_j^{k+1}}^{\rho_j^k}-\id\Big\rangle\,\dd\rho_j^{k+1}\\
&\qquad\qquad\qquad\qquad\qquad\qquad - \int_{\m X_j}\Big\langle\nabla_{\W_2}\m W_j(\rho_j^{k+1}), T_{\rho_j^{k+1}}^{\rho_j^k}-\id\Big\rangle\,\dd\rho_j^{k+1} + \frac{1}{\tau}\W_2^2(\rho_j^{k+1}, \rho_j^k)\bigg]\\
&\stackrel{\ri}{\geq} \sum_{j=1}^m\bigg[\int_{\m X_j}\Big\langle\nabla V_j^{k+1}-\nabla V_j^k, T_{\rho_j^{k+1}}^{\rho_j^k}-\id\Big\rangle\,\dd\rho_j^{k+1} - \Big(\m H_j(\rho_j^k) - \m H_j(\rho_j^{k+1})\Big)\\ &\qquad\qquad\qquad\qquad\qquad\qquad\qquad\qquad\qquad\qquad - \Big(\m W_j(\rho_j^k) - \m W_j(\rho_j^{k+1})\Big)\bigg] + \frac{1}{\tau}\W_2^2(\rho^{k+1}, \rho^k)\\
&\stackrel{\rii}{\geq} -\sum_{j=1}^m\W_2(q_j^{k+1}, q_j^k)\cdot\Big\|\nabla V_j^{k+1}-\nabla V_j^k\Big\|_{L^2(\rho_j^{k+1}; \m X_j)} - \Big(\m H(\rho^k) - \m H(\rho^{k+1})\Big)\\
&\qquad\qquad\qquad\qquad\qquad\qquad\qquad\qquad\qquad - \Big(\m W(\rho^k) - \m W(\rho^{k+1})\Big) + \frac{1}{\tau}\W_2^2(\rho^{k+1}, \rho^k)\\
&\stackrel{\riii}{\geq} -\sum_{j=1}^m\W_2(\rho_j^{k+1}, \rho_j^k)\cdot L\W_2(\rho_{-j}^{k+1}, \rho_{-j}^k) - \Big(\m H(\rho^k) - \m H(\rho^{k+1})\Big)\\
&\qquad\qquad\qquad\qquad\qquad\qquad\qquad\qquad\qquad - \Big(\m W(\rho^k) - \m W(\rho^{k+1})\Big) + \frac{1}{\tau}\W_2^2(\rho^{k+1}, \rho^k)\\
&\stackrel{\textrm{(iv)}}{\geq} -L\sum_{j=1}^m \bigg[\frac{\sqrt{m-1}}{2}\W_2^2(\rho_j^{k+1}, \rho_j^k) + \frac{1}{2\sqrt{m-1}}\W_2^2(\rho_{-j}^{k+1}, \rho_{-j}^k)\bigg]\\
&\qquad\qquad\qquad\qquad\qquad - \Big(\m H(\rho^k) - \m H(\rho^{k+1})\Big) - \Big(\m W(\rho^k) - \m W(\rho^{k+1})\Big)+ \frac{1}{\tau}\W_2^2(\rho^{k+1}, \rho^k)\\
&= \big(\tau^{-1}-L\sqrt{m-1}\big)\W_2^2(\rho^{k+1}, \rho^k) - \Big(\m H(\rho^k) - \m H(\rho^{k+1})\Big) - \Big(\m W(\rho^k) - \m W(\rho^{k+1})\Big).
\end{align*}
Here, (i) is due to the fact that both $\m H_j$ and $\m W_j$ are convex along geodesics; (ii) is by Cauchy--Schwarz inequality; (iii) is by Lemma \ref{lem: diff_potential_w2}; (iv) is by AM-GM inequality.

\vspace{0.5em}
\noindent\underline{Step 4.} Finally, let us prove the statement. By Equations~\eqref{eqn: upper_bound_sumV}, \eqref{eqn: lower_bound_sumV}, and \eqref{eqn: lower_bound_inner_prod_gradV}, we have
\begin{align*}
&\quad\,m\m V(\rho^k) + \m H(\rho^k) - \m H(\rho^{k+1}) + \m W(\rho^k) - \m W(\rho^{k+1}) - \frac{1}{2\tau}\W_2^2(\rho_j^{k+1}, \rho_j^k) 
\geq\sum_{j=1}^m\m V(\rho_{-j}^{k+1}\otimes \rho_j^k)\\
&\geq m\m V(\rho^{k+1}) + (m-1)\sum_{j=1}^m\int_{\m X_j}\big\langle\nabla V_j^{k+1}, T_{\rho_j^{k+1}}^{\rho_j^k}-\id\big\rangle\,\dd\rho_j^{k+1}\\
&\geq m\m V(\rho^{k+1}) + (m-1)\Big[\big(\tau^{-1}-L\sqrt{m-1}\big)\W_2^2(\rho^{k+1}, \rho^k) - \Big(\m H(\rho^k) - \m H(\rho^{k+1})\Big)\\
&\qquad\qquad\qquad\qquad\qquad\qquad\qquad\qquad\qquad\qquad\qquad\qquad\qquad  - \Big(\m W(\rho^k) - \m W(\rho^{k+1})\Big)\Big].
\end{align*}
This implies
\begin{align*}
    \m F(\rho^k) - \m F(\rho^{k+1}) \geq \frac{1}{m}\bigg[\frac{1}{2\tau} + (m-1)\Big(\frac{1}{\tau} - L\sqrt{m-1}\Big)\bigg]\W_2^2(\rho^{k+1}, \rho^{k}).
\end{align*}
Substituting $k$ with $k-1$ yields the result.

\subsection{Proof of Lemma~\ref{lem: bound_func_value}}
To prove the result, we need the following lemma.
\begin{lemma}\label{lem: relative_error_cond}
Under Assumptions A and B, for both the parallel update scheme and the sequential update scheme, we have
\begin{align*}
&\sum_{j=1}^m \int_{\m X_j}\big\|\nabla V_j^k + \nabla h_j'(\rho_j^k) + \nabla_{\W_2}\m W_j(\rho_j^k)\big\|^2\,\dd\rho_j^k \leq \Big(2L^2(m-1) + \frac{2}{\tau^2}\Big) \W_2^2(\rho^k, \rho^{k-1}).
\end{align*}
\end{lemma}

\noindent Then, by convexity of $\m F$, we have
\begin{align}
\m F(\rho^k) - \m F(\rho^\ast)
&\leq -\int_{\m X}\Big\langle\nabla\frac{\delta\m F}{\delta\rho}(\rho^k), T_{\rho^k}^{\rho^\ast}-\id\Big\rangle\,\dd\rho^k \nonumber\\
&= - \sum_{j=1}^m\int_{\m X}\Big\langle \nabla_j V + \nabla h_j'(\rho_j^k) + \nabla_{\W_2}\m W_j(\rho_j^{k}), T_{\rho^k_j}^{\rho^\ast_j}-\id\Big\rangle\,\dd\rho^k \nonumber\\
&= -\sum_{j=1}^m \int_{\m X_j}\Big\langle \nabla V_j^k + \nabla h_j'(\rho_j^k) + \nabla_{\W_2}\m W_j(\rho_j^k), T_{\rho_j^k}^{\rho^\ast_j}-\id\Big\rangle\,\dd\rho_j^k \nonumber\\
&\stackrel{\ri}{\leq} \sum_{j=1}^m \sqrt{\int_{\m X_j}\big\|\nabla V_j^k + \nabla h_j'(\rho_j^k) + \nabla_{\W_2}\m W_j(\rho_j^k)\big\|^2\,\dd\rho_j^k}\cdot\W_2(\rho_j^k, \rho^\ast_j) \nonumber\\
&\stackrel{\rii}{\leq}\sqrt{\sum_{j=1}^m\int_{\m X_j}\big\|\nabla V_j^k + \nabla h_j'(\rho_j^k) + \nabla_{\W_2}\m W_j(\rho_j^k)\big\|^2\,\dd\rho_j^k}\cdot\W_2(\rho^k, \rho^\ast ) \nonumber\\
&\stackrel{\riii}{\leq} \sqrt{2L^2(m-1)+\frac{2}{\tau^2}}\W_2(\rho^k, \rho^{k-1})\W_2(\rho^k,\rho^\ast  ).\label{eqn: upper_bound_functional_value2}
\end{align}
Here, both (i) and (ii) are due to Cauchy--Schwarz inequality, and (iii) is by Lemma \ref{lem: relative_error_cond}. By Assumption D ($\lambda$-QG condition), we have
\begin{align*}
    \frac{\lambda}{2}\W_2^2(\rho^k, \rho^\ast ) \leq \m F(\rho^t) - \m F(\rho^\ast  ) \leq \sqrt{2L^2(m-1)+\frac{2}{\tau^2}}\W_2(\rho^k, \rho^{k-1})\W_2(\rho^k,\rho^\ast  ). 
\end{align*}
This implies
\begin{align*}
\W_2(\rho^k, \rho^\ast ) \leq \frac{2}{\lambda}\sqrt{2L^2(m-1)+\frac{2}{\tau^2}}\W_2(\rho^k, \rho^{k-1}).
\end{align*}
Substituting the term $\W_2(\rho^k, \rho^\ast)$ in~\eqref{eqn: upper_bound_functional_value2} with the above inequality yields
\begin{align*}
\m F(\rho^k) - \m F(\rho^\ast ) \leq \frac{2}{\lambda}\Big(2L^2(m-1) + \frac{2}{\tau^2}\Big)\W_2^2(\rho^k, \rho^{k-1}).
\end{align*}

\subsection{Proof of Lemma~\ref{lem: sufficient_decrease_sequential}}
\begin{proof}
By the definition of $\rho_j^{k}$, we have
\begin{align*}
&\quad\,\,\m V(\rho_1^{k-1}, \dots, \rho_j^{k-1}, \rho_{j+1}^{k}, \dots, \rho_m^{k}) + \m H_j(\rho_j^{k-1}) + \m W_j(\rho_j^{k-1})\\
&\geq \m V(\rho_1^{k-1},\dots, \rho_{j-1}^{k-1}, \rho_j^{k},\dots, \rho_m^{k}) + \m H_j(\rho_j^{k}) + \m W_j(\rho_j^k) + \frac{1}{2\tau}\W_2^2(\rho_j^{k-1}, \rho_j^{k}).
\end{align*}
This implies
\begin{align*}
    \m F(\rho^{k-1}) - \m F(\rho^{k}) &= \sum_{j=1}^m \Big[\m V(\rho_1^{k-1}, \dots, \rho_j^{k-1}, \rho_{j}^{k+1}, \dots, \rho_m^{k}) + \m H_j(\rho_j^{k-1}) + \m W_j(\rho_j^{k-1})\\
    &\qquad\qquad\qquad\qquad - \m V(\rho_1^{k-1}\dots, \rho_{j-1}^{k-1}, \rho_j^{k}, \dots, \rho_m^{k}) - \m H_j(\rho_j^{k}) - \m W_j(\rho_j^k)\Big]\\
    &\geq \sum_{j=1}^m \frac{1}{2\tau}\W_2^2(\rho_j^k, \rho_j^{k-1})\\
    &= \frac{1}{2\tau}\W_2^2(\rho^k, \rho^{k-1}).
\end{align*}
\end{proof}


\subsection{Proof of Lemma~\ref{lem: covering}}
\begin{proof}
Let $Y_j = \sum_{l=0}^{M-1} I\{j_l = j\}$ and $A_j = \{Y_j = 0\}\cup\{Y_j > T\}$ be a event. Then
\begin{align*}
    \mb P\big(A_1 \bigcup \cdots \bigcup A_m\big) \leq \sum_{j=1}^m\mb P(A_j) = m\mb P(Y_1 = 0) + m\mb P(Y_1 > T).
\end{align*}
Take $M = Rm\log m$ with some $R$ to be decided later. Note that
\begin{align*}
    \mb P(Y_1 = 0) = \Big(1 - \frac{1}{m}\Big)^M = \bigg[\frac{1}{\big(1 + \frac{1}{m-1}\big)^{m}}\bigg]^{\frac{M}{m}} \leq e^{-\frac{M}{m}} = e^{-R\log m} = m^{-R}.
\end{align*}

Applying Chernoff's inequality to Bernoulli distribution (Theorem 2.3.1 in\citep{vershynin2018high}) yields
\begin{align*}
    \mb P\Big(Y_1 > eR\log m\Big) \leq e^{-R\log m} = m^{-R}.
\end{align*}
Therefore, we have
\begin{align*}
    \mb P\big(A_1 \bigcup \cdots \bigcup A_m\big) \leq \sum_{j=1}^m\mb P(A_j) \leq 2m\cdot m^{-R} = 2m^{1-R}.
\end{align*}
By taking
\[R = 1 - \frac{\log\delta}{\log m}\]
we have $2m^{1-R} = 2\delta$.
\end{proof}

\subsection{Proof of Lemma~\ref{lem: sufficient_decrease_randomized}}
\begin{proof}
By the definition of $\rho_{j_l}^{k, l+1}$, we have
\begin{align*}
\m V(\rho_{j_l}^{k, l+1}, \rho_{-j_l}^{k, l}) + \m H_{j_l}(\rho_{j_l}^{k, l+1}) + \m W_{j_l}(\rho_{j_l}^{k, l+1}) &+ \frac{1}{2\tau}\W_2^2(\rho_{j_l}^{k,l+1}, \rho_{j_l}^{k, l})\\
&\leq     \m V(\rho_{j_l}^{k, l}, \rho_{-j_l}^{k, l}) + \m H_{j_l}(\rho_{j_l}^{k, l}) + \m W_{j_l}(\rho_{j_l}^{k, l}). 
\end{align*}
Therefore, 
\begin{align*}
&\quad\,\m F(\rho^k) - \m F(\rho^{k+1}) = \m F(\rho^{k, 0}) - \m F(\rho^{k, M}) = \sum_{l=0}^{M-1}\big[\m F(\rho^{k, l}) - \m F(\rho^{k, l+1})\big]\\
&= \!\!\sum_{l=0}^{M-1}\!\! \big[\m V(\rho_{j_l}^{k, l}, \rho_{-j_l}^{k, l}) \!+\! \m H_{j_l}(\rho_{j_l}^{k, l}) \!+\! \m W_{j_l}(\rho_{j_l}^{k, l})\big] \!-\! \big[\m V(\rho_{j_l}^{k, l+1}, \rho_{-j_l}^{k, l+1}) \!+\! \m H_{j_l}(\rho_{j_l}^{k, l+1}) \!+\! \m W_{j_l}(\rho_{j_l}^{k, l+1})\big]\\
&\stackrel{\ri}{=}\!\! \sum_{l=0}^{M-1}\!\! \big[\m V(\rho_{j_l}^{k, l}, \rho_{-j_l}^{k, l}) \!+\! \m H_{j_l}(\rho_{j_l}^{k, l}) \!+\! \m W_{j_l}(\rho_{j_l}^{k, l})\big] - \big[\m V(\rho_{j_l}^{k, l+1}, \rho_{-j_l}^{k, l}) \!+\! \m H_{j_l}(\rho_{j_l}^{k, l+1}) \!+\! \m W_{j_l}(\rho_{j_l}^{k, l+1})\big]\\
&\geq \sum_{l=0}^{M-1} \frac{1}{2\tau}\W_2^2(\rho_{j_l}^{k, l}, \rho_{j_l}^{k, l+1}) 
\stackrel{\rii}{=} \frac{1}{2\tau}\sum_{l=0}^{M-1} \W_2^2(\rho^{k, l}, \rho^{k, l+1}).
\end{align*}
Here, both (i) and (ii) are due to $\rho_{-j_l}^{k, l+1} = \rho_{-j_l}^{k, l}$. 
\end{proof}

\subsection{Proof of Lemma~\ref{lem: bound_func_value_randomized}}
\begin{proof}
Similar to the argument in Lemma \ref{lem: bound_func_value}, we have
\begin{align*}
\m F(\rho^{k}) - \m F(\rho^\ast) 
&\leq \sqrt{\sum_{j=1}^m \int_{\m X_j}\big\|\nabla V_j^{k} + \nabla h_j'(\rho_j^{k}) + \nabla_{\W_2}\m W_j(\rho_j^{k})\big\|^2\,\dd\rho_j^{k}} \cdot \W_2(\rho^{k}, \rho^\ast).
\end{align*}
To bound the norm of the blockwise Wasserstein gradient, we need the following lemma.

\begin{lemma}\label{lem: relative_error_bound_randomized}
For any $k\in \mb Z_+$, let $M\in\mb Z_+$ be an integer such that all coordinates are updated at least once but at most $T$ times through $M$ updates in the $k$-th iteration. Under Assumptions A and B, we have
\begin{align*}
\sum_{j=1}^m \int_{\m X_j}\big\|\nabla V_j^{k} + \nabla h_j'(\rho_j^{k}) + \nabla_{\W_2}\m W_j(\rho_j^{k})\big\|^2\,\dd\rho_j^{k}
&\leq \Big(2L^2mT + \frac{2}{\tau^2}\Big) \sum_{l=0}^{M-1}\W_2^2(\rho^{k-1, l}, \rho^{k-1, l+1}).
\end{align*}
\end{lemma}

\noindent With the above lemma, we have
\begin{align}\label{eqn: upper_bound_functional_value2_random}
\m F(\rho^{k}) - \m F(\rho^\ast) \leq \sqrt{\Big(2L^2mT + \frac{2}{\tau^2}\Big) \sum_{l=0}^{M-1}\W_2^2(\rho^{k-1, l}, \rho^{k-1, l+1})} \cdot \W_2(\rho^{k}, \rho^\ast).
\end{align}
By Assumption D ($\lambda$-QG), we have
\begin{align*}
\frac{\lambda}{2}\W_2^2(\rho^k, \rho^\ast) \leq \m F(\rho^k) - \m F(\rho^\ast).
\end{align*}
Combining the above two pieces yields
\begin{align*}
\frac{\lambda}{2}\W_2(\rho^k, \rho^\ast) \leq \sqrt{\Big(2L^2mT + \frac{2}{\tau^2}\Big)\sum_{l=0}^{M-1}\W_2^2(\rho^{k-1, l}, \rho^{k-1, l+1})},
\end{align*}
which implies
\begin{align*}
\m F(\rho^k) - \m F(\rho^\ast) \leq \frac{4}{\lambda}\Big(L^2mT + \frac{1}{\tau^2}\Big)\sum_{l=0}^{M-1}\W_2^2(\rho^{k-1, l}, \rho^{k-1, l+1}).
\end{align*}
\end{proof}

\subsection{Proof of Lemma~\ref{lem: inexact_suff_decrease}}
\underline{Step 1.} By convexity of $\m F$, we have
\begin{align*}
&\quad\,\big[\m V(\wt\rho_j^{k}\otimes\wt \rho_{-j}^k) + \m H_j(\wt\rho_j^{k}) + \m W_j(\wt\rho_j^{k})\big] - \big[\m V(\wt\rho_j^{k+1}\otimes\wt \rho_{-j}^k) + \m H_j(\wt\rho_j^{k+1}) + \m W_j(\wt\rho_j^{k+1})\big]\\
&\geq \int_{\m X_j}\Big\langle\nabla V_j^k + \nabla h_j'(\wt\rho_j^{k+1}) + \nabla_{\W_2}\m W_j(\wt\rho_j^{k+1}), T_{\wt\rho_j^{k+1}}^{\wt\rho_j^k} - \id\Big\rangle\,\dd\wt\rho_j^{k+1}\\
&= \frac{1}{\tau}\int_{\m X_j}\Big\langle T_{\wt\rho_j^{k+1}}^{\wt\rho_j^k} -\id - \eta_j^{k+1}, T_{\wt\rho_j^{k+1}}^{\wt\rho_j^k} - \id\Big\rangle\,\dd\wt\rho_j^{k+1}\\
&= \frac{1}{\tau}\W_2^2(\wt\rho_j^{k+1}, \wt\rho_j^k) - \frac{1}{\tau}\int_{\m X_j}\Big\langle \eta_j^{k+1}, T_{\wt\rho_j^{k+1}}^{\wt\rho_j^k} - \id\Big\rangle\,\dd\wt\rho_j^{k+1}\\
&\geq \frac{1}{\tau}\W_2^2(\wt\rho_j^{k+1}, \wt\rho_j^k) - \frac{1}{\tau}\bigg(\frac{1}{2}\int_{\m X_j}\|\eta_j^{k+1}\|^2\,\dd\wt\rho_j^{k+1} + \frac{1}{2}\W_2^2(\wt\rho_j^{k+1}, \wt\rho_j^k)\bigg)\\
&\geq \frac{1}{2\tau}\W_2^2(\wt\rho_j^{k+1}, \wt\rho_j^k) - \frac{(\varepsilon_j^{k+1})^2}{2\tau}.
\end{align*}
Summing $j$ from $1$ to $m$ yields
\begin{align*}
\sum_{j=1}^m \m V(\wt\rho_j^{k+1}\otimes\wt\rho_{-j}^k) &\leq m\m V(\wt\rho^{k}) + \m H(\wt\rho^k) + \m W(\wt\rho^k)\\
&\qquad\qquad\qquad - \m H(\wt\rho^{k+1}) - \m W(\wt\rho^k) - \frac{1}{2\tau}\W_2^2(\wt\rho^{k+1}, \wt\rho^k) + \frac{\|\varepsilon^{k+1}\|^2}{2\tau}.
\end{align*}
\underline{Step 2.} By convexity of $\m V$, we have
\begin{align*}
\m V(\wt\rho_j^{k+1}\otimes\wt\rho_{-j}^k) 
&\geq \m V(\wt\rho_j^{k+1}\otimes\wt\rho_{-j}^{k+1}) + \int_{\m X_{-j}}\Big\langle\nabla \frac{\delta\m V(\wt\rho_j^{k+1}\otimes\cdot)}{\delta\rho_{-j}}(\wt\rho_{-j}^{k+1}), T_{\wt\rho_{-j}^{k+1}}^{\wt\rho_{-j}^{k}} - \id\Big\rangle\,\dd\wt\rho_{-j}^{k+1}\\
&= \m V(\wt\rho^{k+1}) + \sum_{i\neq j}\int_{\m X_i}\Big\langle\nabla V_i^{k+1}, T_{\wt\rho_i^{k+1}}^{\wt\rho_i^k} - \id\Big\rangle\,\dd\wt\rho_i^{k+1}.
\end{align*}
Summing $j$ from $1$ to $m$ yields
\begin{align*}
\sum_{j=1}^m \m V(\wt\rho_j^{k+1}\otimes\wt\rho_{-j}^k) \geq m\m V(\wt\rho^{k+1}) + (m-1)\sum_{j=1}^m \int_{\m X_j} \Big\langle\nabla V_j^{k+1}, T_{\wt\rho_j^{k+1}}^{\wt\rho_j^k} - \id\Big\rangle\,\dd\wt\rho_j^{k+1}.
\end{align*}
\underline{Step 3.} We have
\begin{align*}
&\quad\,\sum_{j=1}^m \int_{\m X_j} \Big\langle\nabla V_j^{k+1}, T_{\wt\rho_j^{k+1}}^{\wt\rho_j^k} - \id\Big\rangle\,\dd\wt\rho_j^{k+1}\\
&=\sum_{j=1}^m \int_{\m X_j} \Big\langle\frac{T_{\wt\rho_j^{k+1}}^{\wt\rho_j^k} - \id - \eta_j^{k+1}}{\tau}  - \big(\nabla V_j^k + \nabla h_j'(\wt\rho_j^{k+1}) + \nabla_{\W_2}\m W_j(\wt\rho_j^{k+1})\big), T_{\wt\rho_j^{k+1}}^{\wt\rho_j^k} - \id\Big\rangle\\
&\qquad\qquad\qquad\qquad\qquad\qquad\qquad\qquad\qquad\qquad\qquad\qquad\qquad + \Big\langle\nabla V_j^{k+1}, T_{\wt\rho_j^{k+1}}^{\wt\rho_j^k} - \id\Big\rangle\,\dd\wt\rho_j^{k+1}\\
&= \sum_{j=1}^m \bigg[\int_{\m X_j}\Big\langle\nabla V_j^{k+1} - \nabla V_j^k, T_{\wt\rho_j^{k+1}}^{\wt\rho_j^k} - \id\Big\rangle\,\dd\wt\rho_j^{k+1} - \int_{\m X_j}\Big\langle \nabla h_j'(\wt\rho_j^{k+1}), T_{\wt\rho_j^{k+1}}^{\wt\rho_j^k} - \id\Big\rangle\,\dd\wt\rho_j^{k+1}\\
&\quad\qquad\qquad\qquad\qquad\qquad\quad - \int_{\m X_j}\Big\langle \nabla_{\W_2}\m W_j(\wt\rho_j^{k+1}), T_{\wt\rho_j^{k+1}}^{\wt\rho_j^k} - \id\Big\rangle\,\dd\wt\rho_j^{k+1} + \frac{1}{\tau}\W_2^2(\wt\rho_j^{k+1}, \wt\rho_j^k)\\
&\qquad\qquad\qquad\qquad\qquad\qquad\qquad\qquad\qquad\qquad\qquad\qquad - \frac{1}{\tau}\int_{\m X_j}\Big\langle \eta_j^{k+1}, T_{\wt\rho_j^{k+1}}^{\wt\rho_j^k} - \id\Big\rangle\,\dd\wt\rho_j^{k+1}\bigg]\\
&\geq \sum_{j=1}^m \bigg[\int_{\m X_j}\Big\langle\nabla V_j^{k+1} - \nabla V_j^k, T_{\wt\rho_j^{k+1}}^{\wt\rho_j^k} - \id\Big\rangle\,\dd\wt\rho_j^{k+1} - \int_{\m X_j}\Big\langle \nabla h_j'(\wt\rho_j^{k+1}), T_{\wt\rho_j^{k+1}}^{\wt\rho_j^k} - \id\Big\rangle\,\dd\wt\rho_j^{k+1}\\
&\quad\qquad\qquad\qquad\quad - \int_{\m X_j}\Big\langle \nabla_{\W_2}\m W_j(\wt\rho_j^{k+1}), T_{\wt\rho_j^{k+1}}^{\wt\rho_j^k} - \id\Big\rangle\,\dd\wt\rho_j^{k+1} + \frac{1}{2\tau}\W_2^2(\wt\rho_j^{k+1}, \wt\rho_j^k) - \frac{(\varepsilon_j^{k+1})^2}{2\tau}\bigg]\\
&\geq -\sum_{j=1}^m \W_2(\wt\rho_j^{k+1}, \wt\rho_j^k)\cdot\Big\|\nabla V_j^{k+1} - \nabla V_j^k\Big\|_{L^2(\wt\rho_j^{k+1}; \m X_j)} - \big[\m H(\wt\rho^{k}) - \m H(\wt\rho^{k+1})\big]\\
&\qquad\qquad\qquad\qquad\qquad\qquad\qquad\qquad  - \big[\m W(\wt\rho^k) - \m W(\wt\rho^{k+1})\big] + \frac{1}{2\tau}\W_2^2(\wt\rho^{k+1}, \wt\rho^k) - \frac{\|\varepsilon^{k+1}\|^2}{2\tau}\\
&\geq -\sum_{j=1}^m \W_2(\wt\rho_j^{k+1}, \wt\rho_j^k)\cdot L\W_2(\wt\rho_{-j}^{k+1}, \wt\rho_{-j}^k) - \big[\m H(\wt\rho^{k}) - \m H(\wt\rho^{k+1})\big]\\
&\qquad\qquad\qquad\qquad\qquad\qquad\qquad\qquad - \big[\m W(\wt\rho^k) - \m W(\wt\rho^{k+1})\big] + \frac{1}{2\tau}\W_2^2(\wt\rho^{k+1}, \wt\rho^k) - \frac{\|\varepsilon^{k+1}\|^2}{2\tau}\\
&\geq -L \sum_{j=1}^m\bigg[\frac{\sqrt{m-1}}{2}\W_2^2(\wt\rho_j^{k+1}, \wt\rho_j^k) + \frac{1}{2\sqrt{m-1}}\W_2^2(\wt\rho_{-j}^{k+1}, \wt\rho_{-j}^k)\bigg] - \big[\m H(\wt\rho^{k}) - \m H(\wt\rho^{k+1})\big]\\
&\qquad\qquad\qquad\qquad\qquad\qquad\qquad\qquad  - \big[\m W(\wt\rho^k) - \m W(\wt\rho^{k+1})\big] + \frac{1}{2\tau}\W_2^2(\wt\rho^{k+1}, \wt\rho^k) - \frac{\|\varepsilon^{k+1}\|^2}{2\tau}\\
&= \Big(\frac{1}{2\tau} - L\sqrt{m-1}\Big)\W_2^2(\wt\rho^{k+1}, \wt\rho^k) - \big[\m H(\wt\rho^{k}) - \m H(\wt\rho^{k+1})\big] - \big[\m W(\wt\rho^k) - \m W(\wt\rho^{k+1})\big]  - \frac{\|\varepsilon^{k+1}\|^2}{2\tau}.
\end{align*}
\underline{Step 4.} Combining all pieces above yields
\begin{align*}
&\quad\,m\m V(\wt\rho^{k}) + \m H(\wt\rho^k) + \m W(\wt\rho^k) - \m H(\wt\rho^{k+1}) - \m W(\wt\rho^k) - \frac{1}{2\tau}\W_2^2(\wt\rho^{k+1}, \wt\rho^k) + \frac{\|\varepsilon^{k+1}\|^2}{2\tau}\\
& \geq \sum_{j=1}^m \m V(\wt\rho_j^{k+1}\otimes\wt\rho_{-j}^k) 
\geq m\m V(\wt\rho^{k+1}) + (m-1)\sum_{j=1}^m \int_{\m X_j} \Big\langle\nabla V_j^{k+1}, T_{\wt\rho_j^{k+1}}^{\wt\rho_j^k} - \id\Big\rangle\,\dd\wt\rho_j^{k+1}\\
&\geq m\m V(\wt\rho^{k+1}) + (m-1)\bigg[\Big(\frac{1}{2\tau} - L\sqrt{m-1}\Big)\W_2^2(\wt\rho^{k+1}, \wt\rho^k) - \big[\m H(\wt\rho^{k}) - \m H(\wt\rho^{k+1})\big]\\
&\qquad\qquad\qquad\qquad\qquad\qquad\qquad\qquad\qquad\qquad\qquad- \big[\m W(\wt\rho^k) - \m W(\wt\rho^{k+1})\big]  - \frac{\|\varepsilon^{k+1}\|^2}{2\tau}\bigg].
\end{align*}
This implies
\begin{align*}
\m F(\wt\rho^k) - \m F(\wt\rho^{k+1}) \geq \Big(\frac{1}{2\tau} - L\sqrt{m-1}\Big)\W_2^2(\wt\rho^{k+1}, \wt\rho^k) - \frac{\|\varepsilon^{k+1}\|^2}{2\tau}.
\end{align*}

\subsection{Proof of Lemma~\ref{lem: inexact_func_upperbd}}
\underline{Step 1.} We first prove that
\begin{align*}
\sum_{j=1}^m \int_{\m X_j}\big\|\nabla V_j^k + \nabla h_j'(\wt\rho_j^k) + \nabla_{\W_2}\m W_j(\wt\rho_j^k)\big\|^2\,\dd\wt\rho_j^k
\leq \Big(3L^2(m-1) + \frac{3}{\tau^2}\Big)\!\W_2^2(\wt\rho^k, \wt\rho^{k-1}) + \frac{3\|\varepsilon^k\|^2}{\tau^2}.
\end{align*}
Just note that
\begin{align*}
&\quad\,\int_{\m X_j}\big\|\nabla V_j^k + \nabla h_j'(\wt\rho_j^k) + \nabla_{\W_2}\m W_j(\wt\rho_j^k)\big\|^2\,\dd\wt\rho_j^k\\
&= \int_{\m X_j}\Big\|\nabla V_j^k - \nabla V_j^{k-1} + \frac{1}{\tau}\Big(T_{\wt\rho_j^k}^{\wt\rho_j^{k-1}} - \id - \eta_j^k\Big)\Big\|^2\,\dd\wt\rho_j^k\\
&\leq 3\int_{\m X_j}\big\|\nabla V_j^k - \nabla V_j^{k-1}\big\|^2\,\dd\wt\rho_j^k + \frac{3}{\tau^2}\W_2^2(\wt\rho_j^k, \wt\rho_j^{k-1}) + \frac{3(\varepsilon_j^k)^2}{\tau^2}\\
&\leq 3L^2\W_2^2(\wt\rho_{-j}^{k-1}, \wt\rho_{-j}^k) + \frac{3}{\tau^2}\W_2^2(\wt\rho_j^k, \wt\rho_j^{k-1}) + \frac{3(\varepsilon_j^k)^2}{\tau^2}.
\end{align*}
Summing $j$ from $1$ to $m$ yields
\begin{align*}
\sum_{j=1}^m \int_{\m X_j}\big\|\nabla V_j^k + \nabla h_j'(\wt\rho_j^k) + \nabla_{\W_2}\m W_j(\wt\rho_j^k)\big\|^2\,\dd\wt\rho_j^k 
\leq \Big(3L^2(m-1) + \frac{3}{\tau^2}\Big)\W_2^2(\wt\rho^k, \wt\rho^{k-1}) + \frac{3\|\varepsilon^k\|^2}{\tau^2}.
\end{align*}

\noindent\underline{Step 2.} By convexity of $\m F$, we have
\begin{align}
\m F(\wt\rho^k) - \m F(\rho^\ast)
&\leq -\int_{\m X}\Big\langle \nabla\frac{\delta\m F}{\delta\rho}(\wt\rho^k), T_{\wt\rho^k}^{\rho^\ast} - \id\Big\rangle\,\dd\wt\rho^k\nonumber\\
&= -\sum_{j=1}^m \int_{\m X}\Big\langle \nabla_j V + \nabla h_j'(\wt\rho_j^k) + \nabla_{\W_2}\m W_j(\wt\rho_j^k), T_{\wt\rho_j^k}^{\rho_j^\ast}-\id\Big\rangle\,\dd\wt\rho^k\nonumber\\
&= -\sum_{j=1}^m \int_{\m X_j}\Big\langle \nabla V_j^k + \nabla h_j'(\wt\rho_j^k) + \nabla_{\W_2}\m W_j(\wt\rho_j^k), T_{\wt\rho_j^k}^{\rho_j^\ast}-\id\Big\rangle\,\dd\wt\rho_j^k\nonumber\\
&\leq \sum_{j=1}^m \sqrt{\int_{\m X_j}\big\|\nabla V_j^k + \nabla h_j'(\wt\rho_j^k) + \nabla_{\W_2}\m W_j(\wt\rho_j^k)\big\|^2\,\dd\wt\rho_j^k} \cdot \W_2(\wt\rho_j^k, \rho_j^\ast)\nonumber\\
&\leq \sqrt{\sum_{j=1}^m \int_{\m X_j}\big\|\nabla V_j^k + \nabla h_j'(\wt\rho_j^k) + \nabla_{\W_2}\m W_j(\wt\rho_j^k)\big\|^2\,\dd\wt\rho_j^k} \cdot \W_2(\wt\rho^k, \rho^\ast)\nonumber\\
&\leq \sqrt{\Big(3L^2(m-1) + \frac{3}{\tau^2}\Big)\W_2^2(\wt\rho^k, \wt\rho^{k-1}) + \frac{3\|\varepsilon^k\|^2}{\tau^2}} \cdot \W_2(\wt\rho^k, \rho^\ast).\label{eqn: func_value_uppbd}
\end{align}
By $\lambda$-QG condition, we have
\begin{align*}
\frac{\lambda}{2}\W_2^2(\wt\rho^k, \rho^\ast) \leq \m F(\wt\rho^k) - \m F(\rho^\ast) \leq \sqrt{\Big(3L^2(m-1) + \frac{3}{\tau^2}\Big)\W_2^2(\wt\rho^k, \wt\rho^{k-1}) + \frac{3\|\varepsilon^k\|^2}{\tau^2}} \cdot \W_2(\wt\rho^k, \rho^\ast).
\end{align*}
So, we have
\begin{align*}
\W_2(\wt\rho^k, \rho^\ast) \leq \frac{2}{\lambda}\sqrt{\Big(3L^2(m-1) + \frac{3}{\tau^2}\Big)\W_2^2(\wt\rho^k, \wt\rho^{k-1}) + \frac{3\|\varepsilon^k\|^2}{\tau^2}},
\end{align*}
indicating that
\begin{align*}
\m F(\wt\rho^k) - \m F(\rho^\ast) \leq \frac{6}{\lambda}\Big[\Big(L^2(m-1) + \frac{1}{\tau^2}\Big)\W_2^2(\wt\rho^k, \wt\rho^{k-1}) + \frac{\|\varepsilon^k\|^2}{\tau^2}\Big].
\end{align*}

\subsection{Proof of Lemma~\ref{lem: relative_error_cond}}
The proof for two schemes are slightly different. We will consider these two cases separately.

\paragraph{Parallel update scheme.} By Lemma \ref{lem: ot_map}, $T_{\rho_j^{k}}^{\rho_j^{k-1}}-\id = \tau\nabla V_j^{k-1} + \tau\nabla h_j'(\rho_j^{k}) + \tau\nabla_{\W_2}\m W_j(\rho_j^k)$. Therefore, we have
\begin{align*}
\int_{\m X_j}\big\|\nabla V_j^k + \nabla h_j'(\rho_j^k) + \nabla_{\W_2}\m W_j(\rho_j^k)\big\|^2\,\dd\rho_j^k
&= \int_{\m X_j}\Big\|\nabla V_j^k - \nabla V_j^{k-1} + \frac{1}{\tau}\Big(T_{\rho_j^k}^{\rho_j^{k-1}} - \id\Big)\Big\|^2\,\dd\rho_j^k\\
&\stackrel{\ri}{\leq} 2\int_{\m X_j}\big\|\nabla V_j^k - \nabla V_j^{k-1}\big\|^2\,\dd\rho_j^k + \frac{2\W_2^2(\rho_j^k, \rho_j^{k-1})}{\tau^2}\\
&\stackrel{\rii}{\leq} 2L^2\W_2^2(\rho_{-j}^{k-1}, \rho_{-j}^k) + \frac{2\W_2^2(\rho_j^k, \rho_j^{k-1})}{\tau^2}.
\end{align*}
Here, (i) is by Cauchy--Schwarz inequality, and (ii) is by Lemma \ref{lem: diff_potential_w2}. Summing the above inequality from $j=1$ to $m$ yields the desired result.

\paragraph{Sequential update scheme.} For simplicity, let
\begin{align*}
\widetilde V_j^k = \int_{\m X_{-j}}V(x_j, x_{-j})\,\dd\rho_1^{k}\cdots \dd\rho_{j-1}^{k}\,\dd\rho_{j+1}^{k-1}\cdots\dd\rho_m^{k-1}
\end{align*}
be a function on $\m X_j$. By Lemma \ref{lem: ot_map}, we have $T_{\rho_j^{k}}^{\rho_j^{k-1}} - \id = \tau\big[\nabla \widetilde V_j^k + \nabla h_j'(\rho_j^k) + \nabla_{\W_2}\m W_j(\rho_j^k)\big]$. Therefore, we have
\begin{align*}
&\quad\, \sum_{j=1}^m \int_{\m X_j}\big\|\nabla V_j^k + \nabla h_j'(\rho_j^k) + \nabla_{\W_2}\m W_j(\rho_j^k)\big\|^2\,\dd\rho_j^k\\
&= \sum_{j=1}^m \int_{\m X_j}\Big\|\nabla V_j^k - \nabla \widetilde V_j^k + \frac{1}{\tau}\Big(T_{\rho_j^k}^{\rho_j^{k-1}}-\id\Big)\Big\|^2\,\dd\rho_j^k\\
&\stackrel{\ri}{\leq} 2\sum_{j=1}^m \int_{\m X_j}\big\|\nabla V_j^k - \nabla \widetilde V_j^k\big\|^2\,\dd\rho_j^k + \frac{2}{\tau^2}\sum_{j=1}^m\int_{\m X_j}\Big\|T_{\rho_j^k}^{\rho_j^{k-1}}-\id\Big\|^2\,\dd\rho_j^k\\
&\stackrel{\rii}{\leq} 2\sum_{j=1}^m L^2\sum_{l=j+1}^m\W_2^2(\rho_l^k, \rho_l^{k-1}) + \frac{2}{\tau^2}\W_2^2(\rho^k, \rho^{k-1})\\
    &\leq \Big[2L^2(m-1) + \frac{2}{\tau^2} \Big]\W_2^2(\rho^k, \rho^{k-1}).
\end{align*}
Here, (i) is by Cauchy--Schwarz inequality; (ii) is by Lemma \ref{lem: diff_potential_w2}.

\subsection{Proof of Lemma~\ref{lem: relative_error_bound_randomized}}
\begin{proof}
For any $j\in[m]$ and $k\in\mb N$, let $I_{j, k}\in[M]$ be the time of the latest update of the $j$-th coordinate in the $k$-th iteration. By definition, we know $\rho_j^{k-1, M} = \rho_j^{k} = \rho_j^{k-1, I_{j, k}}$ for all $j\in[m]$. Since
\begin{align*}
    \rho_j^{k-1, I_{j, k}} = \argmin_{\rho_j\in\ms P_2^r(\m X_j)} \m V\big(\rho_j, \rho_{-j}^{k-1, I_{j, k}-1}\big) + \m H_j(\rho_j) + \m W_j(\rho_j) + \frac{1}{2\tau}\W_2^2\big(\rho_j, \rho_j^{k-1, I_{j, k}-1}\big),
\end{align*}
Lemma \ref{lem: ot_map} implies 
\begin{align*}
    T_{\rho_j^{k-1, I_{j, k}}}^{\rho_j^{k-1, I_{j, k}-1}} - \id = \tau \Big(\nabla V_j^{k-1, I_{j, k}-1} + \nabla h_j'\big(\rho_j^{k-1, I_{j, k}}\big) + \nabla_{\W_2}\m W_j(\rho_j^{k-1, I_{j, k}})\Big),
\end{align*}
where we let
\begin{align*}
V_j^{k-1, I_{j, k}-1}(x_j) = \int_{\m X_{-j}} V(x_j, x_{-j})\,\dd \rho_{-j}^{k-1, I_{j, k}-1}
\end{align*}
Therefore, we have
\begin{align}
&\quad\,\sum_{j=1}^m\int_{\m X_j}\big\|\nabla V_j^{k} + \nabla h_j'(\rho_j^{k}) + \nabla_{\W_2}\m W_j(\rho_j^{k})\big\|^2\,\dd\rho_j^{k}\nonumber\\
&= \sum_{j=1}^m\int_{\m X_j}\big\|\nabla V_j^{k} + \nabla h_j'(\rho_j^{k-1, I_{j, k}}) + \nabla_{\W_2}\m W_j(\rho_j^{k-1, I_{j, k}})\big\|^2\,\dd\rho_j^{k}\nonumber\\
&= \sum_{j=1}^m\int_{\m X_j}\Big\|\nabla V_j^{k} - \nabla V_j^{k-1, I_{j, k}-1} + \frac{1}{\tau}\Big( T_{\rho_j^{k-1, I_{j, k}}}^{\rho_j^{k-1, I_{j, k}-1}} - \id\Big)\Big\|^2\,\dd\rho_j^{k}\nonumber\\
&\leq 2\sum_{j=1}^m\int_{\m X_j}\big\|\nabla V_j^{k} - \nabla V_j^{k-1, I_{j, k}-1}\big\|^2\,\dd\rho_j^{k} + \frac{2}{\tau^2}\sum_{j=1}^m\int_{\m X_j}\Big\|T_{\rho_j^{k-1, I_{j, k}}}^{\rho_j^{k-1, I_{j, k}-1}} - \id\Big\|^2\,\dd\rho_j^{k}\nonumber\\
&\stackrel{\ri}{\leq} 2L^2\sum_{j=1}^m \W_2^2\big(\rho_{-j}^{k}, \rho_{-j}^{k-1, I_{j, k}-1}\big) + \frac{2}{\tau^2}\sum_{j=1}^m\W_2^2\big(\rho_j^{k-1, I_{j, k}-1}, \rho_j^{k-1, I_{j, k}}\big)\nonumber\\
&= 2L^2\sum_{j=1}^m\W_2^2\big(\rho_{-j}^{k-1, M}, \rho_{-j}^{k-1, I_{j, k}-1}\big) + \frac{2}{\tau^2}\sum_{j=1}^m\W_2^2\big(\rho_j^{k-1, I_{j, k}-1}, \rho_j^{k-1, I_{j,k}}\big)\label{eqn: rand_grad_bound}.
\end{align}
Here, (i) is by Lemma \ref{lem: diff_potential_w2}. In the $k$-th iteration, recall that $j_l\in[m]$ be the coordinate updated in the $l$-th update. The goal is to bound~\eqref{eqn: rand_grad_bound}.

\vspace{0.5em}
\noindent\underline{Bound of the second term in~\eqref{eqn: rand_grad_bound}.} It is obvious that
\begin{align*}
\sum_{j=1}^m \W_2^2\big(\rho_j^{k-1, I_{j, k}-1}, \rho_j^{k-1, I_{j,k}}\big) \leq \sum_{l=0}^{M-1}\W_2^2(\rho_{j_l}^{k-1, l}, \rho_{j_l}^{k-1, l+1}) = \sum_{l=0}^{M-1}\W_2^2(\rho^{k-1, l}, \rho^{k-1, l+1}).
\end{align*}
The first inequality is because each coordinate is updated at least once, and the second equation is because only the $j_l$-th coordinate is updated in the $l$-th update. 

\vspace{0.5em}
\noindent\underline{Bound of the first term in~\eqref{eqn: rand_grad_bound}.} Note that
\begin{align*}
    \sum_{j=1}^m\W_2^2\big(\rho_{-j}^{k-1, M}, \rho_{-j}^{k-1, I_{j, k}-1}\big)
    &\leq \sum_{j=1}^m\W_2^2\big(\rho^{k-1, M}, \rho^{k-1, I_{j, k}-1}\big)
    = \sum_{i=1}^m\sum_{j=1}^m\W_2^2\big(\rho_i^{k-1, M}, \rho_i^{k-1, I_{j, k}-1}\big).
\end{align*}
To bound this sum, define the set of iteration 
\begin{align*}
S_j^{k-1}(a, b) =\{a\leq l\leq b: \rho_j^{k-1, l-1} \neq \rho_j^{k-1, l}\}.
\end{align*}
We know $|S_j^{k-1}(1, M)| \leq T$ since each coordinate is updated at most $T$ times. Thus
\begin{align*}
\sum_{i,j=1}^m\W_2^2\big(\rho_i^{k-1, M}, \rho_i^{k-1, I_{j, k}-1}\big) &\leq \sum_{i,j=1}^m\Big(\sum_{l\in S_i^{k-1}(I_{j, k}, M)}\W_2(\rho_i^{k-1, l-1}, \rho_i^{k-1, l})\Big)^2\\
&\leq \sum_{i, j=1}^m |S_i^{k-1}(I_{j, k}, M)| \cdot\sum_{l\in S_i^{k-1}(I_{j, k}, M)}\W_2^2(\rho_i^{k-1, l-1}, \rho_i^{k-1, l})\\
&\leq T\sum_{i,j=1}^m\sum_{l\in S_i^{k-1}(I_{j, k}, M)}\W_2^2(\rho_i^{k-1, l-1}, \rho_i^{k-1, l})\\
&\leq T\sum_{i, j=1}^m \sum_{l=1}^{M}\W_2^2(\rho_i^{k-1, l-1}, \rho_i^{k-1, l})\\
    &=mT\sum_{l=1}^{M}\W_2^2(\rho^{k-1, l}, \rho^{k-1, l-1}).
\end{align*}
Combining all pieces above yields the desired result.
\end{proof}

\section{Other Technical Lemmas}\label{appendix: lemmas}
\subsection{Proof of Proposition \ref{prop: suff_cond_QG}}
\begin{align*}
&\quad\,\,\m F(\rho_1, \dots, \rho_m) - \m F(\rho_1^\ast, \dots, \rho_m^\ast)
\geq \int_{\m X}\Big\langle \nabla \frac{\delta\m F}{\delta\rho}(\rho^\ast), T_{\rho^\ast}^\rho-\id\Big\rangle\,\dd\rho^\ast + \frac{\lambda}{2}\W_2^2(\rho, \rho^\ast)\\
&= \int_{\m X}\big\langle\nabla V, T_{\rho^\ast}^{\rho} - \id\big\rangle\,\dd\rho^\ast + \sum_{j=1}^m \int_{\m X_j}\big\langle h_j'(\rho_j^\ast) + \nabla_{\W_2}\m W_j(\rho_j^\ast), T_{\rho_j^\ast}^{\rho_j}-\id\big\rangle\,\dd\rho_j^\ast + \frac{\lambda}{2}\sum_{j=1}^m\W_2^2(\rho_j, \rho_j^\ast)\\
&= \sum_{j=1}^m \int_{\m X}\big\langle\nabla_j V, T_{\rho_j^\ast}^{\rho_j}-\id\big\rangle\,\dd\rho^\ast + \sum_{j=1}^m \int_{\m X_j}\big\langle h_j'(\rho_j^\ast) + \nabla_{\W_2}\m W_j(\rho_j^\ast), T_{\rho_j^\ast}^{\rho_j}-\id\big\rangle\,\dd\rho_j^\ast + \frac{\lambda}{2}\sum_{j=1}^m\W_2^2(\rho_j, \rho_j^\ast)\\
&= \sum_{j=1}^m \int_{\m X_j}\big\langle \nabla V_j^\ast + h_j'(\rho_j^\ast) + \nabla_{\W_2}\m W_j(\rho_j^\ast), T_{\rho_j^\ast}^{\rho_j}-\id\big\rangle\,\dd\rho_j^\ast + \frac{\lambda}{2}\sum_{j=1}^m\W_2^2(\rho_j, \rho_j^\ast),
\end{align*}
where we let
\begin{align*}
    V_j^\ast = \int_{\m X_{-j}}V(x_j, x_{-j})\,\dd\rho_{-j}^\ast
\end{align*}
for simplicity. Since $\rho^\ast$ is the stationary point of the optimization problem~\eqref{eqn: obj_func} on the Wasserstein space, we know $\nabla V_j^\ast + h_j'(\rho_j^\ast) + \nabla_{\W_2}\m W_j(\rho_j^\ast) = 0$ (e.g., using Lemma~\ref{lem: relative_error_cond}). This implies
\begin{align*}
    \m F(\rho_1, \dots, \rho_m) - \m F(\rho_1^\ast, \dots, \rho_m^\ast) \geq  \frac{\lambda}{2}\sum_{j=1}^m\W_2^2(\rho_j, \rho_j^\ast).
\end{align*}

\subsection{Proof of Lemma~\ref{lem: ot_map} (first-order optimality condition)}\label{app: proof_FOC}
\begin{proof}
Since $\rho\in\ms P_2^r(\m X)$, by Brenier's theorem (Theorem 2.12 in \citep{villani2021topics}), there is a unique optimal map $\nabla\phi$ from $\rho_\tau$ to $\rho$ and the Kantorovich potential $\phi$ is convex. Then, by Lemma 2.1 in \citep{del2019central} we know $\phi$ is unique up to a constant. By Proposition 7.17 in \citep{santambrogio2015optimal}, we know the first variation of the Wasserstein term is $\frac{\phi}{\tau}$. The first variation of the potential energy term and the self-interaction term are $U$ and $\int W(\cdot, y)\dd\rho_\tau(y) + \int W(y, \cdot)\dd\rho_\tau(y)$ respectively (see the first bullet point of Remark 7.13 in \citep{santambrogio2015optimal}). Next, we will show the first variation of the internal energy functional is $h'(\rho_\tau)$.

Let $\tilde\rho = c = |\m X|^{-1}$ be the constant positive density on $\m X$, and let $\rho_\varepsilon = (1-\varepsilon)\rho_\tau + \varepsilon\tilde\rho$. Similar to Lemma 8.6 in \citep{santambrogio2015optimal}, we have $C_U\varepsilon \geq \varepsilon(\rho_\tau - c)h'(\rho_\varepsilon)$ where $C_U$ is a positive value depending on $U$. Let
\begin{align*}
    A = \{x\in\m X: \rho_\tau(x) > 0\}, \quad B = \{x\in\m X: \rho_\tau(x) = 0\}.
\end{align*}
Then, we have
\begin{align*}
C_U \geq \int_{\m X} (\rho_\tau - c)h'(\rho_\varepsilon) 
&= \int_A(\rho_\tau - c)h'\big((1-\varepsilon)\rho_\tau + \varepsilon c\big) -c|B|h'(\varepsilon c)\label{eqn: 1}\\
&\geq \int_A (\rho_\tau - c)h'(c) - c|B|h'(\varepsilon c).
\end{align*}
Here, we get the last inequality since $h$ is convex and $h'$ in nondecreasing. This implies $\rho_\tau\in L^1(\m X)$. Taking $\tau\to 0$ and applying Fatou's lemma, we have $(\rho_\tau-c)h'(\rho_\tau) \in L^1(\m X)$. Since $h(\rho_\tau)\in L^1(\m X)$ which implies $\rho_\tau h'(\rho_\tau)\in L^1(\m X)$, we know $h'(\rho_\tau)\in L^1(\m X)$. Therefore, for any $\bar\rho\in L^\infty(\m X)$, we have
\begin{align*}
\Big|\frac{\partial}{\partial\varepsilon} h\big((1-\varepsilon)\rho_\tau + \varepsilon\bar\rho\big)\Big| = \Big|h'\big((1-\varepsilon)\rho_\tau + \varepsilon\bar\rho\big)\cdot(\bar\rho - \rho_\tau)\Big| \leq \big(\|\bar\rho\|_\infty + \rho_\tau\big)\max\Big\{|h'(\rho_\tau)|, \|\bar\rho\|_\infty\Big\} \in L^1(\m X).
\end{align*}
Now, by Theorem 2.27 in \citep{folland1999real} we know
\begin{align*}
    \frac{\partial}{\partial\varepsilon}\int_{\m X} h\big((1-\varepsilon)\rho_\tau + \varepsilon\bar\rho\big)\bigg|_{\varepsilon=0} = \int_{\m X} \frac{\partial}{\partial\varepsilon}h\big((1-\varepsilon)\rho_\tau + \varepsilon\bar\rho\big)\bigg|_{\varepsilon=0} = \int_{\m X} h'(\rho_\tau)\,\dd (\bar\rho-\rho_\tau).
\end{align*}
This indicates that $h'(\rho_\tau)$ is a subdifferential.

By Proposition 7.20 in~\cite{santambrogio2015optimal}, we know 
\begin{align*}
    U + h'(\rho_\tau) + \frac{\phi}{\tau} + \int_{\m X}W(\cdot, y)\,\dd\rho_\tau(y) + \int_{\m X}W(y, \cdot)\,\dd\rho_\tau(y)
\end{align*}
is a constant [$\rho_\tau$]-a.e.. Taking the gradient yields the result.
\end{proof}

\subsection{Proof of Lemma~\ref{lem: diff_potential_w2}}
\begin{lemma}\label{lem: diff_potential_w2}
Let $\rho, \rho'\in\ms P_2^r(\m X)$. Under Assumption B, for any $j\in[m]$ we have
\begin{align*}
    \bigg\|\nabla\int_{\m X_{-j}}V(x_j, x_{-j})\,\dd\rho_{-j} - \nabla\int_{\m X_{-j}}V(x_j, x_{-j})\,\dd\rho_{-j}'\bigg\|^2 \leq L^2\W_2^2(\rho_{-j}, \rho_{-j}').
\end{align*}
\end{lemma}
\begin{proof}
Just notice that
\begin{align*}
&\quad\,\bigg\|\nabla\int_{\m X_{-j}}V(x_j, x_{-j})\,\dd\rho_{-j} - \nabla\int_{\m X_{-j}}V(x_j, x_{-j})\,\dd\rho'_{-j}\bigg\|^2\\
&= \bigg\|\nabla\int_{\m X_{-j}}V(x_j, x_{-j})\,\dd\rho_{-j} - \nabla\int_{\m X_{-j}}V(x_j, x_{-j})\,\dd \big(T_{\rho_{-j}}^{\rho'_{-j}}\big)_\# \rho_{-j}\bigg\|^2\\
&= \bigg\|\int_{\m X_{-j}}\nabla_j V(x_j, x_{-j}) - \nabla_j V\big(x_j, T_{\rho_{-j}}^{\rho'_{-j}}(x_{-j})\big)\,\dd\rho_{-j}\bigg\|^2\\
&\stackrel{\ri}{\leq} \int_{\m X_{-j}}\big\|\nabla_j V(x_j, x_{-j}) - \nabla_j V\big(x_j, T_{\rho_{-j}}^{\rho'_{-j}}(x_{-j})\big)\big\|^2\,\dd\rho_{-j}\\
&\stackrel{\rii}{\leq} L^2\int_{\m X_{-j}}\big\|T_{\rho_{-j}}^{\rho'_{-j}}(x_{-j}) - x_{-j}\big\|^2\,\dd\rho_{-j}\\
&= L^2\W_2^2(\rho_{-j}, \rho'_{-j}).
\end{align*}
Here, (i) is by Cauchy--Schwarz inequality; (ii) is by the $L$-smoothness of $V$. 
\end{proof}

\subsection{Tensorization of Wasserstein Distance}
\begin{lemma}[tensorization]\label{lem: tensorization}
Assume $\rho=\bigotimes_{j=1}^m\rho_j$ and $\rho' = \bigotimes_{j=1}^m\rho'_j$, where $\rho_j, \rho_j \in\ms P_2^r(\m X_j)$ for all $j\in[m]$. Then
\begin{align*}
    \W_2^2(\rho, \rho') = \sum_{j=1}^m\W_2^2(\rho_j, \rho'_j)\quad\mbox{and}\quad T_{\mu}^\nu(x) = (T_{\rho_1}^{\rho'_1}(x_1), \dots, T_{\rho_m}^{\rho'_m}(x_m)).
\end{align*}
\end{lemma}
\begin{proof}
By definition, we have
\begin{align*}
    \W_2^2(\rho, \rho') &= \min\bigg\{\int_{\m X^2}\|x-y\|^2\,\dd\pi: \pi\in\Gamma(\rho, \rho')\bigg\}\\
    &= \min\bigg\{\sum_{j=1}^m\int_{\m X_j^2}\|x_j-y_j\|^2\,\dd\pi(x, y): \pi\in\Gamma(\rho, \rho')\bigg\}\\
    &\geq \sum_{j=1}^m\min\bigg\{\int_{\m X_j^2}\|x_j-y_j\|^2\,\dd\pi(x, y): \pi\in\Gamma(\rho, \rho')\bigg\}\\
    &= \sum_{j=1}^m \min\bigg\{\int_{\m X_j^2}\|x_j-y_j\|^2\,\dd\pi_j(x_j, y_j): \pi_j\in\Gamma(\rho_j, \rho'_j)\bigg\} = \sum_{j=1}^m\W_2^2(\rho_j, \rho'_j).
\end{align*}
The inequality holds while taking $\pi = (\id, T_\rho^{\rho'})_\#\rho$. This implies the desired result.
\end{proof}

\subsection{Proof of Proposition~\ref{prop: FA}}
For any $\tilde\rho$, let $\tilde T$ be the optimal map from $\rho^k$ to $\tilde\rho$. Note that
\begin{align*}
&\quad\, \m V(\rho^{k+1}) + \m H(\rho^{k+1}) + \m W(\rho^{k+1}) + \frac{1}{2\tau}\W_2^2(\rho^k, \rho^{k+1})\\
&\leq \m V(T^{k+1}_\#\rho^k) + \m H(T^{k+1}_\#\rho^k) + \m W(T^{k+1}_\#\rho^k) + \frac{1}{2\tau}\int \|T^{k+1}(x) - x\|^2\,\dd\rho^k\\
&\leq \m V(\tilde T_\#\rho^k) + \m H(\tilde T_\#\rho^k) + \m W(\tilde T_\#\rho^k) + \frac{1}{2\tau}\int \|\tilde T(x) - x\|^2\,\dd\rho^k\\
&= \m V(\tilde \rho) + \m H(\tilde \rho) + \m W(\tilde \rho) + \frac{1}{2\tau}\W_2^2(\rho^k, \tilde\rho).
\end{align*}
So, $\rho^{k+1} = T^{k+1}_\#\rho^k$ minimizes~\eqref{eqn: JKO}.

\subsection{Lemmas for Controlling Numerical Error}
\begin{lemma}\label{lem: seq_QG_inexact}
Let $\{x_k\}_{k=0}^\infty$ and $\{\xi_k\}_{k=1}^\infty$ be two non-negative sequences such that
\begin{align*}
x_k \leq A x_{k-1} + B \xi_k,\qquad\forall\, k\in\mb Z_+
\end{align*}
for some constants $0 < A < 1$ and $B > 0$. 

\vspace{0.5em}
\noindent (1) If $\xi_k \leq \kappa \xi^k$ for some constants $0 < \xi < 1$ and $\kappa > 0$, we have
\begin{align*}
x_k \leq A^kx_0 + \frac{B\kappa}{|A-\xi|}\cdot \max\{A, \xi\}^{k+1}.
\end{align*}

\vspace{0.5em}
\noindent (2) If $\xi_k \leq \xi k^{-\alpha}$ for some $\xi, \alpha \geq 0$, there is a constant $C(\alpha, A) > 0$ such that
\begin{align*}
x_k \leq  A^kx_0 + \frac{C(\alpha, A)B\xi}{k^\alpha}.
\end{align*}
\end{lemma}
\begin{proof}
It is easy to see that
\begin{align*}
x_k \leq A^k x_0 + B\sum_{l=1}^k A^{k-l}\xi_l,\qquad\forall\, k\in\mb Z_+.
\end{align*}
(1) When $\xi_k \leq \kappa\xi^k$, we have
\begin{align*}
x_k \leq A^k x_0 + B\sum_{l=1}^kA^{k-l}\kappa\xi^l = A^k x_0 + BA^k\kappa\sum_{l=1}^K\Big(\frac{\xi}{A}\Big)^l.
\end{align*}
Note that
\begin{align*}
\sum_{l=1}^k\Big(\frac{\xi}{A}\Big)^l 
\leq \begin{cases}
\frac{A}{A - \xi} & \xi < A\\
&\\
\frac{\xi^{k+1}}{A^k(\xi - A)} & \xi > A
\end{cases}.
\end{align*}
So, we have
\begin{align*}
BA^k\kappa\sum_{l=1}^k\Big(\frac{\xi}{A}\Big)^l
\leq \begin{cases}
\frac{AB\kappa}{A-\xi} A^k & \xi < A\\
&\\
\frac{B\xi\kappa}{\xi - A} \xi^k & \xi > A
\end{cases}
\,\,\leq \frac{B\kappa}{|A-\xi|}\max\{A, \xi\}^{k+1}.
\end{align*}
Therefore, we have
\begin{align*}
x_k \leq A^k x_0 + \frac{B\kappa}{|A-\xi|}\max\{A, \xi\}^{k+1} \leq \Big(\frac{x_0}{A} + \frac{B\kappa}{|A-\xi|}\Big)\max\{A, \xi\}^{k+1}.
\end{align*}
(2) When $\xi_k \leq \xi k^{-\alpha}$, we have
\begin{align*}
x_k \leq A^k x_0 + B\sum_{l=1}^k A^{k-l}\frac{\xi}{l^\alpha} = A^kx_0 + BA^k\xi\sum_{l=1}^k \frac{A^{-l}}{l^\alpha}.
\end{align*}
Similar to the proof in~\cite[Proof of Theorem 4,][]{yao2023minimizing}, there is a constant $C = C(\alpha, A) > 0$ such that
\begin{align*}
\sum_{l=1}^k \frac{A^{-l}}{l^\alpha} \leq \frac{CA^{-k}}{k^\alpha}.
\end{align*}
So, we have
\begin{align*}
x_k \leq A^kx_0 + BA^k\xi \frac{CA^{-k}}{k^\alpha} = A^kx_0 + \frac{C(\alpha, A)B\xi}{k^\alpha}.
\end{align*}
\end{proof}

\begin{lemma}\label{lem: seq_convex_inexact}
Let $\{x_k\}_{k=0}^\infty$ and $\{\xi_k\}_{k=1}^\infty$ be two non-negative sequences and $\{\xi_k\}$ are non-increasing. Assume that there are constants $A, B, C > 0$ such that
\begin{align*}
A {x_k}^2 + (1-B\xi_k)x_k \leq x_{k-1} + C\xi_k^2.
\end{align*}
If $\xi_k \leq \frac{\xi}{k^\alpha}$ for some $\alpha > 0$, then there exists a constant $K = K(\alpha, \xi, A, B, C) > 0$, such that
\begin{equation}\label{eqn: best_rate}
x_k \leq \frac{K}{k^{\min\{1, \alpha\}}}.
\end{equation}
\end{lemma}
\begin{proof}
Without loss of generality, we assume $\xi_k \leq \frac{1}{k^\alpha}$. Otherwise, we have
\begin{align*}
A {x_k}^2 + \big(1-\frac{B\xi}{k^\alpha}\big)x_k \leq x_{k-1} + \frac{C\xi^2}{k^{2\alpha}}, 
\end{align*}
and then take $B' = B\xi, C' = C\xi^2$, and $\xi_k' = 1/k^\alpha$.

We can take $K$ large enough such that inequality~\eqref{eqn: best_rate} is true for all $k \geq k_0$, where $k_0 \geq 1$ is a constant such that $Bk_0^{-\alpha} < 1$. We will use induction to prove~\eqref{eqn: best_rate}. Assume the inequality is true for $k-1 \geq k_0$. Now let us prove~\eqref{eqn: best_rate} holds for $k$. Then, we have
\begin{align*}
x_k \leq \frac{-(1-B\xi_k) + \sqrt{(1-B\xi_k)^2 + 4A(x_{k-1} + C\xi_k^2)}}{2A}.
\end{align*}
Since $1 - B\xi_k \geq 0$, we have
\begin{align*}
x_k \leq \frac{K}{k^{\min\{1, \alpha\}}}
&\Longleftarrow \frac{-(1-B\xi_k) + \sqrt{(1-B\xi_k)^2 + 4A(x_{k-1} + C\xi_k^2)}}{2A} \leq \frac{K}{k^{\min\{1, \alpha\}}}\\
&\Longleftrightarrow (1-B\xi_k)^2 + 4A(x_{k-1}+C\xi_k^2) \leq \Big(\frac{2AK}{k^{\min\{1, \alpha\}}} + (1-B\xi_k)\Big)^2\\
&\Longleftrightarrow x_{k-1} \leq \frac{AK^2}{k^{2\min\{1,\alpha\}}} + \frac{K(1-Bk^{-\alpha})}{k^{\min\{1, \alpha\}}} - \frac{C}{k^{2\alpha}}\\
&\Longleftrightarrow \frac{K}{(k-1)^{\min\{1, \alpha\}}} \leq \frac{AK^2}{k^{2\min\{1,\alpha\}}} + \frac{K(1-Bk^{-\alpha})}{k^{\min\{1, \alpha\}}} - \frac{C}{k^{2\alpha}}\\
&\Longleftrightarrow  \frac{K}{(k-1)^{\min\{1, \alpha\}}} - \frac{K}{k^{\min\{1, \alpha\}}} \leq \frac{AK^2}{k^{2\min\{1,\alpha\}}} - \frac{BK}{k^{\alpha + \min\{1, \alpha\}}} - \frac{C}{k^{2\alpha}}.
\end{align*}
\underline{Case 1.} When $\alpha > 1$, we only need to prove
\begin{align*}
\frac{K}{k-1} - \frac{K}{k} \leq \frac{AK^2}{k^2} - \frac{BK}{k^{1+\alpha}} - \frac{C}{k^{2\alpha}}
&\Longleftarrow \frac{K}{k(k-1)} \leq \frac{AK^2 - BK - C}{k^2},
\end{align*}
which only requires $AK^2 - (B+2)K - C \geq 0$ since $k\geq k_0 + 1 \geq 2$.

\noindent\underline{Case 2.} When $\alpha \leq 1$, we only need to show
\begin{align*}
\frac{K}{(k-1)^\alpha} - \frac{K}{k^\alpha} \leq \frac{AK^2 - BK - C}{k^{2\alpha}}.
\end{align*}
Note that
\begin{align*}
k^\alpha\Big[\frac{1}{(k-1)^\alpha} - \frac{1}{k^\alpha}\Big] = \Big(1 + \frac{1}{k-1}\Big)^\alpha - 1 \leq 2^{(\alpha-1)_+}\frac{\alpha}{k-1} \leq \frac{2^\alpha\alpha}{k}\qquad\forall\,k\geq 2,
\end{align*}
So, we only need to prove
\begin{align*}
\frac{2^\alpha\alpha K}{k} \leq \frac{AK^2 - BK - C}{k^\alpha}.
\end{align*}
Since $\alpha < 1$, we just need $AK^2 - (B+2^\alpha\alpha)K - C \geq 0$. By induction, we know such $K$ exists.
\end{proof}

\subsection{Numerical Methods in FA Approach}\label{app: FA_detail}
By changing variables, we can write the interaction and self-interaction functionals as
\begin{align*}
\m V(T_\#\rho^k) = \int V(x)\,\dd [T_\#\rho^k](x) = \int V(T(x))\,\dd\rho^k(x),
\end{align*}
and
\begin{align*}
\m W(T_\#\rho^k) = \int\!\!\int W(x, x')\,\dd[T_\#\rho^k](x)\dd[T_\#\rho^k](x') = \int\!\!\int W(T(x), T(x'))\,\dd\rho^k(x)\dd\rho^k(x').
\end{align*}
These expressions suggest the following sample approximations 
\begin{align*}
\m V(T_\#\rho^k) \approx \frac{1}{B}\sum_{b=1}^B V(T(X_b^k))
\quad\mx{and}\quad 
\m W(T_\#\rho^k)\approx \frac{1}{B^2}\sum_{b, b'=1}^B W\big(T(X_b^k), T(X_{b'}^k)\big),
\end{align*}
where $\{X_b^{k}: b\in[B]\}$ are $B$ samples from $\rho^k$. Regarding the internal energy term $\m H(\rho^k)$, we will consider two leading cases $\m H(\rho^k) = \int\rho_k\log\rho_k$ and $\m H(\rho^k) = \int(\rho^k)^{n}$ with some positive integer $n \geq 2$.

For $\m H(\rho^k) = \int\rho^k\log\rho^k$, we can apply the argument in~\citep{mokrov2021large} by noting that
\begin{align*}
\m H(T_\#\rho^k) = \m H(\rho^k) - \int \log\,\lvert\det \nabla T(x)\rvert\,\dd\rho^k,
\end{align*}
which suggests approximating the negative self-entropy functional by
\begin{align*}
\m H(T_\#\rho^k) \approx \m H(\rho^k) - \frac{1}{B} \sum_{b=1}^B \log\,\lvert\det\nabla T(X_b^k)\rvert
\end{align*}
Since the first term does not depend on $T$, the above discussions yield the approximation in Example~\ref{ex: negative_entropy}

For $\m H(\rho^k) = \int(\rho^k)^n$, an additional kernel density estimation (KDE) step is required. Since
\begin{align*}
\m H(T_\#\rho^k) = \int \big[\rho^k(x)\cdot\lvert\det\nabla T(x)\rvert^{-1}\big]^{n-1}\,\dd\rho^k,
\end{align*}
which suggests that we consider the approximation
\begin{align*}
\m H(T_\#\rho^k) \approx \frac{1}{B}\sum_{b=1}^B \big[\wht\rho_{\rm kde}^k(X_b^k)\cdot\lvert\det\nabla T(X_b^k)\rvert^{-1}\big]^{n-1},
\end{align*}
where $\wht\rho_{\rm kde}^k$ is the kernel density estimation of $\rho^k$ by $\{X_b^k: b\in[B]\}$. This leads to the approximation in Example \ref{ex: porous_medium}.

\section{Results for Multi-species Systems}
\subsection{Proof of Proposition \ref{lem: multi_species_system}}
Since $\rho^\ast = (\rho_1^\ast, \dots, \rho_m^\ast)$ is the minimum of~\eqref{eqn: obj_func}, by Lemma \ref{lem: ot_map}, we have
\begin{align*}
    C_j(\rho_{-j}^\ast) &= \int_{\m X_{-j}}V(x_j, x_{-j})\,\dd\rho_{-j}^\ast + h_j'(\rho_j^\ast) + \int_{\m X_j} W_j(\cdot, y)\,\dd\rho_j^\ast(y) + \int_{\m X_j} W_j(y, \cdot)\,\dd\rho_j^\ast(y)\\
    &= \Big(V_j - \sum_{i\neq j} K_{ij}\ast\rho_i^\ast\Big) + h_j'(\rho_j^\ast) - K_{jj}\ast\rho_j^\ast
\end{align*}
is a constant [$\rho_j^\ast$]-a.e. on $\m X_j$. If $\rho(t) = \rho^\ast$ at time $t=0$, we have
\begin{align}
\rho_j(t)\, \Big[\,\nabla V_j(t) - \sum_{i=1}^m (\nabla K_{ij})\ast\rho_i + \nabla h_j'(\rho_j)\,\Big] = 0.
\end{align}
This implies that $\rho^\ast$ is a stationary measure.

\subsection{Example in Section \ref{subsec:multi-species_systems}}\label{appendix: multi_species}
Recall that $h_j(\rho_j) = \rho_j\log\rho_j$,
\begin{align*}
    V(x_1, x_2, x_3) = \sum_{i=1}^n \frac{r_i}{2}\|x_i - m_i\|^2 - \sum_{1\leq i<j\leq 3}\frac{Q_iQ_j}{2}\arctan\|x_i-x_j\|^2
\end{align*}
and 
\begin{align*}
    W_j(x_j, x_j') = \frac{Q_j^2}{2}\arctan\|x_j - x_j'\|^2,\quad 1\leq j\leq 3.
\end{align*}
Let us now check all assumptions in Section \ref{sec: WPCG}.

\noindent\underline{Assumption A.} It is clear that $\rho_j^2$ satisfies the Assumption A.

\vspace{0.5em}
\noindent\underline{Assumption B.} Let $R(y) = y / (1+\|y\|^2)$ be a function on $\mb R^2$. For any $x = (x_1, x_2, x_3)$ and $x' = (x_1', x_2', x_3')$, we have
\begin{align*}
\big\|\nabla_1 V(x) - \nabla_1 V(x')\big\|
&\lesssim \big\|r_1(x_1 - x_1')\big\| + |Q_1Q_2|\cdot\Big\|\frac{x_1 - x_2}{1 + \|x_1 - x_2\|^2} - \frac{x_1' - x_2'}{1 + \|x_1' - x_2'\|^2}\Big\|\\
&\qquad\qquad\qquad\qquad + |Q_1Q_3|\cdot\Big\|\frac{x_1 - x_3}{1 + \|x_1 - x_3\|^2} - \frac{x_1' - x_3'}{1 + \|x_1' - x_3'\|^2}\Big\|\\
&= |r_1|\cdot\|x_1 - x_1'\| + |Q_1Q_2|\cdot \big\|R(x_1 - x_2) - R(x_1' - x_2')\big\|\\
&\qquad\qquad\qquad\qquad\qquad + |Q_1Q_3|\cdot\big\|R(x_1 - x_3) - R(x_1'-x_3')\big\|.
\end{align*}
Notice that the Jacobian matrix of $R$ at $y = (y_1, y_2)$ is
\begin{align*}
    JR(y) = \frac{1}{(1+y_1^2 + y_2^2)^2}
    \begin{pmatrix}
    1 + y_2^2 - y_1^2 & -2y_1y_2\\
    -2y_1y_2 & 1 + y_1^2 - y_2^2
    \end{pmatrix}
\end{align*}
with eigenvalues
\begin{align*}
    \lambda_1 = \frac{1}{1+y_1^2 + y_2^2}, \quad\mx{and}\quad \lambda_2 = \frac{1 - y_1^2 - y_2^2}{(1+y_1^2 + y_2^2)^2}.
\end{align*}
Clearly, we have $|\lambda_1| \leq 1$ and $|\lambda_2| \leq 1$. This implies $\|R(y) - R(y')\| \leq \|y - y'\|$. Therefore, we have
\begin{align*}
    \big\|\nabla_1 V(x) - \nabla_1 V(x')\big\|
    &\lesssim \|x_1 - x_1'\| + \|(x_1-x_2) - (x_1'-x_2')\| + \|(x_1-x_3) 
    - (x_1'-x_3')\|\\
    &\lesssim \|x - x'\|.
\end{align*}
This indicates that 
\begin{align*}
    \big\|\nabla V(x) - \nabla V(x')\big\| \lesssim \sum_{i=1}^3\big\|\nabla_i V(x) - \nabla_i V(x')\big\| \lesssim \|x-x'\|.
\end{align*}
So, $\nabla V$ is Lipschitz.

\vspace{0.5em}
\noindent\underline{Assumption C.} To verify this, we first prove that when $\alpha_j = \beta_j = Q_j^2$,
\begin{align*}
    \widetilde W_j(x_j, x_j') = \frac{\alpha_j}{2}\|x_j\|^2 + W_j(x_j, x_j') + \frac{\beta_j}{2}\|x_j'\|^2 = \frac{\alpha_j}{2}\|x_j\|^2 + \frac{Q_j^2}{2}\arctan\|x_j-x_j'\|^2 + \frac{\beta_j}{2}\|x_j'\|^2
\end{align*}
is a convex function on $\mb R^2\times\mb R^2$. To show this, notice that
\begin{align*}
    \frac{\partial W_j}{\partial x_j} = Q_j^2R(x_j - x_j')\quad\mx{and}\quad \frac{\partial W_j}{\partial x_j'} = Q_j^2R(x_j' - x_j).
\end{align*}
Therefore, the Hessian of $W_j$ is
\begin{align*}
\nabla^2 W_j = Q_j^2\begin{pmatrix}
    JR(x_j - x_j') & -JR(x_j - x_j')\\
    -JR(x_j - x_j') & JR(x_j - x_j')
\end{pmatrix}.
\end{align*}
Let
\begin{align*}
    \eta_{j1} = \frac{1}{1 + (x_{j1} - x_{j1}')^2 + (x_{j2} - x_{j2}')^2}\quad\mx{and}\quad\eta_{j2} = \frac{1 - (x_{j1} - x_{j1}')^2 - (x_{j2} - x_{j2}')^2}{(1 + (x_{j1} - x_{j1}')^2 + (x_{j2} - x_{j2}')^2)^2}
\end{align*}
be the eigenvalues of $JR(x_j-x_j')$. Assume $P\in O(2)$ be the orthogonal transition matrix such that
\begin{align*}
    JR(x_j - x_j') = P\begin{pmatrix}
        \eta_{j1} & 0\\
        0 & \eta_{j2}.
    \end{pmatrix} P^{\rm T}.
\end{align*}
Therefore, we have
\begin{align*}
\begin{pmatrix}
    P^{\rm T} & 0\\
    0 & P^{\rm T}
\end{pmatrix} 
\nabla^2\widetilde W_j
\begin{pmatrix}
    P & 0\\
    0 & P
\end{pmatrix}
&= \begin{pmatrix}
\alpha_j + Q_j^2\eta_{j1} & 0 & -\eta_{j1} & 0\\
0 & \alpha_j + Q_j^2\eta_{j2} & 0 & -\eta_{j2}\\
-\eta_{j1} & 0 & \beta_j + Q_j^2\eta_{j1} & 0\\
0 & -\eta_{j2} & 0 & \beta_j + Q_j^2\eta_{j2}
\end{pmatrix},
\end{align*}
which is p.s.d. when $\alpha_j = \beta_j \geq -Q_j^2\min\{0, \eta_{j1}, \eta_{j2}\}$. We can then take $\alpha_j = \beta_j = Q_j^2$ since $\eta_{j1}, \eta_{j2} > -1$.

Next, we will show that
\begin{align*}
    \widetilde V(x) &= V(x) - \sum_{i=1}^3 Q_i^2\|x_i\|^2\\
    &= \sum_{i=1}^n \Big(\frac{r_i}{2}\|x_i - m_i\|^2 - Q_i^2\|x_i\|^2\Big) - \sum_{1\leq i<j\leq 3}\frac{Q_iQ_j}{2}\arctan\|x_i-x_j\|^2
\end{align*}
is $\lambda$-convex for $\lambda = \min_{1\leq i\leq 3}\Big(r_i - 4Q_i^2 - |Q_i|\sum_{j\neq i}|Q_j|\Big)$. Let $H_{ij} = JR(x_i - x_j)$ for simplicity. It is easy to see that
\begin{align*}
    \frac{\partial^2\widetilde V}{\partial x_i^2} = (r_i-2Q_i^2)I_2 - Q_i\sum_{l\neq i}Q_lH_{il}\quad\mx{and}\quad
    \frac{\partial\widetilde V}{\partial x_i\partial x_j} = Q_iQ_jH_{ij}, \quad\forall\, 1\leq i\neq j\leq 3.
\end{align*}
Therefore, we have
\begin{align*}
&\quad\,
\begin{pmatrix}
    y_1^{\rm T} & y_2^{\rm T} & y_3^{\rm T}
\end{pmatrix}
\nabla^2\widetilde V(x)
\begin{pmatrix}
    y_1\\ y_2\\ y_3
\end{pmatrix}\\
&= \sum_{i=1}^3 (r_i-2Q_i^2)\|y_i\|^2 - \sum_{i=1}^3Q_i\sum_{j\neq i}Q_j y_i^{\rm T} H_{ij}y_i  + 2\sum_{1\leq i<j\leq 3}Q_iQ_j y_i^{\rm T}H_{ij}y_j\\
&\geq \sum_{i=1}^3 (r_i-2Q_i^2)\|y_i\|^2 - \sum_{i=1}^3\sum_{j\neq i} |Q_iQ_j|\matnorm{H_{ij}}\|y_i\|^2 - 2\sum_{1\leq i<j\leq 3}|Q_iQ_j|\matnorm{H_{ij}}\|y_i\|\|y_j\|\\
&\stackrel{\ri}{\geq} \sum_{i=1}^3\Big(r_i-2Q_i^2 - |Q_i|\sum_{j\neq i}|Q_j|\Big)\|y_i\|^2 - 2\sum_{1\leq i<j\leq 3}|Q_iQ_j|\|y_i\|\|y_j\|\\
&\stackrel{\rii}{\geq} \sum_{i=1}^3 \Big(r_i - 4Q_i^2 - |Q_i|\sum_{j\neq i}|Q_j|\Big)\|y_i\|^2\\
&\geq \min_{1\leq i\leq 3}\Big(r_i - 4Q_i^2 - |Q_i|\sum_{j\neq i}|Q_j|\Big)\|y\|^2
\end{align*}
Here, (i) is because $\matnorm{H_{ij}}\leq 1$; (ii) is due to AM-GM inequality. This implies $\widetilde V$ is $\min_{1\leq i\leq 3}\Big(r_i - 4Q_i^2 - |Q_i|\sum_{j\neq i}|Q_j|\Big)$-convex.

\section{Connections between Different Lipschitz Conditions}\label{app: connection}
First, let us show $0 \leq L/L_c \leq \sqrt{m-1}$. The lower bound is because $L$ and $L_c$ are non-negative. For the upper bound, notice that $L = \max_{j\in[m]}\sup_{x}\matnorm{\nabla_j\nabla_{-j} V(x)}$ and $L_c = \sup_{x}\matnorm{\nabla_j^2 V(x)}$. Also, when $\nabla_j^2 V$ is invertible, we have
\begin{align*}
    0 \preceq \begin{pmatrix}
        \nabla_{-j}^2 V & \nabla_j\nabla_{-j}V\\
        \nabla_{-j}\nabla_j V & \nabla_j^2 V
    \end{pmatrix}
    \sim\begin{pmatrix}
        \nabla_{j}^2 V & 0\\
        0 & \nabla_{-j}^2V - \nabla_j\nabla_{-j}V (\nabla_j^2V)^{-1}\nabla_{-j}\nabla_j V
    \end{pmatrix},
\end{align*}
where $A\sim B$ means $A$ and $B$ are congruent, i.e., there exists invertible matrix $P$ such that $A = P^TBP$. It is known that congruent transformation does not change the semi-positive definiteness of a matrix. This implies $0 \preceq \nabla_j^2V - \nabla_j\nabla_{-j}V (\nabla_j^2V)^{-1}\nabla_{-j}\nabla_j V$, i.e.,
\begin{align*}
    v^T \nabla_{-j}^2V v \geq v^T \nabla_j\nabla_{-j}V (\nabla_j^2V)^{-1}\nabla_{-j}\nabla_j V v,\quad \forall\, v\in\mb R^{d-d_j}.
\end{align*}
Therefore, we have
\begin{align*}
    \matnorm{\nabla_{-j}^2 V} &= \sup_{\|v\|=1}v^T \nabla_{-j}^2V v\\
    &\geq \sup_{\|v\|=1} v^T \nabla_j\nabla_{-j}V (\nabla_j^2V)^{-1}\nabla_{-j}\nabla_j V v\\
    &\geq \sup_{\|v\|=1} \matnorm{\nabla_j^2 V}^{-1}\|\nabla_{-j}\nabla_j Vv\|^2\\
    &= \matnorm{\nabla_j^2 V}^{-1}\matnorm{\nabla_{-j}\nabla_j V}^2.
\end{align*}
This implies
\begin{align*}
    L^2 = \sup_x \matnorm{\nabla_{-j}\nabla_j V(x)}^2 &\leq \sup_x \matnorm{\nabla_{-j}^2 V(x)}\matnorm{\nabla_j^2 V(x)}\\
    &\leq \sup_x \matnorm{\nabla_{-j}^2 V(x)} \sup_x\matnorm{\nabla_j^2 V(x)}\\
    &\stackrel{\ri}{\leq} (m-1)\max_{i\neq j}\sup_x\matnorm{\nabla_i^2 V(x)}\sup_x\matnorm{\nabla_j^2 V(x)}\\
    &\leq (m-1)\Big(\max_{j\in[m]}\sup_x\matnorm{\nabla_j^2V(x)}\Big)^2\\
    &= (m-1) L_c^2.
\end{align*}
Here, (i) is due to Section 3.2 in \citep{wright2015coordinate}. Thus, we have $L / L_c \leq \sqrt{m-1}$.
To see the bounds are tight, consider the examples $V(x_1, x_2) = x_1^2 + x_2^2$ and $V(x_1, \cdots, x_n) = (x_m + \cdots, x_m)^2$. In the first example, we have $L = 0$ and $L_c = 2$, indicating that $L/L_c = 0$. In the second example, we have $L = 2\sqrt{m-1}$ and $L_c = 2$, indicating that $L / L_c = \sqrt{m-1}$.

Next, we will prove that $0 \leq L / L_r \leq \sqrt{1-1/m}$. The lower bound is due to $L, L_r \geq 0$ and is tight due to the same reason as of $L / L_c$. To see the upper bound, note that $L_r = \max_{j\in[m]}\matnorm{\nabla_j\nabla V}$. Therefore, we have
\begin{align*}
L_r^2 &= \max_{j\in[m]}\sup_{\|v_1\|^2 + \|v_2\|^2=1} \|\nabla_{-j}\nabla_j Vv_1 + \nabla_j^2 Vv_2\|^2\\
&\stackrel{\ri}{=} \max_{j\in[m]} \sup_{\|v_1\|^2 + \|v_2\|^2 = 1} \Big(\|\nabla_{-j}\nabla_j Vv_1\| + \|\nabla_j^2 Vv_2\|\Big)^2\\
&\stackrel{\rii}{=} \max_{j\in[m]}\sup_{\|v_1\|^2 + \|v_2\|^2=1}\Big(\matnorm{\nabla_{-j}\nabla_j V}\|v_1\| + \matnorm{\nabla_j^2 V}\|v_2\|\Big)^2\\
&\stackrel{\riii}{\geq} \sup_{\|v_1\|^2+\|v_2\|^2 = 1}\Big(L\|v_1\| + \frac{L}{\sqrt{m-1}}\|v_2\|\Big)^2\\
&= \frac{mL^2}{m-1}.
\end{align*}
Here, (i) is because we can rotate $v_2$ such that $\nabla_{-j}\nabla_j Vv_1$ and $\nabla_j^2 Vv_2$ have the same direction; (ii) is by the definition of operator norm; (iii) is by $L_c \geq L/\sqrt{m-1}$. Thus, we have $L/L_r \leq \sqrt{1 - 1/m}$. The tightness of this bound can been seen by considering $V(x) = (x_1 + \cdots + x_n)^2$.

Finally, let us show $0\leq L/L_{\rm g} \leq 1$. This inequality is obviously true. The lower bound is tight by considering $V(x) = x_1^2 + \cdots + x_m^2$. For the upper bound, consider the example $V(x) = (x_1+x_2)^2 + x_3^2 + \cdots + x_m^2$. In this example, we have $L = L_{\rm g} = 2$. Therefore, $L / L_{\rm g}\leq 1$ is the tightest bound.

\bibliography{main}
\end{document}